\documentclass[11pt,twoside]{article}

\usepackage{amssymb}
\usepackage{amsmath}
\usepackage{amsthm}
\usepackage{color}
\usepackage{listings}
\usepackage{tikz}
\usepackage{mathrsfs}
\usepackage{indentfirst}
\usepackage{anysize}
\usepackage{txfonts}
\usepackage{booktabs}
\usepackage{tabularx}

\allowdisplaybreaks

\newtheorem{theorem}{Theorem}[section]
\newtheorem{lemma}[theorem]{Lemma}
\newtheorem{corollary}[theorem]{Corollary}
\newtheorem{proposition}[theorem]{Proposition}

\theoremstyle{definition}
\newtheorem{remark}[theorem]{Remark}
\newtheorem*{remark*}{Remark}
\newtheorem{definition}[theorem]{Definition}
\newtheorem*{definition*}{Definition}

\newtheorem{openproblem}[theorem]{Open Problem}
\renewcommand{\appendix}{\par
\setcounter{section}{0}%
\setcounter{subsection}{0}%
\setcounter{subsubsection}{0}%
\gdef\thesection{\Alph{section}}%
}

\lstdefinestyle{certificate}{
language=Python,
basicstyle=\ttfamily\scriptsize,
numbers=left,
stepnumber=5,
firstnumber=1,
numberfirstline=true,
numberstyle=\tiny,
numbersep=6pt,
xleftmargin=2em,
breaklines=true,
breakatwhitespace=false,
columns=fullflexible,
keepspaces=true,
showstringspaces=false,
tabsize=4
}

\numberwithin{equation}{section}

\begin{document}

\arraycolsep=1pt

\title{\bf\Large From Sharp Self-Improvement
to Exact Perimeter Formulas in Bourgain--Brezis--Mironescu Spaces
\footnotetext{\hspace{-0.35cm} 2020 {\it
Mathematics Subject Classification}.
Primary 42B35; Secondary 46E30, 26B30, 28A75.
\endgraf {\it Key words and phrases.}
Bourgain--Brezis--Mironescu space, sharp self-improvement,
Hardy-type space, perimeter, relative isoperimetric inequality.
\endgraf This project is partially supported by
the National Natural Science Foundation of China
(Nos. 12301112 and 12501120), the Natural
Science Foundation of Henan Province (No. 262300421227),
and China Postdoctoral Science Foundation (No. 2025M783094).}}
\author{Hongchao Jia and Xianjie Yan\footnote{Corresponding author, E-mail:
\texttt{xianjieyan@henu.edu.cn}/{\color{red}
\today}/Final version.}}
\date{}
\maketitle

\vspace{-0.7cm}

\begin{center}
\begin{minipage}{15cm}
{\bf Abstract}\quad
Let $Q_0:=(0,1)^n$.
We define $\mathcal{B}_q(Q_0)$ by replacing
the local $L^1$ mean oscillation
in the Bourgain--Brezis--Mironescu space
$\mathcal{B}(Q_0)$ with its $L^q$ counterpart,
which coincides with $\mathrm{BMO}(Q_0)$ when $n=1$.
For $n\geq2$, our results reveal different
endpoint behaviors for general functions,
characteristic functions and the exact perimeter formula.
(i) For general functions, $
\mathcal{B}_q(Q_0)=\mathcal{B}(Q_0)$ precisely when
$q\in(0,\frac{n}{n-1})$.
We also establish the corresponding sharp
self-improvement and duality results
for Hardy-type preduals.
(ii) For characteristic functions,
the endpoint is attained: the equivalence
$\|\mathbf{1}_E\|_{\mathcal{B}_q(Q_0)}
\sim\|\mathbf{1}_E\|_{\mathcal{B}(Q_0)}
\sim\min\{1,P(E,Q_0)\}$
holds uniformly over measurable $E\subset Q_0$ precisely when
$q\in(0,\frac{n}{n-1}]$, where
$P(E,Q_0)$ denotes relative perimeter.
(iii) The isotropic functional
$I_{\varepsilon,q}^{*}(E)$,
defined using arbitrarily oriented cubes,
converges to $\frac{1}{2}\min\{1,P(E,Q_0)\}$
as $\varepsilon\to0^+$ throughout $q\in(0,\frac{n}{n-1}]$.
In dimension two, the formula extends to $q\in(0,q_*]$
for some universal $q_*>\frac{19}{9}$,
whereas it fails in every dimension for every $q>q_0$,
where $q_0>2$ is defined variationally and numerical computation
suggests that $q_0\approx2.15391$.
The proofs of the three sharp
self-improvements use sparse median domination,
atomic block decompositions, and relative isoperimetry,
respectively.
For the perimeter formula, a uniform De Giorgi blow-up reduces
the problem to a geometric cube-section inequality;
geometric averaging and an explicit packing yield the planar extension
and the obstruction, respectively.
\end{minipage}
\end{center}

\section{Introduction\label{sec-new-introduction}}

Throughout this article, $n\in\mathbb{N}$ and
$Q_0:=(0,1)^n$. Recall that the classical space
$\mathrm{BMO\,}(Q_0)$ of functions with bounded mean oscillation,
introduced by John and Nirenberg \cite{JohnNirenberg1961},
is the set of all measurable functions $f$ on $Q_0$ such that
$$
\|f\|_{\mathrm{BMO}(Q_0)}
:=\sup_{\text{cube }P\subset Q_0}\fint_{P}|f(x)-f_P|\,dx<\infty,
$$
where, here and throughout, for every cube $P\subset\mathbb{R}^n$,
$
f_P:=\fint_Pf(x)\,dx:=\frac{1}{|P|}\int_Pf(x)\,dx.
$
It is well known that $(H^1(Q_0))^*=\mathrm{BMO\,}(Q_0)$;
see Fefferman and Stein \cite{FS72}. John and Nirenberg
\cite{JohnNirenberg1961} also
introduced the John--Nirenberg space,
{a BMO-type space whose norm involves
sums of local mean oscillations over disjoint cube families;}
see \cite{HuTaoYang2023,jtyyz1,JTYYZ-aamp,
TaoYangYuan2019,tyy2022,TaoYangYuan2023,
ZengChangTaoYang2023} for more studies.

{In 2015, Bourgain et al. \cite{BBM2015}
introduced another BMO-type space $\mathcal{B}(Q_0)$,
which consists of all integrable functions $f$
on $Q_0$ such that}
$$
\|f\|_{\mathcal{B}(Q_0)}:=\sup_{\epsilon\in(0,1)}
\sup_{\{Q_j\}_{j\in J}}
\epsilon^{n-1}\sum_{j\in J}
\fint_{Q_j}|f(x)-f_{Q_j}|\,dx<\infty.
$$
{Here the second supremum is taken over finite families
$\{Q_j\}_{j\in J}$ of open
axis-parallel cubes with pairwise disjoint interiors in $Q_0$,
with the same edge length $\epsilon$, and
$\#J\leq\frac{1}{\epsilon^{n-1}}$.
The essential difference between the John--Nirenberg
and $\mathcal{B}(Q_0)$ norms is that the latter
requires a common scale and bounds the number of cubes
in terms of that scale.}
When $n=1$, $\mathcal{B}(Q_0)=\mathrm{BMO\,}(Q_0)$.
Bourgain et al. \cite{BBM2015}
proved that when $n\geq2$, for every $p\in(1,\infty)$,
$$
\mathrm{VMO\,}(Q_0)+\mathrm{BV\,}(Q_0)
+W^{\frac{1}{p},p}(Q_0)\subset \mathcal{B}(Q_0)
$$
and
$\mathcal{B}(Q_0)\subset L^{\frac{n}{n-1},\infty}(Q_0)$
and the corresponding inclusion
into $L^{\frac{n}{n-1}}(Q_0)$ fails.
Subsequent work related this weak endpoint behavior
to Garsia--Rodemich spaces and
rearrangement-invariant embeddings \cite{Milman2019}.
Moreover, Bourgain et al. \cite{BBM2015}
proved that every integer-valued function belonging to the
VMO-type subspace $\mathcal{B}_0$ of $\mathcal{B}$ is necessarily constant.
D'Onofrio et al. \cite{dgp2020} obtained
an atomic predual and a two-star theorem, while
Liu et al. \cite{LiuYangYuan2021} developed nonlinear functional
calculus on $\mathcal{B}(Q_0)$ and
some other related spaces.
Very recently, B\'enyi et al.
\cite{BenyiHuZhou2026} studied the related spaces
$\mathcal{B}_{\mathrm C}^p$ on $Q_0\times(0,1)$
using upper Carleson tents.
{We refer to \cite{BessasGuarinoSchiattarella2026,
DiFrattaFiorenza2020,
FarroniGuarinoSchiattarella2020,
FuscoMoscarielloSbordone2016,
FuscoMoscarielloSbordone2018,
GuarinoLoBiancoSchiattarella2023,
jyy2025} for more studies
on this type of spaces.}

The bound $\#J\leq\varepsilon^{1-n}$ reflects codimension-one geometry.
Indeed, a hypersurface whose $(n-1)$-dimensional measure is of order one
typically meets about $\varepsilon^{1-n}$ cubes of edge length
$\varepsilon$. This makes the functional well suited to detecting
oscillation near jump sets and boundaries.
Ambrosio et al. \cite{abbf2016} studied the
corresponding isotropic small-scale functional
related to $\mathcal{B}(Q_0)$, keeping the same cardinality
bound but allowing arbitrary cube orientations. They proved that, for every
measurable $E\subset Q_0$, its limit is
$\frac{1}{2}\min\{1,P(E,Q_0)\}$.
This result confirms the codimension-one normalization of the cardinality
bound, but it does not by itself extend to other exponents. Peri's
inequalities for hyperplane sections of convex bodies provide closely
related geometric tools \cite{Peri1999,Peri2008}. In particular,
Corollary 2 and Theorem 3 in \cite{Peri2008} provide a central-section
estimate and a linear relative isoperimetric estimate.
A precise comparison, including the recovery of the $q=1$
cell inequality from earlier results and the distinction from the stronger
nonlinear estimate used here, is given in Remark
\ref{rem-new-cube-section-earlier}.
{The present problem also belongs
to the broader program of representing variation
by nonlocal and mean-oscillation functionals; see
\cite{AmbrosioComi2018,ArroyoBonicattoDelNin2024,
ArroyoBonicattoDelNin2026,BBM2001,BrezisNguyen2018,
Davila2002,FarroniFuscoGuarinoSchiattarella2020,
PonceSpector2017}.}

A natural question is whether the local $L^1$ mean oscillation in the
definition of $\mathcal{B}(Q_0)$ can be replaced by its local $L^q$
counterpart without changing the space.
The main difficulty is that the cubes in the definition of
$\mathcal{B}(Q_0)$ must have the same edge length and share a single
cardinality budget. Thus, unlike the well-known identity
$\mathrm{BMO}_q(Q_0)=\mathrm{BMO\,}(Q_0)$,
a cube-by-cube estimate is insufficient; one must control how the
oscillations are distributed across the entire family of cubes.
Answering this question is also important because it connects exponent
self-improvement with atomic duality and with perimeter formulas for
characteristic functions.

In this article, we answer the above question.
For $q\in(0,\infty)$, we introduce the space
$\mathcal{B}_q(Q_0)$ by
replacing the local $L^1$ mean oscillation
in the Bourgain--Brezis--Mironescu space $\mathcal{B}(Q_0)$
with its local $L^q$
counterpart, which coincides with $\mathrm{BMO}(Q_0)$ when $n=1$.
{For $n\geq2$,
our results reveal different
endpoint behaviors for general
functions, characteristic functions
and the exact perimeter formula.
The main results are as follows.}

\begin{enumerate}
\item[{\rm(i)}]
We first establish
the sharp self-improvement
\begin{align*}
\mathcal{B}(Q_0)=\mathcal{B}_q(Q_0)
\Longleftrightarrow
q\in\left(0,\frac{n}{n-1}\right);
\end{align*}
see Theorem \ref{the-new-self}.
We then establish the finite-exponent duality
$[\mathcal{H}_{q'}(Q_0)]^*=\mathcal{B}_q(Q_0)$
for $q\in(1,\infty)$ and the sharp self-improvement
\begin{align*}
\mathcal{H}_{s}(Q_0)=\mathcal{H}_\infty(Q_0)\Longleftrightarrow
s\in(n,\infty];
\end{align*}
see Theorems \ref{the-new-Hs-duality}
and \ref{the-new-Hs-self}.
As a consequence, we recover the endpoint predual
identification previously obtained in \cite[Theorem 1]{dgp2020}.

\item[{\rm(ii)}]
Moreover, we show that, for characteristic functions, the
endpoint $\frac{n}{n-1}$ is precisely recovered:
for any measurable $E\subset Q_0$,
\begin{align*}
\|\mathbf{1}_E\|_{\mathcal{B}(Q_0)}
\sim_{(n,q)}\|\mathbf{1}_E\|_{\mathcal{B}_q(Q_0)}
\sim_{(n,q)} \min\{1,P(E,Q_0)\}
\Longleftrightarrow
q\in\left(0,\frac{n}{n-1}\right],
\end{align*}
where $P(E,Q_0)$ denotes the
perimeter relative to $Q_0$;
see Theorem \ref{the-new-fixed-perimeter}.

\item[{\rm(iii)}]
Furthermore, to obtain an exact perimeter formula,
we introduce the isotropic scale functional
$I_{\varepsilon,q}^{*}(E)$ defined using arbitrarily oriented cubes
and show that, for any measurable
$E\subset Q_0$ and any $q\in(0,\frac{n}{n-1}]$,
\begin{align*}
\lim_{\varepsilon\to0^+}I_{\varepsilon,q}^{*}(E)
=\frac{1}{2}\min\{1,P(E,Q_0)\}.
\end{align*}
For $q=1$, this formula is
just \cite[Theorem 4.1]{abbf2016} but
we need new methods for $q\ne1$.
When $n=2$, the validity range extends to
$q\in(0,q_*]$ for some
$q_*>\frac{19}{9}$. The formula fails in every
dimension for every $q>q_0$, where $q_0>2$ is defined variationally
and numerical computation suggests that $q_0\approx2.15391$;
see Theorem \ref{the-new-exact}
and Proposition \ref{prop-new-packing-obstruction}.
\end{enumerate}

The proof of the self-improvement results in {\rm(i)} uses
sparse median domination for the $\mathcal{B}_q(Q_0)$ spaces and an
atomic block decomposition for the Hardy-type spaces, without relying
on a John--Nirenberg-type inequality. The characteristic-function result
in {\rm(ii)} is based on relative isoperimetry. For the exact perimeter
formula in {\rm(iii)}, cube families that nearly attain the $q=1$ limit
give the missing lower bound for $q<1$, whereas a uniform De Giorgi
blow-up and a geometric cube-section inequality give the missing upper
bound for $q>1$. Moreover, a geometric averaging argument yields the
extension in dimension two, while an explicit two-row packing of rotated
cubes gives a dimension-free upper bound for the optimal validity
exponent.

The remainder of this article is organized as follows.
Unless otherwise indicated, the results on sharp ranges, endpoint duality,
perimeter, and obstructions summarized here concern $n\geq2$.

In Section \ref{sec-new-self}, we first introduce the
Bourgain--Brezis--Mironescu space $\mathcal{B}_q(Q_0)$.
We then determine the sharp
range $q\in(0,\frac{n}{n-1})$
in which $\mathcal{B}_q(Q_0)$ agrees with $\mathcal{B}(Q_0)$;
see Theorem \ref{the-new-self}. To this end, we combine local
median sparse domination with a new sparse estimate that compares different
exponents under the common cardinality constraint; see Proposition
\ref{prop-new-sparse-cross-exponent}.
It is worth noting that this argument does not use a
John--Nirenberg-type inequality and therefore differs from the proof of
$\mathrm{BMO}_q(Q_0)=\mathrm{BMO\,}(Q_0)$.
Moreover, sharpness follows from the example constructed in
\cite[Proposition 3]{BBM2015}.
Finally, we compare
$\mathcal{B}_q(Q_0)$ with Lebesgue spaces. We also prove
strict inclusions beyond the critical exponent; see Proposition
\ref{prop-new-Lq-position}.

Section \ref{sec-new-predual} is divided into three subsections
to study the sharp self-improvement of
preduals of $\mathcal{B}_q(Q_0)$.
In Subsection 3.1, we introduce
$\mathcal{H}_s(Q_0)$ and prove the duality
$[\mathcal{H}_{q'}(Q_0)]^*=\mathcal{B}_q(Q_0)$ for $q\in(1,\infty)$;
see Theorem \ref{the-new-Hs-duality}.
Subsection 3.2 determines
the sharp range $\mathcal{H}_s(Q_0)=\mathcal{H}_\infty(Q_0)$ if and only if
$s\in(n,\infty]$; see Theorem \ref{the-new-Hs-self}. Since standard
atomic decompositions do not preserve the common-scale block structure,
we use a quantitative block decomposition adapted to it; see Lemma
\ref{lem-new-block-decomposition}.
Combining the two self-improvement results with finite-exponent duality,
we obtain the endpoint duality
$[\mathcal{H}_\infty(Q_0)]^*=\mathcal{B}(Q_0)$; see Corollary
\ref{cor-new-Hinfty-endpoint-duality}. Finally,
Subsection 3.3 uses the same block decomposition
to give a new proof
of $\mathcal{B}_q(Q_0)=\mathcal{B}(Q_0)$ for
$q\in(1,\frac{n}{n-1})$ without sparse domination, a
John--Nirenberg-type inequality, or duality; see Proposition
\ref{prop-new-block-alternative}.

In Section \ref{sec-new-fixed-perimeter},
we show that the endpoint $\frac{n}{n-1}$ excluded
for general functions for $\mathcal{B}_q(Q_0)=\mathcal{B}(Q_0)$
is recovered for characteristic functions: for
any measurable $E\subset Q_0$,
$\|\mathbf{1}_E\|_{\mathcal{B}_q(Q_0)}$ and
$\|\mathbf{1}_E\|_{\mathcal{B}(Q_0)}$
are uniformly comparable to the truncated
relative perimeter $\min\{1,P(E,Q_0)\}$ if and only if
$q\in(0,\frac{n}{n-1}]$; see Theorem
\ref{the-new-fixed-perimeter}. This endpoint is reached because, at
$q=\frac{n}{n-1}$, the power of the smaller volume fraction in the local
oscillation matches the exponent in the relative isoperimetric inequality,
yielding linear perimeter control.

The comparison in Section \ref{sec-new-fixed-perimeter} shows that
the quantities
$\|\mathbf{1}_E\|_{\mathcal{B}_q(Q_0)}$, including the case $q=1$, where
$\mathcal{B}_1(Q_0)=\mathcal{B}(Q_0)$, detect the truncated perimeter only
up to constants. In Section \ref{sec-new-exact}, we ask whether the
corresponding small-scale isotropic functional recovers it with the exact
constant and prove that it does for arbitrarily oriented cubes; see Theorem
\ref{the-new-exact}. For the lower bound, H\"older's inequality transfers
the known $q=1$ formula to $q\geq1$; for $q<1$, cube families nearly
attaining this limit have the same limit for their $q$-oscillation sums;
see Proposition \ref{prop-new-universal-lower}. For the upper bound, we use
a uniform De Giorgi blow-up to reduce the problem to a sharp half-space
inequality, retaining the exact factor $\frac{1}{2}$; see Proposition
\ref{prop-new-localization}. In dimension two, averaging over common shifts
of the entire cube family yields the missing upper bound
up to some universal exponent $q_*>\frac{19}{9}$; see Lemma
\ref{lem-new-aggregate-planar-localization}.

In Section \ref{sec-new-open-problems}, we
build a two-row packing of
rotated cubes along a flat interface
and then take products to extend the planar
construction to higher dimensions. This fixed set violates the
exact perimeter formula whenever $q>q_0>2$; see Proposition
\ref{prop-new-packing-obstruction}. Together with the positive
results, this obstruction provides an upper bound for the unknown
optimal exponent. Finally, we state the remaining range and endpoint
questions; see Open Problem \ref{prob-new-optimal-exact-range}.

In Appendix \ref{app-new-planar-cell-certificate}, we provide the computer
program used in the proof of Lemma \ref{lem-new-averaged-planar-cell}.

Finally, we use the following conventions on notation.
Let $\mathbb{N}:=\{1,2,\ldots\}$ and
$\mathbb{Z}_+:=\mathbb{N}\cup\{0\}$. For every $t\in\mathbb{R}$,
$\lfloor t\rfloor$ denotes the greatest integer not greater than $t$.
For every $A\subset\mathbb{R}^n$, we write $A^\circ$ for its interior.
For every cube $Q$, we use $\ell(Q)$
to denote its edge length. Unless starred notation
is used or otherwise stated, its edges are coordinate-parallel. For every
open $U$, the notation $Q\subset U$ means $Q^\circ\subset U$. Unless
otherwise stated, disjointness of cubes refers to their interiors.
We write $\#J$ for the cardinality of a set $J$, and
$\mathbf{1}_E$ for the characteristic function of a set $E$. We write
$\mathbb{K}\in\{\mathbb{R},\mathbb{C}\}$ for the underlying scalar
field and identify $\mathbb{K}$ with the subspace of constant functions
on $Q_0$.
When $\mathbb K=\mathbb C$, Banach-space duals consist of
complex-linear functionals, and all duality pairings used below are
bilinear, with no complex conjugation inserted. Thus, for example,
$\langle f,g\rangle=\int_{Q_0}f(x)g(x)\,dx$ is linear in $g$ when $f$ is
fixed.
We also use $C_{(\alpha,\beta,\ldots)}$ to denote a
positive constant depending only on the indicated parameters;
its value may vary from line to line.
The symbol $f\lesssim g$ means that $f\le Cg$.
If $f\lesssim g$ and $g\lesssim f$,
we then write $f\sim g$. If $f\le Cg$ and $g=h$
or $g\le h$, we then write $f\lesssim g=h$ or $f\lesssim g\le h$.
For every $p\in[1,\infty]$,
we denote by $p'$ its \emph{conjugate index},
that is, $\frac{1}{p}+\frac{1}{p'}=1$.
Finally, when we prove a theorem
(and the like), in its proof we always use the same
symbols as in the statement itself of that theorem
(and the like).

\section{The Sharp Self-Improvement Range
for $\mathcal{B}_q(Q_0)$\label{sec-new-self}}

In this section, we first introduce
the space $\mathcal{B}_q(Q_0)$. Then we prove the
sharp self-improvement range through a cross-exponent sparse estimate.
For every $q\in(0,\infty)$ and cube $Q\subset\mathbb{R}^n$, we define the
\emph{local $L^q$ mean oscillation} of $f\in L^1(Q)$ on $Q$ by
\begin{align*}
\operatorname{osc}_q(f;Q)
:=\left[
\fint_Q|f(x)-f_Q|^q\,dx
\right]^{\frac{1}{q}},
\end{align*}
allowing the value $+\infty$. In particular,
$\operatorname{osc}_q(f;Q)<\infty$ whenever $q\in(0,1]$.
We next define the admissible family of cubes.

\begin{definition}\label{def-F}
Let $\varepsilon\in(0,1]$. A collection
$\{Q_j\}_{j\in J}$ of axis-parallel cubes with pairwise disjoint
interiors contained in $Q_0$ is called an
\emph{admissible family at scale $\varepsilon$} if
$\ell(Q_j)=\varepsilon$
for every $j\in J$
and $\#J\leq\varepsilon^{1-n}.$
$\mathscr{F}_{\varepsilon}(Q_0)$ is
defined to be the collection of
all such families, including the empty family.
\end{definition}

Then we present the concept of
the Bourgain--Brezis--Mironescu space $\mathcal{B}_q(Q_0)$.

\begin{definition}\label{def-Bq}
Let $q\in(0,\infty)$. The \emph{Bourgain--Brezis--Mironescu space}
$\mathcal{B}_q(Q_0)$ is the quotient, modulo addition of constant
functions, of the set of all $f\in L^1(Q_0)$ for which
\begin{align}\label{eq-new-Bq}
\|f\|_{\mathcal{B}_q(Q_0)}
:=\sup_{\varepsilon\in(0,1]}
\sup_{\{Q_j\}_{j\in J}\in\mathscr{F}_{\varepsilon}(Q_0)}
\varepsilon^{n-1}\sum_{j\in J}
\operatorname{osc}_q(f;Q_j)
\end{align}
is finite.
\end{definition}

\begin{remark}\label{rem-Bq}
\begin{itemize}
\item[{\rm(i)}]
The displayed quantity is invariant under addition of constants.
It induces a norm on $\mathcal{B}_q(Q_0)$ when $q\geq1$ and a
quasi-norm when $0<q<1$. Throughout the article, the notation
$\|\cdot\|_{\mathcal{B}_q(Q_0)}$ is used for both cases.

\item[{\rm(ii)}]
It is easy to show that
the original Bourgain--Brezis--Mironescu space
$\mathcal{B}(Q_0)=\mathcal{B}_1(Q_0)$.

\item[{\rm(iii)}]
The outer supremum in \eqref{eq-new-Bq} may equivalently be taken over
$\varepsilon\in(0,1)$. Indeed, approximating $Q_0$ from within by
concentric cubes and applying Fatou's lemma shows that the scale-one test
is controlled by the smaller-scale tests.
\end{itemize}
\end{remark}

When $n=1$, $\mathcal{B}(Q_0)=\mathrm{BMO}(Q_0)$, so the corresponding
self-improvement follows from the John--Nirenberg theorem. Henceforth, unless
otherwise stated, we assume $n\geq2$ whenever a result depends on the
dimension; dimension-free statements remain valid also for $n=1$.

We now state the sharp self-improvement theorem.

\begin{theorem}\label{the-new-self}
Let $n\in\mathbb{N}\cap[2,\infty)$ and $q\in(0,\infty)$. Then
$$
\mathcal{B}(Q_0)=\mathcal{B}_q(Q_0)
\ \Longleftrightarrow\
q\in\left(0,\frac{n}{n-1}\right).
$$
Moreover, for every $q\in(0,\frac{n}{n-1})$
and $f\in L^1(Q_0)$,
\begin{align}\label{eq-new-self-norms}
\|f\|_{\mathcal{B}(Q_0)}
\sim\|f\|_{\mathcal{B}_q(Q_0)},
\end{align}
where the implicit positive constant depends
only on $n$ and $q$.
\end{theorem}

To show Theorem \ref{the-new-self},
we first establish the local sparse domination used throughout the
positive range. For every cube
$Q\subset\mathbb{R}^n$, let $\mathcal{D}(Q)$ denote the family of
dyadic subcubes generated from $Q$ by repeatedly bisecting every side.
Whenever unions or characteristic functions are used,
we take the usual
compatible half-open representatives; this does not affect any
Lebesgue integral or average. For every
$R\in\mathcal{D}(Q)\setminus\{Q\}$, let $\widehat R$ denote its unique
dyadic parent. If $\mathcal{A}\subset\mathcal{D}(Q)$, a cube
$R\in\mathcal{A}$ is called \emph{maximal in $\mathcal{A}$} if
no $S\in\mathcal{A}$ satisfies $R\subsetneq S$.
The notion of sparseness and the recursive stopping-time mechanism
behind it are standard; see, for example,
\cite[Section 2.1]{Duarte2024} for an expository account. The precise
external result used in Lemma \ref{lem-new-subunit-sparse} is
Hyt\"onen's streamlined form
\cite[Theorem 2.3 and Section 3]{Hytonen2014} of Lerner's local
oscillation formula; the original local-oscillation result is
\cite{Lerner2010}. The short proof below records explicitly the only
additional step needed here: converting Hyt\"onen's
rearrangement-based coefficients into the arithmetic-mean
$L^p$ oscillations used in the common-scale argument.
We call $\mathcal{S}\subset\mathcal{D}(Q)$ a
\emph{$\frac{1}{2}$-sparse} if
for every $R\in\mathcal{S},$
there exists a measurable set
$G_R\subset R$ such that
$|G_R|\geq\frac{1}{2}|R|$ and the sets
$\{G_R\}_{R\in\mathcal{S}}$ are pairwise disjoint.
A \emph{median} of a real-valued measurable function
$u$ on a cube $R$ is a
number $m_u(R)\in\mathbb{R}$ such that
\begin{align*}
|\{x\in R:\ u(x)>m_u(R)\}|
\leq\frac{1}{2}|R|
\quad\hbox{and}\quad
|\{x\in R:\ u(x)<m_u(R)\}|
\leq\frac{1}{2}|R|.
\end{align*}

We now state the required sparse domination lemma.

\begin{lemma}\label{lem-new-subunit-sparse}
Let $p\in(0,1]$, $Q\subset\mathbb{R}^n$ be a cube, and
$u\in L^1(Q)$ be real-valued. Then there exist a median $m_u(Q)$ of $u$
on $Q$ and a $\frac{1}{2}$-sparse collection
$\mathcal{S}\subset\mathcal{D}(Q)$ with $Q\in\mathcal{S}$ such that
\begin{align}\label{eq-new-subunit-sparse-volume}
\sum_{R\in\mathcal{S}}|R|\leq2|Q|
\end{align}
and
\begin{align}\label{eq-new-subunit-sparse-oscillation}
|u-m_u(Q)|\mathbf{1}_Q
\leq C_{(n,p)}
\sum_{R\in\mathcal{S}}\operatorname{osc}_p(u;R)\mathbf{1}_R
\end{align}
almost everywhere on $Q$.
\end{lemma}

\begin{proof}
Let $\lambda:=2^{-n-2}$. For a measurable function $g$
supported on a set of finite measure, let $g^*$ denote its
\emph{non-increasing rearrangement},
which is defined by setting, for every $t\in(0,\infty)$,
\begin{align*}
g^*(t):=\inf\{\alpha\in[0,\infty):
|\{x:\ |g(x)|>\alpha\}|\leq t\}.
\end{align*}
For a cube $R$, let
\begin{align*}
\omega_\lambda(u;R)
:=\inf_{c\in\mathbb{R}}
\left[(u-c)\mathbf{1}_R\right]^*(\lambda|R|).
\end{align*}
The local median oscillation formula
\cite[Theorem 2.3]{Hytonen2014} provides
a median $m_u(Q)$ and a
$\frac12$-sparse collection
$\mathcal{S}\subset\mathcal{D}(Q)$ such that
\begin{align}\label{eq-new-subunit-hytonen}
|u-m_u(Q)|\mathbf{1}_Q
\leq2\sum_{R\in\mathcal{S}}
\omega_\lambda(u;R)\mathbf{1}_R
\end{align}
almost everywhere on $Q$.
The construction in
\cite[Section 3]{Hytonen2014} starts with $Q$ as its
zeroth-generation cube, so $\mathcal{S}$ may be chosen with
$Q\in\mathcal{S}$.

Fix $R\in\mathcal{S}$. $p\in(0,1]$ and $u\in L^1(Q)$ imply
$\operatorname{osc}_p(u;R)<\infty$. Using this and Chebyshev's inequality,
we find that
\begin{align}\label{2.5x}
\left|\{x\in R:\ |u(x)-u_R|>
\lambda^{-\frac1p}\operatorname{osc}_p(u;R)\}\right|
\leq\lambda|R|.
\end{align}
Therefore, by the definitions of
$\omega_{\lambda}(u;R)$ and the non-increasing rearrangement,
together with \eqref{2.5x}, we conclude that
\begin{align}\label{eq-new-subunit-rearrangement-comparison}
\omega_\lambda(u;R)
\leq\left[(u-u_R)\mathbf{1}_R\right]^*(\lambda|R|)
\leq\lambda^{-\frac1p}\operatorname{osc}_p(u;R).
\end{align}
Combining \eqref{eq-new-subunit-hytonen} and
\eqref{eq-new-subunit-rearrangement-comparison}, we obtain
\eqref{eq-new-subunit-sparse-oscillation} with
$
C_{(n,p)}:=2\lambda^{-\frac1p}
=2^{1+\frac{n+2}{p}}.
$
Finally, the $\frac12$-sparseness established in
\cite[Theorem 2.3 and (3.4)]{Hytonen2014} provides pairwise disjoint
witness sets $\{G_R\}_{R\in\mathcal{S}}$, and hence
\begin{align*}
\sum_{R\in\mathcal{S}}|R|
\leq2\sum_{R\in\mathcal{S}}|G_R|
\leq2|Q|,
\end{align*}
which proves \eqref{eq-new-subunit-sparse-volume}.
This completes the proof of Lemma
\ref{lem-new-subunit-sparse}.
\end{proof}

The following standard sparse-sequence estimate is an adjoint form of the
dyadic Carleson embedding theorem; see, for instance,
\cite[Lemma 2.11]{HanninenLoristSinko2025}. We include the short proof
for completeness.

\begin{lemma}\label{lem-new-positive-sparse}
Let $r\in[1,\infty)$, $Q\subset\mathbb{R}^n$ be a cube,
$\mathcal{S}\subset\mathcal{D}(Q)$ be $\frac12$-sparse, and
$\{c_R\}_{R\in\mathcal{S}}\subset[0,\infty)$. Then
there exists a positive constant $C_{(r)}$,
depending only on $r$, such that
\begin{align}\label{eq-new-positive-sparse}
\left\|
\sum_{R\in\mathcal{S}}c_R\mathbf{1}_R
\right\|_{L^r(Q)}
\leq C_{(r)}
\left[
\sum_{R\in\mathcal{S}}|R|c_R^r
\right]^{\frac1r}.
\end{align}
\end{lemma}

\begin{proof}
When $r=1$, the assertion follows directly from Tonelli's theorem.
Assume that $r\in(1,\infty)$, and let
$\{G_R\}_{R\in\mathcal{S}}$ be sparse witness sets. Denote by
$M_Q^{\mathcal{D}}$ the dyadic maximal operator on $Q$. If
$\sum_{R\in\mathcal{S}}|R|c_R^r=+\infty$,
the assertion is immediate, so assume that this sum
is finite. Let $g\in L^{r'}(Q)$ be
nonnegative. For every $R\in\mathcal{S}$, sparseness implies
$G_R\subset R$ and $|R|\leq2|G_R|$.
Moreover, by the definition of $M_Q^{\mathcal{D}}g$, for every
$x\in G_R$,
$\fint_R g(y)\,dy
\leq(M_Q^{\mathcal{D}}g)(x).$
From these and the non-negativity of $g$, we infer that
\begin{align*}
|R|\fint_R g(y)\,dy
\leq2|G_R|\fint_R g(y)\,dy
\leq2\int_{G_R}\left(M_Q^{\mathcal{D}}g\right)(x)\,dx.
\end{align*}
Since $\mathcal{S}\subset\mathcal{D}(Q)$ is countable and every
summand below is nonnegative, Tonelli's theorem applies. Using the
preceding estimate, the pairwise disjointness of the sets $G_R$, and
H\"older's inequality, we obtain
\begin{align*}
\int_Q\left[
\sum_{R\in\mathcal{S}}c_R\mathbf{1}_R(x)
\right]g(x)\,dx
&=\sum_{R\in\mathcal{S}}c_R|R|\fint_R g(y)\,dy
\leq2\sum_{R\in\mathcal{S}}
c_R\int_{G_R}\left(M_Q^{\mathcal{D}}g\right)(x)\,dx\\
&\leq2
\left[
\sum_{R\in\mathcal{S}}|G_R|c_R^r
\right]^{\frac1r}
\left\|M_Q^{\mathcal{D}}g\right\|_{L^{r'}(Q)}
\leq2
\left[
\sum_{R\in\mathcal{S}}|R|c_R^r
\right]^{\frac1r}
\left\|M_Q^{\mathcal{D}}g\right\|_{L^{r'}(Q)}\\
&\leq C_{(r)}
\left[
\sum_{R\in\mathcal{S}}|R|c_R^r
\right]^{\frac1r}
\|g\|_{L^{r'}(Q)}.
\end{align*}
The last inequality follows from the $L^{r'}(Q)$ boundedness of the
dyadic maximal operator. Hence, the nonnegative form of $L^r(Q)$ duality
implies
\begin{align*}
\left\|
\sum_{R\in\mathcal{S}}c_R\mathbf{1}_R
\right\|_{L^r(Q)}
&=\sup_{g\geq0,\
\|g\|_{L^{r'}(Q)}\leq1}
\int_Q\left[
\sum_{R\in\mathcal{S}}c_R\mathbf{1}_R(x)
\right]g(x)\,dx
\leq C_{(r)}
\left[
\sum_{R\in\mathcal{S}}|R|c_R^r
\right]^{\frac1r},
\end{align*}
which is \eqref{eq-new-positive-sparse}.
This completes the proof of Lemma
\ref{lem-new-positive-sparse}.
\end{proof}

The following cross-exponent estimate is the main sparse
self-improvement principle of this section.

\begin{proposition}\label{prop-new-sparse-cross-exponent}
Let $n\in\mathbb{N}\cap[2,\infty)$,
$p\in(0,1]$, and
$r\in[1,\frac{n}{n-1})$. Then there exists a positive constant
$C_{(n,p,r)}$ such that, for every
$f\in L^1(Q_0)$,
\begin{align}\label{eq-new-sparse-cross-exponent}
\|f\|_{\mathcal{B}_r(Q_0)}
\leq C_{(n,p,r)}\|f\|_{\mathcal{B}_p(Q_0)}.
\end{align}
\end{proposition}

\begin{proof}
Assume first that $f$ is real-valued and let
$
M:=\|f\|_{\mathcal{B}_p(Q_0)}.
$
We may assume that $M<\infty$. Fix $\varepsilon\in(0,1]$ and
$\{Q_j\}_{j\in J}\in\mathscr{F}_{\varepsilon}(Q_0)$. Apply Lemma
\ref{lem-new-subunit-sparse} to each $Q_j$, and denote the resulting
collection by $\mathcal{S}_j\subset\mathcal{D}(Q_j)$ and the
corresponding median by $m_{Q_j}$.
Using H\"older's inequality and
Lemmas \ref{lem-new-subunit-sparse} and
\ref{lem-new-positive-sparse}, we have
\begin{align}\label{eq-new-local-cross-exponent}
\|f-f_{Q_j}\|_{L^r(Q_j)}
&\leq2\|f-m_{Q_j}\|_{L^r(Q_j)}
\leq2C_{(n,p)}
\left\|
\sum_{R\in\mathcal{S}_j}
\operatorname{osc}_p(f;R)\mathbf{1}_R
\right\|_{L^r(Q_j)}\nonumber\\
&\leq C_{(n,p,r)}
\left[
\sum_{R\in\mathcal{S}_j}|R|
\operatorname{osc}_p(f;R)^r
\right]^{\frac1r}.
\end{align}
For every $k\in\mathbb{Z}_+$, let
\begin{align*}
\rho_k:=2^{-k}\varepsilon,
\quad
\mathcal{T}_k
:=\left\{
R\in\bigcup_{j\in J}\mathcal{S}_j:\ \ell(R)=\rho_k
\right\},
\quad\text{and }
K_k:=\left\lfloor\rho_k^{1-n}\right\rfloor.
\end{align*}
The cubes in $\mathcal{T}_k$ have pairwise disjoint interiors. Hence,
for every $\mathcal{A}\subset\mathcal{T}_k$ satisfying
$\#\mathcal{A}\leq K_k$, we have
$\mathcal{A}\in\mathscr{F}_{\rho_k}(Q_0)$ and therefore
\begin{align}\label{eq-new-level-coefficient-budget}
\sum_{R\in\mathcal{A}}\operatorname{osc}_p(f;R)
\leq\rho_k^{1-n}M.
\end{align}
For every $k\in\mathbb{Z}_+$, let
$$
\mathcal{L}_k
:=\{R\in\mathcal{T}_k:\ \operatorname{osc}_p(f;R)\leq2M\}
$$
and
$$
\mathcal{U}_k
:=\{R\in\mathcal{T}_k:\ \operatorname{osc}_p(f;R)>2M\}.
$$
Also let
$\mathcal{L}:=\bigcup_{k=0}^{\infty}\mathcal{L}_k$ and
$\mathcal{U}:=\bigcup_{k=0}^{\infty}\mathcal{U}_k$.
Since $\rho_k^{1-n}\geq1$, we have
$
K_k\geq\frac12\rho_k^{1-n}.
$
If $\#\mathcal{U}_k\geq K_k$, choosing $K_k$ cubes from
$\mathcal{U}_k$ would contradict
\eqref{eq-new-level-coefficient-budget}. Consequently,
\begin{align}\label{eq-new-high-coefficient-budget}
\#\mathcal{U}_k<K_k
\quad\hbox{and}\quad
\sum_{R\in\mathcal{U}_k}\operatorname{osc}_p(f;R)
\leq\rho_k^{1-n}M.
\end{align}
By the sparse packing estimate
\eqref{eq-new-subunit-sparse-volume} and the cardinality bound on $J$,
we find that
\begin{align}\label{eq-new-low-sparse-energy}
\sum_{j\in J}
\left[
\sum_{R\in\mathcal{S}_j\cap\mathcal{L}}
|R|\operatorname{osc}_p(f;R)^r
\right]^{\frac1r}
&\leq
2M\sum_{j\in J}
\left[
\sum_{R\in\mathcal{S}_j}|R|
\right]^{\frac1r}\nonumber\\
&\leq
C_{(r)}M(\#J)\varepsilon^{\frac nr}
\leq C_{(r)}M\varepsilon^{1-n+\frac nr}.
\end{align}
Since $r<\frac{n}{n-1}$, we have
$1-n+\frac nr>0$. Using Minkowski's inequality, the
bound of the $\ell^r$-norm by the $\ell^1$ norm,
\eqref{eq-new-high-coefficient-budget}, and the geometric-series
identity, we find that
\begin{align}\label{eq-new-high-sparse-energy}
&\sum_{j\in J}
\left[
\sum_{R\in\mathcal{S}_j\cap\mathcal{U}}
|R|\operatorname{osc}_p(f;R)^r
\right]^{\frac1r}\nonumber\\
&\quad\leq
\sum_{k=0}^{\infty}\rho_k^{\frac nr}
\sum_{j\in J}
\left[
\sum_{R\in\mathcal{S}_j\cap\mathcal{U}_k}
\operatorname{osc}_p(f;R)^r
\right]^{\frac1r}\nonumber\\
&\quad\leq
\sum_{k=0}^{\infty}\rho_k^{\frac nr}
\sum_{R\in\mathcal{U}_k}\operatorname{osc}_p(f;R)
\leq
M\sum_{k=0}^{\infty}\rho_k^{1-n+\frac nr}\nonumber\\
&\quad=
M\varepsilon^{1-n+\frac nr}
\sum_{k=0}^{\infty}2^{-k(1-n+\frac nr)}
=\frac{M\varepsilon^{1-n+\frac nr}}
{1-2^{-(1-n+\frac nr)}}.
\end{align}
Thus the restriction
$r<\frac{n}{n-1}$ enters this proof precisely through the positivity of
$1-n+\frac nr$, which makes the dyadic high-oscillation geometric series
converge.
Combining \eqref{eq-new-local-cross-exponent},
\eqref{eq-new-low-sparse-energy}, and
\eqref{eq-new-high-sparse-energy}, we obtain
\begin{align*}
\sum_{j\in J}\|f-f_{Q_j}\|_{L^r(Q_j)}
\leq C_{(n,p,r)}M\varepsilon^{1-n+\frac nr}.
\end{align*}
Consequently,
\begin{align*}
\varepsilon^{n-1}\sum_{j\in J}\operatorname{osc}_r(f;Q_j)
=\varepsilon^{n-1-\frac nr}
\sum_{j\in J}\|f-f_{Q_j}\|_{L^r(Q_j)}
\leq C_{(n,p,r)}M.
\end{align*}
Taking both suprema in \eqref{eq-new-Bq}
proves \eqref{eq-new-sparse-cross-exponent} for
real-valued $f$.

For a general complex-valued function $f$, write $f=u+iv$.
Applying the real-valued estimate just proved to $u$ and $v$, we
conclude that
\begin{align*}
\|f\|_{\mathcal{B}_r(Q_0)}
&\leq\|u\|_{\mathcal{B}_r(Q_0)}
+\|v\|_{\mathcal{B}_r(Q_0)}\\
&\leq C_{(n,p,r)}
\left[
\|u\|_{\mathcal{B}_p(Q_0)}
+\|v\|_{\mathcal{B}_p(Q_0)}
\right]\\
&\leq2C_{(n,p,r)}\|f\|_{\mathcal{B}_p(Q_0)}.
\end{align*}
This completes the proof of
Proposition \ref{prop-new-sparse-cross-exponent}.
\end{proof}

We can now prove the main theorem in this section.

\begin{proof}[Proof of Theorem \ref{the-new-self}]
When $q\in(0,1)$, by Jensen's inequality and Proposition
\ref{prop-new-sparse-cross-exponent} with $(p,r)=(q,1)$,
we find that
\begin{align*}
\|f\|_{\mathcal{B}_q(Q_0)}
\leq\|f\|_{\mathcal{B}(Q_0)}
\leq C_{(n,q)}\|f\|_{\mathcal{B}_q(Q_0)}.
\end{align*}
When $q=1$, the assertion is immediate. When
$q\in(1,\frac{n}{n-1})$, using H\"older's inequality and Proposition
\ref{prop-new-sparse-cross-exponent} with $(p,r)=(1,q)$,
we conclude that
\begin{align*}
\|f\|_{\mathcal{B}(Q_0)}
\leq\|f\|_{\mathcal{B}_q(Q_0)}
\leq C_{(n,q)}\|f\|_{\mathcal{B}(Q_0)}.
\end{align*}
These prove \eqref{eq-new-self-norms}.

It remains to prove that the exponent range is sharp.
Applying \cite[Proposition 3]{BBM2015} to $Q_0$,
we can choose a nonnegative representative $f$ of an element of
$\mathcal{B}(Q_0)$ such that
$
f\notin L^{\frac{n}{n-1}}(Q_0).
$
Assume, to the contrary, that $f\in\mathcal{B}_q(Q_0)$ for some
$q\in[\frac{n}{n-1},\infty)$. Testing
the definition in \eqref{eq-new-Bq}
at scale $1$, we obtain
\begin{align*}
\|f-f_{Q_0}\|_{L^q(Q_0)}
=\operatorname{osc}_q(f;Q_0)
\leq\|f\|_{\mathcal{B}_q(Q_0)}.
\end{align*}
Since $f\in L^1(Q_0)$, the average $f_{Q_0}$ is finite, and therefore
$f\in L^q(Q_0)$. Since
$q\geq\frac{n}{n-1}$ and $|Q_0|=1$, H\"older's inequality then implies
\begin{align*}
\|f\|_{L^{\frac{n}{n-1}}(Q_0)}
\leq\|f\|_{L^q(Q_0)}<\infty,
\end{align*}
which is a contradiction.
This completes the proof of Theorem
\ref{the-new-self}.
\end{proof}

\begin{remark}\label{rem-new-self-perspective}
The local median sparse domination and the sparse-sequence estimate
used above are standard ingredients; see
\cite{Lerner2010,Hytonen2014,HanninenLoristSinko2025}. The contribution
specific to the present setting is the cross-exponent estimate in
Proposition \ref{prop-new-sparse-cross-exponent}. Thus one sparse
median-oscillation argument treats both sides of $q=1$, including the
subunit range, without using a John--Nirenberg-type inequality or
duality. It is also worth noting that Theorem
\ref{the-new-self} is
not a formal consequence of the known weak endpoint embedding
$f-f_{Q_0}\in L^{\frac{n}{n-1},\infty}(Q_0)$ for
$f\in\mathcal{B}(Q_0)$
obtained in \cite[Theorem 2]{BBM2015}.
\end{remark}

At the end of this section,
we compare $\mathcal{B}_q(Q_0)$ with Lebesgue spaces modulo
constants. We then determine how
$\mathcal{B}_q(Q_0)$ changes as $q$ varies.

\begin{proposition}\label{prop-new-Lq-position}
\begin{itemize}
\item[{\rm(i)}]
For every $q\in(0,\infty)$, the following inclusion is continuous and
proper:
\begin{align}\label{eq-new-Bq-proper-Lq}
\mathcal{B}_q(Q_0)\subsetneq L^q(Q_0)/\mathbb{K}.
\end{align}

\item[{\rm(ii)}]
Assume that $n\in\mathbb{N}\cap[2,\infty)$.
For every $q\in(0,\frac{n}{n-1})$, the spaces
$\mathcal{B}_q(Q_0)$ and
$L^{\frac{n}{n-1}}(Q_0)/\mathbb{K}$ are incomparable,
namely,
\begin{align*}
\mathcal{B}_q(Q_0)
\nsubseteq L^{\frac{n}{n-1}}(Q_0)/\mathbb{K}
\quad\text{and}\quad
L^{\frac{n}{n-1}}(Q_0)/\mathbb{K}
\nsubseteq\mathcal{B}_q(Q_0).
\end{align*}
In contrast, for every $q\in[\frac{n}{n-1},\infty)$, the following
inclusion is continuous and proper:
\begin{align}\label{eq-new-Bq-proper-critical-Lq}
\mathcal{B}_q(Q_0)
\subsetneq L^{\frac{n}{n-1}}(Q_0)/\mathbb{K}.
\end{align}

\item[{\rm(iii)}]
Assume that $n\in\mathbb{N}\cap[2,\infty)$.
For every $q_1,q_2,q_3\in(0,\infty)$ satisfying
$q_1<\frac{n}{n-1}\leq q_2<q_3$,
\begin{align}\label{eq-new-Bq-strict-scale}
\mathcal{B}_{q_3}(Q_0)
\subsetneq\mathcal{B}_{q_2}(Q_0)
\subsetneq\mathcal{B}_{q_1}(Q_0)
=\mathcal{B}(Q_0).
\end{align}
\end{itemize}
\end{proposition}

\begin{proof}
We first prove {\rm(i)}. Let $q\in(0,\infty)$ and
$f\in\mathcal{B}_q(Q_0)$.
Without loss of generality, we may assume that $f_{Q_0}=0$.
By the definition in \eqref{eq-new-Bq}, we find that
\begin{align*}
\|f+\mathbb{K}\|_{L^q(Q_0)/\mathbb{K}}
&=\inf_{c\in\mathbb{K}}\|f-c\|_{L^q(Q_0)}
\leq\|f\|_{L^q(Q_0)}\\
&=\operatorname{osc}_q(f;Q_0)
\leq\|f\|_{\mathcal{B}_q(Q_0)}.
\end{align*}
This proves the continuous inclusion in
\eqref{eq-new-Bq-proper-Lq}. To prove its properness and the failure
of the reverse inclusions in {\rm(ii)}, we use a disjoint-slab
construction. For every $s\in[1,\infty)$, choose an integer
$M>\max\{s,2\}$ and let
$\varepsilon_k:=2^{-Mk}$ for every $k\in\mathbb{N}$.
Then $
\sum_{k=1}^{\infty}\varepsilon_k
=\frac{1}{2^M-1}<\frac{1}{4}.
$
Thus, we can choose pairwise disjoint intervals
$I_k\subset(\frac{1}{4},\frac{1}{2})$ of length $\varepsilon_k$.
Let $I_k^+$ be the left half of $I_k$ and let
\begin{align*}
g_s:=\sum_{k=1}^{\infty}
2^k\mathbf{1}_{I_k^+\times
(\frac{1}{4},\frac{3}{4})^{n-1}}.
\end{align*}
Using the pairwise disjointness of the supporting rectangles, we
conclude that
\begin{align*}
\|g_s\|_{L^s(Q_0)}^s
=2^{-n}\sum_{k=1}^{\infty}\varepsilon_k2^{ks}
=2^{-n}\sum_{k=1}^{\infty}2^{-k(M-s)}<\infty.
\end{align*}
Since $s\geq1$ and $|Q_0|=1$, this also implies
$g_s\in L^1(Q_0)$.
For every $k\in\mathbb{N}$, the rectangle
$I_k\times(\frac{1}{4},\frac{3}{4})^{n-1}$ can be tiled by $N_k$
axis-parallel cubes $\{Q_{k,j}\}_{j=1}^{N_k}$ of edge length
$\varepsilon_k$, where
\begin{align*}
N_k:=(2\varepsilon_k)^{1-n}
\leq\varepsilon_k^{1-n}.
\end{align*}
Thus,
$\{Q_{k,j}\}_{j=1}^{N_k}\in
\mathscr F_{\varepsilon_k}(Q_0)$.
Since $g_s=2^k$ on one half of $Q_{k,j}$ and $g_s=0$ on the other
half, we have, for every $t\in(0,\infty)$,
$
(g_s)_{Q_{k,j}}
=2^{k-1}$ and
\begin{align*}
\operatorname{osc}_t(g_s;Q_{k,j})
=\left[
\frac{1}{2}\left|2^k-2^{k-1}\right|^t
+\frac{1}{2}\left|2^{k-1}\right|^t
\right]^{\frac{1}{t}}
=2^{k-1}.
\end{align*}
Therefore, for every $k\in\mathbb{N}$,
the definition in \eqref{eq-new-Bq} implies that
\begin{align*}
\|g_s\|_{\mathcal{B}_t(Q_0)}
\geq\varepsilon_k^{n-1}
\sum_{j=1}^{N_k}\operatorname{osc}_t(g_s;Q_{k,j})
=2^{k-n}.
\end{align*}
Letting $k\to\infty$, we conclude that
$
\|g_s\|_{\mathcal{B}_t(Q_0)}=\infty.
$
Combining the preceding two estimates, we deduce that, for every
$s\in[1,\infty)$ and $t\in(0,\infty)$,
\begin{align}\label{eq-new-slab-counterexample}
g_s+\mathbb{K}\in L^s(Q_0)/\mathbb{K}
\quad\text{and}\quad
g_s+\mathbb{K}\notin\mathcal{B}_t(Q_0).
\end{align}
Since
$[g_s-(g_s)_{Q_0}]+\mathbb{K}=g_s+\mathbb{K}$,
the same conclusion holds for the mean-zero representative of $g_s$.
Taking $s:=\max\{1,q\}$ and $t:=q$ in
\eqref{eq-new-slab-counterexample}, and using $|Q_0|=1$, we obtain
\begin{align*}
g_s+\mathbb{K}
\in L^s(Q_0)/\mathbb{K}
\subset L^q(Q_0)/\mathbb{K}
\quad\text{and}\quad
g_s+\mathbb{K}
\notin\mathcal{B}_q(Q_0).
\end{align*}
Thus,
$L^q(Q_0)/\mathbb{K}\nsubseteq\mathcal{B}_q(Q_0)$,
which proves the properness in {\rm(i)}.

Next, we prove {\rm(ii)}. For every $q\in(0,\frac{n}{n-1})$, using Theorem
\ref{the-new-self} and \cite[Proposition 3]{BBM2015}, we conclude that
\begin{align*}
\mathcal{B}_q(Q_0)=\mathcal{B}(Q_0)\nsubseteq
L^{\frac{n}{n-1}}(Q_0)/\mathbb{K}.
\end{align*}
On the other hand, taking
$(s,t)=(\frac{n}{n-1},q)$ in
\eqref{eq-new-slab-counterexample}, we find that
\begin{align*}
L^{\frac{n}{n-1}}(Q_0)/\mathbb{K}
\nsubseteq\mathcal{B}_q(Q_0).
\end{align*}
These two noninclusions prove the first assertion in {\rm(ii)}. Let
$q\in[\frac{n}{n-1},\infty)$ and
$f\in\mathcal{B}_q(Q_0)$ satisfy $f_{Q_0}=0$. By the scale-one estimate
established in {\rm(i)} and H\"older's inequality, we obtain
\begin{align*}
\|f+\mathbb{K}\|_{L^{\frac{n}{n-1}}(Q_0)/\mathbb{K}}
\leq\|f\|_{L^{\frac{n}{n-1}}(Q_0)}
\leq\|f\|_{L^q(Q_0)}
\leq\|f\|_{\mathcal{B}_q(Q_0)}.
\end{align*}
This proves the continuous inclusion in
\eqref{eq-new-Bq-proper-critical-Lq}. Taking
$(s,t)=(\frac{n}{n-1},q)$ in
\eqref{eq-new-slab-counterexample}, we further obtain
\begin{align*}
g_{\frac{n}{n-1}}+\mathbb{K}
\in\left[L^{\frac{n}{n-1}}(Q_0)/\mathbb{K}\right]
\setminus\mathcal{B}_q(Q_0),
\end{align*}
which proves the properness.

Finally, we prove {\rm(iii)}. By H\"older's inequality,
we easily find that
\begin{align*}
\mathcal{B}_{q_3}(Q_0)
\subset\mathcal{B}_{q_2}(Q_0)
\subset\mathcal{B}_{q_1}(Q_0).
\end{align*}
Since $q_1<\frac{n}{n-1}\leq q_2$, Theorem \ref{the-new-self} implies
\begin{align*}
\mathcal{B}_{q_2}(Q_0)
\subsetneq\mathcal{B}_{q_1}(Q_0)
=\mathcal{B}(Q_0).
\end{align*}
It remains to prove that the first inclusion
$\mathcal{B}_{q_3}(Q_0)
\subset\mathcal{B}_{q_2}(Q_0)$ is strict. Let
$
p_2:=\frac{nq_2}{n+q_2}\in[1,n)$. Then
$
\frac{np_2}{n-p_2}=q_2.
$
Let $\varepsilon\in(0,1]$ and
$\{Q_j\}_{j\in J}\in\mathscr F_{\varepsilon}(Q_0)$. Using
Poincar\'e's inequality, H\"older's inequality, the pairwise
disjointness of the cubes, $\#J\leq\varepsilon^{1-n}$,
$p_2\geq1$, and $\varepsilon\leq1$, we find that, for every
$u\in W^{1,p_2}(Q_0)$,
\begin{align*}
\varepsilon^{n-1}\sum_{j\in J}
\operatorname{osc}_{q_2}(u;Q_j)
&\leq C_{(n,q_2)}\varepsilon^{n-\frac{n}{p_2}}
\sum_{j\in J}
\left[\int_{Q_j}|\nabla u(x)|^{p_2}\,dx\right]^{\frac{1}{p_2}}\\
&\leq C_{(n,q_2)}\varepsilon^{n-\frac{n}{p_2}}
(\#J)^{1-\frac{1}{p_2}}
\left[\sum_{j\in J}\int_{Q_j}
|\nabla u(x)|^{p_2}\,dx\right]^{\frac{1}{p_2}}\\
&\leq C_{(n,q_2)}\varepsilon^{1-\frac{1}{p_2}}
\|\nabla u\|_{L^{p_2}(Q_0)}\\
&\leq C_{(n,q_2)}\|\nabla u\|_{L^{p_2}(Q_0)}.
\end{align*}
Taking the supremum over all such $\varepsilon$ and
$\{Q_j\}_{j\in J}$, we conclude that
$
\|u\|_{\mathcal{B}_{q_2}(Q_0)}
\leq C_{(n,q_2)}\|\nabla u\|_{L^{p_2}(Q_0)},
$
and hence
$
W^{1,p_2}(Q_0)\subset\mathcal{B}_{q_2}(Q_0).
$
Choose $x_0\in Q_0$ and $\delta\in(0,\infty)$ such that
$B(x_0,2\delta)\subset Q_0$. Choose
$\eta\in C_{\mathrm{c}}^\infty(Q_0)$ such that
$\eta=1$ on $B(x_0,\delta)$, and then choose $\alpha$ satisfying
$
\frac{n}{q_3}\leq\alpha<\frac{n}{q_2}.
$
Let
\begin{align*}
u(x):=\eta(x)|x-x_0|^{-\alpha}.
\end{align*}
By the choices of $p_2$ and $\alpha$, we have
$(\alpha+1)p_2<n$
and
$\alpha q_3\geq n$.
Using these inequalities and polar coordinates, we find that
\begin{align*}
\int_{Q_0}\left[|u(x)|^{p_2}
+|\nabla u(x)|^{p_2}\right]\,dx
&\leq C_{(n,p_2,\eta)}
\left[1+\int_0^\delta
r^{n-1-(\alpha+1)p_2}\,dr\right]
<\infty
\end{align*}
and
\begin{align*}
\int_{Q_0}|u(x)|^{q_3}\,dx
&\geq C_{(n)}\int_0^\delta
r^{n-1-\alpha q_3}\,dr
=\infty.
\end{align*}
Combining these estimates with the Sobolev inclusion above and the
inclusion proved in {\rm(i)}, we conclude that
\begin{align*}
u+\mathbb{K}\in\mathcal{B}_{q_2}(Q_0),
\quad
u+\mathbb{K}\notin L^{q_3}(Q_0)/\mathbb{K}
\Longrightarrow
u+\mathbb{K}\notin\mathcal{B}_{q_3}(Q_0).
\end{align*}
This proves the strict inclusion in
\eqref{eq-new-Bq-strict-scale} and completes the proof of Proposition
\ref{prop-new-Lq-position}.
\end{proof}

\section{Atomic Preduals:
Duality and Sharp Self-Improvement
\label{sec-new-predual}}

In this section, we first introduce
the Hardy-type space $\mathcal{H}_{s}(Q_0)$
and prove the duality theorem
$[\mathcal{H}_{q'}(Q_0)]^*=\mathcal{B}_{q}(Q_0)$
for every $q\in(1,\infty)$.
We then prove
the sharp self-improvement
$\mathcal{H}_s(Q_0)=\mathcal{H}_\infty(Q_0)$
if and only if $s\in(n,\infty]$ by
using a quantitative decomposition of $s$-blocks into
$\infty$-blocks, and derive the endpoint duality as a consequence.
Finally, we use the same decomposition to
provide an independent
proof of
$\mathcal{B}_q(Q_0)=\mathcal{B}(Q_0)$ for
$q\in(1,\frac{n}{n-1})$. Although the block argument is not needed
for the logical proof of Theorem \ref{the-new-self}, it makes the
conjugate threshold $q'\in(n,\infty)$ explicit
and offers a novel approach to similar problems.

\subsection{Atomic Construction and Duality}

In this subsection, we introduce a unified family of
Bourgain--Brezis--Mironescu $s$-blocks and the associated Hardy-type
atomic spaces, establish their canonical block pairings, and identify
their finite-exponent duals. We begin with the blocks.

\begin{definition}\label{def-new-blocks}
Let $s\in(1,\infty]$ and $\varepsilon\in(0,1]$. A
function $b\in L^s(Q_0)$ is called a
\emph{Bourgain--Brezis--Mironescu $s$-block at scale $\varepsilon$} if
there exist $\{Q_j\}_{j\in J}\in\mathscr{F}_{\varepsilon}(Q_0)$ and
$\{b_j\}_{j\in J}\subset L^1(Q_0)$ such that
\begin{align*}
b=\sum_{j\in J}b_j,
\quad
\operatorname{supp}b_j\subset Q_j,
\quad
\int_{Q_j}b_j(x)\,dx=0,
\quad\text{and }
\|b_j\|_{L^s(Q_j)}\leq\varepsilon^{\frac{n}{s}-1}.
\end{align*}
Here and thereafter, $\frac{n}{\infty}:=0$.
Moreover, let $\mathfrak{A}_s$ denote the collection of all
Bourgain--Brezis--Mironescu $s$-blocks.
\end{definition}

We first record the $L^1$ and testing estimates for a single
Bourgain--Brezis--Mironescu block, which shows that
every Bourgain--Brezis--Mironescu $s$-block $b$
induces an element of $[\mathcal{B}_{s'}(Q_0)]^*$.

\begin{proposition}\label{prop-new-block-functional}
Let $s\in(1,\infty]$.
Every Bourgain--Brezis--Mironescu $s$-block $b$ satisfies
\begin{align}\label{eq-new-block-L1-bound}
\|b\|_{L^1(Q_0)}\leq1.
\end{align}
Moreover, for every $f\in\mathcal{B}_{s'}(Q_0)$, the integral
$\int_{Q_0}f(x)b(x)\,dx$
is well defined and satisfies
\begin{align*}
\left|\int_{Q_0}f(x)b(x)\,dx\right|
\leq\|f\|_{\mathcal{B}_{s'}(Q_0)}.
\end{align*}
Consequently,
$
\|b\|_{[\mathcal{B}_{s'}(Q_0)]^*}\leq1.
$
\end{proposition}

\begin{proof}
Let $b=\sum_{j\in J}b_j$ be an $s$-block at scale
$\varepsilon\in(0,1]$, associated with
$\{Q_j\}_{j\in J}\in\mathscr{F}_{\varepsilon}(Q_0)$.
By Definition \ref{def-new-blocks}, H\"older's inequality, and the
admissibility condition,
\begin{align*}
\|b\|_{L^1(Q_0)}
\leq\sum_{j\in J}|Q_j|^{1-\frac1s}
\|b_j\|_{L^s(Q_j)}
\leq(\#J)\varepsilon^{n(1-\frac1s)}
\varepsilon^{\frac ns-1}
\leq1.
\end{align*}
Choose a representative of $f\in\mathcal{B}_{s'}(Q_0)$. By Proposition
\ref{prop-new-Lq-position}(i), it belongs to $L^{s'}(Q_0)$. Using this,
$\int_{Q_j}b_j(x)\,dx=0$, and H\"older's inequality, we conclude that
\begin{align*}
\left|\int_{Q_0}f(x)b(x)\,dx\right|
&=\left|\sum_{j\in J}\int_{Q_j}[f(x)-f_{Q_j}]b_j(x)\,dx\right|\\
&\leq\sum_{j\in J}
\|f-f_{Q_j}\|_{L^{s'}(Q_j)}
\|b_j\|_{L^{s}(Q_j)}\\
&\leq
\varepsilon^{n-1}\sum_{j\in J}
\operatorname{osc}_{s'}(f;Q_j)
\leq\|f\|_{\mathcal{B}_{s'}(Q_0)}.
\end{align*}
This completes the proof of Proposition \ref{prop-new-block-functional}.
\end{proof}

Based on Proposition \ref{prop-new-block-functional},
we can now use the blocks to define the Hardy-type space
$\mathcal{H}_s(Q_0)$, including
the endpoint case.

\begin{definition}\label{def-new-Hs}
Let $s\in(1,\infty]$.
The \emph{Hardy-type space}
$\mathcal{H}_s(Q_0)$ is defined by setting
\begin{align*}
\mathcal{H}_s(Q_0)
:=\left\{
\sum_{m=1}^{\infty}\lambda_mb_m\text{ in }L^1(Q_0):
\{\lambda_m\}_{m\in\mathbb{N}}\subset\mathbb{K},\
\{b_m\}_{m\in\mathbb{N}}\subset \mathfrak{A}_s,\text{ and }
\sum_{m=1}^{\infty}|\lambda_m|<\infty\right\}.
\end{align*}
Every such series is called an \emph{$s$-block representation} of its
sum.
Moreover, for every $h\in \mathcal{H}_s(Q_0)$,
\begin{align}\label{eq-new-Hs-atomic-norm}
\|h\|_{\mathcal{H}_s(Q_0)}
:=\inf\left\{
\sum_{m=1}^{\infty}|\lambda_m|:
h=\sum_{m=1}^{\infty}\lambda_mb_m
\ \hbox{is an $s$-block representation}
\right\}.
\end{align}
The blocks $b_m$ need not have a common scale.
\end{definition}

\begin{remark}\label{rem-new-Hs-L1-realization}
For every $r\in[1,\infty]$, define
\begin{align*}
L_0^r(Q_0)
:=
\left\{
g\in L^r(Q_0):
\int_{Q_0}g(x)\,dx=0
\right\}.
\end{align*}
By \eqref{eq-new-block-L1-bound}, every series occurring in
Definition \ref{def-new-Hs} satisfies
\begin{align*}
\sum_{m=1}^{\infty}
\|\lambda_mb_m\|_{L^1(Q_0)}
\leq\sum_{m=1}^{\infty}|\lambda_m|<\infty.
\end{align*}
Consequently,
\begin{align*}
\sum_{m=1}^{N}\lambda_mb_m
\to\sum_{m=1}^{\infty}\lambda_mb_m
\quad\text{in }L^1(Q_0)
\end{align*}
as $N\to\infty$.
Since every block has integral
zero, continuity of the integral under $L^1(Q_0)$ convergence implies
$
\mathcal{H}_s(Q_0)\subset L_0^1(Q_0)
$
and, for every $h\in\mathcal{H}_s(Q_0)$,
$
\|h\|_{L^1(Q_0)}
\leq\|h\|_{\mathcal{H}_s(Q_0)}.
$
More precisely, if $\ell^1(\mathfrak{A}_s)$ denotes the space of
absolutely summable scalar families indexed by $\mathfrak{A}_s$, then
the \emph{synthesis map}
\begin{align*}
T_s:
\begin{cases}
\displaystyle
\ell^1(\mathfrak{A}_s)
\longrightarrow L_0^1(Q_0)\\
\displaystyle
(\lambda_b)_{b\in\mathfrak{A}_s}\longmapsto
\sum_{b\in\mathfrak{A}_s}\lambda_bb
\end{cases}
\end{align*}
is bounded. For every
$(\lambda_b)_{b\in\mathfrak{A}_s}\in\ell^1(\mathfrak{A}_s)$, at most
countably many coefficients $\lambda_b$ are nonzero, so the displayed
sum is unambiguous. Definition \ref{def-new-Hs} identifies
$\mathcal{H}_s(Q_0)$ isometrically with
$\ell^1(\mathfrak{A}_s)/\ker T_s$. Since $\ker T_s$ is closed,
we find that $\mathcal{H}_s(Q_0)$ is a Banach space with norm
$\|\cdot\|_{\mathcal{H}_s(Q_0)}$.
\end{remark}

We now present the main result of this subsection,
which proves
$[\mathcal{H}_{q'}(Q_0)]^*=\mathcal{B}_q(Q_0)$
for every $q\in(1,\infty)$.

\begin{theorem}\label{the-new-Hs-duality}
Let $q\in(1,\infty)$. Then
$
[\mathcal{H}_{q'}(Q_0)]^*
=\mathcal{B}_q(Q_0)
$
with equivalent norms in the following sense:
\begin{itemize}
\item[{\rm(i)}]
For every $f\in\mathcal{B}_q(Q_0)$,
the linear functional $\mathcal{L}_f$ defined
by setting, for every
$g\in L_0^{q'}(Q_0)$,
\begin{align*}
\mathcal{L}_f(g):=\int_{Q_0}f(x)g(x)\,dx
\end{align*}
extends uniquely to a bounded linear functional on
$\mathcal{H}_{q'}(Q_0)$. Moreover,
\begin{align}\label{eq-new-Hs-duality-norms}
\frac{1}{2}\|f\|_{\mathcal{B}_q(Q_0)}
\leq
\|\mathcal{L}_f\|_{[\mathcal{H}_{q'}(Q_0)]^*}
\leq
\|f\|_{\mathcal{B}_q(Q_0)}.
\end{align}
\item[{\rm(ii)}]
For every
$\mathcal{L}\in[\mathcal{H}_{q'}(Q_0)]^*$, there exists a unique
$f\in\mathcal{B}_q(Q_0)$ (that is, unique modulo constants)
such that, for every $g\in L_0^{q'}(Q_0)$,
\begin{align}\label{eq-new-Hs-representation-functional}
\mathcal{L}(g)=\int_{Q_0}f(x)g(x)\,dx.
\end{align}
Moreover, \eqref{eq-new-Hs-duality-norms} holds
with $\mathcal{L}_f=\mathcal{L}$.
\end{itemize}
\end{theorem}

To show Theorem \ref{the-new-Hs-duality},
we first establish two technical lemmas.
We begin with the density
of $L_0^s(Q_0)$ in $\mathcal{H}_s(Q_0)$
and that every $f\in\mathcal{B}_{s'}(Q_0)$
induces an element of $[\mathcal{H}_s(Q_0)]^*$.

\begin{lemma}
\label{lem-new-atomic-core-properties}
The following statements hold.
\begin{itemize}
\item[{\rm(i)}]
Let $s\in(1,\infty]$. Then $L_0^s(Q_0)$ is dense in
$\mathcal{H}_s(Q_0)$ and, for every $g\in L_0^s(Q_0)$,
\begin{align}\label{eq-new-Hs-core-bound}
\|g\|_{\mathcal{H}_s(Q_0)}
\leq
\|g\|_{L^s(Q_0)}.
\end{align}

\item[{\rm(ii)}]
Let $s\in(1,\infty]$. If
$f\in\mathcal{B}_{s'}(Q_0)$, $h\in\mathcal{H}_s(Q_0)$, and
$h=\sum_{m=1}^{\infty}\lambda_mb_m$ is any $s$-block representation,
then
\begin{align}\label{eq-new-canonical-block-pairing}
\langle f,h\rangle_s
:=\sum_{m=1}^{\infty}\lambda_m
\int_{Q_0}f(x)b_m(x)\,dx
\end{align}
is absolutely convergent and independent of the $s$-block
representation. Moreover,
\begin{align}\label{eq-new-canonical-block-pairing-bound}
|\langle f,h\rangle_s|
\leq
\|f\|_{\mathcal{B}_{s'}(Q_0)}
\|h\|_{\mathcal{H}_s(Q_0)}.
\end{align}
\end{itemize}
\end{lemma}

\begin{proof}
We first prove (i). Let $s\in(1,\infty]$
and $g\in L_0^s(Q_0)$. Without loss of generality,
we may assume that $\|g\|_{L^s(Q_0)}\neq 0$. Then
$\frac{g}{\|g\|_{L^s(Q_0)}}$ is an $s$-block at scale $1$
and hence
$
\|g\|_{\mathcal{H}_s(Q_0)}
\leq\|g\|_{L^s(Q_0)}.
$
This proves \eqref{eq-new-Hs-core-bound}. Let
$h\in\mathcal{H}_s(Q_0)$, fix an $s$-block representation
\begin{align*}
h=\sum_{m=1}^{\infty}\lambda_mb_m
\quad\hbox{in }L^1(Q_0),
\end{align*}
and, for every $N\in\mathbb{N}$, let
$
h_N:=\sum_{m=1}^{N}\lambda_mb_m
$
Then, for every $N\in\mathbb{N}$,
$h_N\in L_0^s(Q_0)$ and
\begin{align*}
\|h-h_N\|_{\mathcal{H}_s(Q_0)}
\leq\sum_{m=N+1}^{\infty}|\lambda_m|
\to0
\end{align*}
as $N\to\infty$.
Thus the family of
finite $s$-block sums, and hence $L_0^s(Q_0)$, is dense in
$\mathcal{H}_s(Q_0)$. This proves (i).

We now prove (ii). From Proposition
\ref{prop-new-block-functional},
we infer that the right-hand side of
\eqref{eq-new-canonical-block-pairing}
converges absolutely. To prove
the independence of the
representation, it suffices to suppose that
\begin{align}\label{eq-new-zero-L1-block-sum}
\sum_{m=1}^{\infty}\lambda_mb_m=0
\quad\hbox{in }L^1(Q_0)
\end{align}
and prove that the series in
\eqref{eq-new-canonical-block-pairing} has sum zero.

Choose a representative of $f\in\mathcal{B}_{s'}(Q_0)$.
By Proposition \ref{prop-new-Lq-position}(i), it belongs to
$L^{s'}(Q_0)$.
For every $N\in\mathbb{N}$, define
$P_N:\mathbb K\to\mathbb K$ by setting, for every $z\in\mathbb K$,
\begin{align*}
P_N(z):=
\begin{cases}
z,& |z|\leq N,\\[1mm]
\displaystyle N\frac{z}{|z|},& |z|>N,
\end{cases}
\end{align*}
and let $f_N:=P_N\circ f$. When $\mathbb K=\mathbb R$, $P_N$ is the
usual truncation
$\max\{-N,\min\{t,N\}\}$. Viewing $\mathbb K$ as a real
Hilbert space, $P_N$ is the metric projection onto a closed convex ball
and is therefore $1$-Lipschitz. In particular, $f_N$ is bounded.
By the definition of $f_N$,
H\"older's inequality, the $1$-Lipschitz property of $P_N$, and
Minkowski's inequality, we conclude that
\begin{align*}
\operatorname{osc}_{s'}(f_N;Q)
&=\left[\fint_Q\left|
\fint_Q[f_N(x)-f_N(y)]\,dy\right|^{s'}\,dx\right]^{\frac1{s'}}\\
&\leq
\left[\fint_Q\fint_Q
|f_N(x)-f_N(y)|^{s'}\,dy\,dx\right]^{\frac1{s'}}\\
&\leq
\left[\fint_Q\fint_Q
|f(x)-f(y)|^{s'}\,dy\,dx\right]^{\frac1{s'}}\\
&\leq
\left[\fint_Q\fint_Q
|f(x)-f_Q|^{s'}\,dy\,dx\right]^{\frac1{s'}}
+\left[\fint_Q\fint_Q
|f(y)-f_Q|^{s'}\,dy\,dx\right]^{\frac1{s'}}\\
&=2\operatorname{osc}_{s'}(f;Q).
\end{align*}
This further implies that
\begin{align}\label{eq-new-pairing-truncation-bounds}
\|f_N\|_{\mathcal{B}_{s'}(Q_0)}
\leq2\|f\|_{\mathcal{B}_{s'}(Q_0)}
\quad\text{and}\quad
\|f-f_N\|_{\mathcal{B}_{s'}(Q_0)}
\leq3\|f\|_{\mathcal{B}_{s'}(Q_0)}.
\end{align}
Moreover, since $|f_N|\leq|f|$ and $f_N\to f$ almost everywhere
as $N\to\infty$, the dominated convergence theorem implies
$f_N\to f$ in $L^{s'}(Q_0)$ as $N\to\infty$. Since $f_N$ is bounded, the
$L^1(Q_0)$ convergence in
\eqref{eq-new-zero-L1-block-sum} implies
\begin{align*}
\sum_{m=1}^{\infty}\lambda_m
\int_{Q_0}f_N(x)b_m(x)\,dx
=0.
\end{align*}
From this, Proposition
\ref{prop-new-block-functional}, and
\eqref{eq-new-pairing-truncation-bounds},
we infer that, for every
$M\in\mathbb{N}$,
\begin{align*}
\left|
\sum_{m=1}^{\infty}\lambda_m
\int_{Q_0}f(x)b_m(x)\,dx
\right|
&\leq
\sum_{m=1}^{M}|\lambda_m|
\left|\int_{Q_0}[f(x)-f_N(x)]b_m(x)\,dx\right|\\
&\quad+
3\|f\|_{\mathcal{B}_{s'}(Q_0)}
\sum_{m=M+1}^{\infty}|\lambda_m|.
\end{align*}
Letting first $N\to\infty$ and then
$M\to\infty$ on both sides,
we obtain the required representation independence.

Finally, let
$h=\sum_{m=1}^{\infty}\lambda_mb_m$
be any $s$-block representation. Then, by Proposition
\ref{prop-new-block-functional}, we have
\begin{align*}
|\langle f,h\rangle_s|
\leq
\sum_{m=1}^{\infty}|\lambda_m|
\left|\int_{Q_0}f(x)b_m(x)\,dx\right|
\leq
\|f\|_{\mathcal{B}_{s'}(Q_0)}
\sum_{m=1}^{\infty}|\lambda_m|.
\end{align*}
Since the left-hand side is independent of the chosen representation,
taking the infimum over all $s$-block representations of $h$ and using
\eqref{eq-new-Hs-atomic-norm}, we obtain
\begin{align*}
|\langle f,h\rangle_s|
\leq
\|f\|_{\mathcal{B}_{s'}(Q_0)}
\|h\|_{\mathcal{H}_s(Q_0)},
\end{align*}
which is \eqref{eq-new-canonical-block-pairing-bound}.
This completes the proof of Lemma \ref{lem-new-atomic-core-properties}.
\end{proof}

In what follows, for every $f\in\mathcal{B}_{s'}(Q_0)$ and
$h\in\mathcal{H}_s(Q_0)$, the notation
$\langle f,h\rangle_s$ always denotes the
\emph{canonical block pairing} defined by
\eqref{eq-new-canonical-block-pairing}. By Lemma
\ref{lem-new-atomic-core-properties}(ii), its value is independent of
the chosen $s$-block representation of $h$. This notation does not
presuppose that $fh\in L^1(Q_0)$.
We next establish a norming-block identity for the
canonical pairing.

\begin{lemma}\label{lem-new-block-pairing}
Let $q\in(1,\infty)$. If $f\in L^q(Q_0)$,
$\varepsilon\in(0,1]$, and
$\{Q_j\}_{j\in J}\in\mathscr{F}_{\varepsilon}(Q_0)$, then there
exists a Bourgain--Brezis--Mironescu $q'$-block $b$ associated with the
same family such that
\begin{align}\label{eq-new-norming-block-pairing}
\int_{Q_0}f(x)b(x)\,dx
=\frac12\varepsilon^{n-1}\sum_{j\in J}
\operatorname{osc}_q(f;Q_j).
\end{align}
\end{lemma}

\begin{proof}
Fix $q\in(1,\infty)$, $f\in L^q(Q_0)$,
$\varepsilon\in(0,1]$, and
$\{Q_j\}_{j\in J}\in\mathscr{F}_{\varepsilon}(Q_0)$.
Let $\omega_j:=f-f_{Q_j}$ and define
\begin{align*}
\varphi_j:=
\begin{cases}\displaystyle
\frac{|\omega_j|^{q-2}\overline{\omega_j}}
{\|\omega_j\|_{L^q(Q_j)}^{q-1}},
&\|\omega_j\|_{L^q(Q_j)}\neq0,\\[3mm]
0,&\|\omega_j\|_{L^q(Q_j)}=0.
\end{cases}
\end{align*}
By $(q-1)q'=q$, we easily find that
\begin{align*}
\|\varphi_j\|_{L^{q'}(Q_j)}\leq1
\quad\text{and }
\int_{Q_j}\omega_j(x)\varphi_j(x)\,dx
=\|\omega_j\|_{L^q(Q_j)}.
\end{align*}
For every $j\in J$, let
\begin{align*}
b_j:=\frac12\varepsilon^{n-1}|Q_j|^{-\frac1q}
[\varphi_j-(\varphi_j)_{Q_j}]\mathbf{1}_{Q_j}
\quad\text{and }
b:=\sum_{j\in J}b_j.
\end{align*}
Then $\int_{Q_j}b_j(x)\,dx=0$ for every $j\in J$.
In addition, using H\"older's inequality, we conclude that
\begin{align*}
\|b_j\|_{L^{q'}(Q_j)}
&\leq\varepsilon^{n-1}|Q_j|^{-\frac1q}
\|\varphi_j\|_{L^{q'}(Q_j)}
\leq\varepsilon^{\frac{n}{q'}-1}.
\end{align*}
Thus $b$ is a $q'$-block. Since
$\int_{Q_j}b_j(x)\,dx=\int_{Q_j}\omega_j(x)\,dx=0$, we have
\begin{align*}
\int_{Q_0}f(x)b(x)\,dx
&=\sum_{j\in J}\int_{Q_j}\omega_j(x)b_j(x)\,dx\\
&=\frac12\varepsilon^{n-1}
\sum_{j\in J}|Q_j|^{-\frac1q}
\int_{Q_j}\omega_j(x)
\big[\varphi_j(x)-(\varphi_j)_{Q_j}\big]\,dx\\
&=\frac12\varepsilon^{n-1}
\sum_{j\in J}|Q_j|^{-\frac1q}
\int_{Q_j}\omega_j(x)\varphi_j(x)\,dx\\
&=\frac12\varepsilon^{n-1}
\sum_{j\in J}|Q_j|^{-\frac1q}
\|\omega_j\|_{L^q(Q_j)}\\
&=\frac12\varepsilon^{n-1}
\sum_{j\in J}\operatorname{osc}_q(f;Q_j),
\end{align*}
which is \eqref{eq-new-norming-block-pairing}.
This completes the proof of Lemma \ref{lem-new-block-pairing}.
\end{proof}

We now prove Theorem \ref{the-new-Hs-duality}.

\begin{proof}[Proof of Theorem \ref{the-new-Hs-duality}]
We divide the proof into two steps.

\textbf{Step 1.} \noindent\textit{In this step, we prove
{\rm(i)}.}
Let $f\in\mathcal{B}_q(Q_0)$. For $h\in\mathcal{H}_{q'}(Q_0)$, let
\begin{align*}
\widetilde{\mathcal{L}}_f(h):=\langle f,h\rangle_{q'}.
\end{align*}
Every nonzero element of $L_0^{q'}(Q_0)$, after normalization by its
$L^{q'}(Q_0)$ norm, is a $q'$-block at scale $1$. Hence, using the
one-block representation
$
g=\|g\|_{L^{q'}(Q_0)}
\frac{g}{\|g\|_{L^{q'}(Q_0)}}
$
in \eqref{eq-new-canonical-block-pairing}, we obtain, for every
nonzero $g\in L_0^{q'}(Q_0)$,
\begin{align*}
\widetilde{\mathcal{L}}_f(g)
&=\|g\|_{L^{q'}(Q_0)}
\int_{Q_0}f(x)
\frac{g(x)}{\|g\|_{L^{q'}(Q_0)}}\,dx\\
&=\int_{Q_0}f(x)g(x)\,dx
=\mathcal{L}_f(g).
\end{align*}
The same identity is immediate if $\|g\|_{L^{q'}(Q_0)}=0$. Moreover,
$\widetilde{\mathcal{L}}_f$ is independent of the representative of
$f$. By Lemma \ref{lem-new-atomic-core-properties}(ii),
we have
\begin{align*}
|\widetilde{\mathcal{L}}_f(h)|
&\leq
\|h\|_{\mathcal{H}_{q'}(Q_0)}
\|f\|_{\mathcal{B}_q(Q_0)}.
\end{align*}
Thus $\widetilde{\mathcal{L}}_f$ is a bounded extension of
$\mathcal{L}_f$. Since $L_0^{q'}(Q_0)$ is dense in
$\mathcal{H}_{q'}(Q_0)$ by Lemma
\ref{lem-new-atomic-core-properties}(i), this bounded extension is
unique. We henceforth denote it
by $\mathcal{L}_f$. Consequently,
\begin{align}\label{eq-L-B-1}
\|\mathcal{L}_f\|_{[\mathcal{H}_{q'}(Q_0)]^*}
\leq
\|f\|_{\mathcal{B}_q(Q_0)}.
\end{align}
For the reverse inequality, fix $\varepsilon\in(0,1]$ and
$\{Q_j\}_{j\in J}\in\mathscr{F}_{\varepsilon}(Q_0)$.
Choose a representative of $f$. By Proposition
\ref{prop-new-Lq-position}(i), it belongs to $L^q(Q_0)$. Applying Lemma
\ref{lem-new-block-pairing} to this
representative
and the given family,
and choosing a $q'$-block $b_f$ satisfying
\eqref{eq-new-norming-block-pairing}, we have
\begin{align*}
\frac12\varepsilon^{n-1}\sum_{j\in J}
\operatorname{osc}_q(f;Q_j)
=|\mathcal{L}_f(b_f)|
\leq\|\mathcal{L}_f\|_{[\mathcal{H}_{q'}(Q_0)]^*},
\end{align*}
where the last step follows from
$\|b_f\|_{\mathcal{H}_{q'}(Q_0)}\leq1$.
Taking the two suprema establishes the lower bound in
\eqref{eq-new-Hs-duality-norms}.
Combining this and \eqref{eq-L-B-1}, we obtain
\eqref{eq-new-Hs-duality-norms} and hence complete {\rm(i)}.

\textbf{Step 2.} \noindent\textit{In this step, we prove
{\rm(ii)}.}
Let $\mathcal{L}\in[\mathcal{H}_{q'}(Q_0)]^*$. By
Lemma \ref{lem-new-atomic-core-properties}(i), we conclude that,
for every
$g\in L_0^{q'}(Q_0)$,
\begin{align*}
|\mathcal{L}(g)|
\leq\|\mathcal{L}\|_{[\mathcal{H}_{q'}(Q_0)]^*}
\|g\|_{\mathcal{H}_{q'}(Q_0)}
\leq\|\mathcal{L}\|_{[\mathcal{H}_{q'}(Q_0)]^*}
\|g\|_{L^{q'}(Q_0)}.
\end{align*}
By the Hahn--Banach theorem and the
$L^{q'}(Q_0)$--$L^q(Q_0)$ duality, since $q\in(1,\infty)$,
we conclude that there exists a unique element
$f\in L^q(Q_0)/\mathbb{K}$ such that,
for every $g\in L_0^{q'}(Q_0)$,
\begin{align*}
\mathcal{L}(g)=\int_{Q_0}f(x)g(x)\,dx.
\end{align*}
Here uniqueness follows from
$(L_0^{q'}(Q_0))^\perp=\mathbb{K}$, where the annihilator
is taken in $L^q(Q_0)$.
Thus we obtain \eqref{eq-new-Hs-representation-functional}.
We choose the unique representative satisfying
$f_{Q_0}=0$.

It remains to prove that
$f\in\mathcal{B}_q(Q_0)$.
Fix $\varepsilon\in(0,1]$ and
$\{Q_j\}_{j\in J}\in\mathscr{F}_{\varepsilon}(Q_0)$.
Applying Lemma \ref{lem-new-block-pairing} to the
representing function $f\in L^q(Q_0)$,
we can choose a $q'$-block $b$ satisfying
\eqref{eq-new-norming-block-pairing}.
Using $b\in L_0^{q'}(Q_0)$,
\eqref{eq-new-Hs-representation-functional}, and
$\|b\|_{\mathcal{H}_{q'}(Q_0)}\leq1$, we obtain
\begin{align*}
\frac12\varepsilon^{n-1}\sum_{j\in J}
\operatorname{osc}_q(f;Q_j)
=\int_{Q_0}f(x)b(x)\,dx
=|\mathcal{L}(b)|
\leq\|\mathcal{L}\|_{[\mathcal{H}_{q'}(Q_0)]^*}.
\end{align*}
This further implies that
$f\in\mathcal{B}_q(Q_0)$. By {\rm(i)},
$\mathcal{L}_f$ is a bounded linear functional on
$\mathcal{H}_{q'}(Q_0)$. Since $\mathcal{L}$ and
$\mathcal{L}_f$ agree on the dense subspace
$L_0^{q'}(Q_0)$, they agree on $\mathcal{H}_{q'}(Q_0)$.
Consequently, \eqref{eq-new-Hs-duality-norms} applies with
$\mathcal{L}_f=\mathcal{L}$.
Since $\mathcal{B}_q(Q_0)$ is defined modulo constants, the uniqueness
of $f$ in $\mathcal{B}_q(Q_0)$ follows from the uniqueness of the
representing class in $L^q(Q_0)/\mathbb{K}$ established above. This
completes the proof of Theorem
\ref{the-new-Hs-duality}.
\end{proof}

\subsection{Sharp Self-Improvement
of $\mathcal{H}_s(Q_0)$}

In this subsection, we study the
sharp self-improvement of $\mathcal{H}_s(Q_0)$,
that is, $\mathcal{H}_s(Q_0)
=\mathcal{H}_{\infty}(Q_0)$ if and only if $s\in(n,\infty]$,
which is the
conjugate result of Theorem
\ref{the-new-self}.
Moreover, by Definition \ref{def-new-Hs} and Remark
\ref{rem-new-Hs-L1-realization}, we find that
every $\mathcal{H}_s(Q_0)$ is a
subspace of $L_0^1(Q_0)$. Thus the equality
$\mathcal{H}_s(Q_0)
=\mathcal{H}_{\infty}(Q_0)$ for every $s\in(n,\infty)$
studied in this subsection
means equality as subspaces of $L_0^1(Q_0)$
but with their own block norms.

Throughout this subsection, lowercase $b$ denotes a finite-exponent
block or one of its local components, whereas capital $B$ is reserved
for $\infty$-blocks and their components.
We now present the main result in this subsection.

\begin{theorem}\label{the-new-Hs-self}
Let $n\in\mathbb{N}\cap[2,\infty)$ and $s\in(1,\infty]$. Then
as subspaces of $L_0^1(Q_0)$,
\begin{align*}
\mathcal{H}_s(Q_0)=\mathcal{H}_\infty(Q_0)
\Longleftrightarrow
s\in(n,\infty].
\end{align*}
Moreover, if $s\in(n,\infty)$, then, for every
$g\in\mathcal{H}_s(Q_0)$,
\begin{align}\label{eq-new-Hs-self-norms}
\|g\|_{\mathcal{H}_s(Q_0)}
\leq
\|g\|_{\mathcal{H}_\infty(Q_0)}
\leq
C_{(n,s)}\|g\|_{\mathcal{H}_s(Q_0)}.
\end{align}
\end{theorem}

\begin{remark}
Let $q\in[1,\infty)$. The
conclusion of Theorem \ref{the-new-Hs-self} can equivalently be written as
\begin{align*}
\mathcal{H}_{q'}(Q_0)=\mathcal{H}_\infty(Q_0)
\Longleftrightarrow
q\in\left[1,\frac{n}{n-1}\right).
\end{align*}
\end{remark}

To show Theorem \ref{the-new-Hs-self},
we need to establish two technical lemmas. First recall that
the underlying Calder\'on--Zygmund iteration is the standard argument
used to decompose atoms satisfying an $L^s$ size condition into
$L^\infty$-atoms; see
\cite[Theorem (2.8)]{TaiblesonWeiss1980} and
\cite{CoifmanWeiss1977}. However, the cited
results do not directly provide the
generation-by-generation bounds,
uniform over all cubes in an admissible family,
that are needed to regroup the resulting pieces
into Bourgain--Brezis--Mironescu blocks.
We therefore state and prove
the precise quantitative version needed below.
\begin{lemma}
\label{lem-new-dyadic-decomposition}
Let $s\in(1,\infty)$, $Q$ be a cube, and $F\in L^s(Q)$
satisfy $\int_QF(x)\,dx=0$
and
$\|F\|_{L^s(Q)}\leq|Q|^{\frac{1}{s}}$.
Then there exist $G\in L^\infty(Q)$ and, for each
$k\in\mathbb{Z}_+$, an at most countable index set $J_k$, a collection
$\{R_{k,j}\}_{j\in J_k}$ of dyadic subcubes of $Q$ with pairwise
disjoint interiors, and functions
$\{H_{k,j}\}_{j\in J_k}\subset L^\infty(Q)$ such that
\begin{align}\label{eq-new-dyadic-expansion}
F=G+\sum_{k=0}^{\infty}\sum_{j\in J_k}H_{k,j}
\quad\hbox{in }L^s(Q),
\end{align}
where each inner sum converges unconditionally in $L^s(Q)$,
while the outer series converges in $L^s(Q)$ when summed in increasing
order of $k$.
Moreover,
\begin{align}\label{eq-new-dyadic-good}
\int_QG(x)\,dx=0,
\quad
\|G\|_{L^\infty(Q)}\leq C_{(n,s)},
\end{align}
and, for every $k\in\mathbb{Z}_+$ and $j\in J_k$,
\begin{align}\label{eq-new-dyadic-bad}
&\operatorname{supp}H_{k,j}\subset R_{k,j},
\quad
\int_{R_{k,j}}H_{k,j}(x)\,dx=0,
\quad\text{and}\quad
\|H_{k,j}\|_{L^\infty(R_{k,j})}
\leq C_{(n,s)}2^k.
\end{align}
Finally, for every $k\in\mathbb{Z}_+$
and $j\in J_k$,
\begin{align}\label{eq-new-dyadic-volume}
\sum_{j\in J_k}|R_{k,j}|
\leq2^{-ks}\int_Q|F(x)|^s\,dx
\quad\text{and}\quad
\ell(R_{k,j})
&\leq\ell(Q)2^{-\frac{s k}{n}}.
\end{align}
\end{lemma}

\begin{proof}
We divide the proof into three steps. Throughout the proof, a level
index $k$ ranges over $\mathbb{Z}_+$ and an index $j$ attached to level
$k$ ranges over $J_k$, unless another range is explicitly stated.

\textbf{Step 1.} \noindent\textit{In this step, we construct the stopping cubes and
record their nesting and size.}
For each $k\in\mathbb{Z}_+$, let
\begin{align*}
\mathcal{S}_k
:=\left\{R\in\mathcal{D}(Q):
\fint_R|F(x)|^s\,dx>2^{ks}\right\},
\end{align*}
and let $\{R_{k,j}\}_{j\in J_k}$ be the maximal cubes in
$\mathcal{S}_k$. We call these maximal cubes the
\emph{level-$k$ stopping cubes}; thus they are precisely the maximal
dyadic subcubes of $Q$ on which the average of $|F|^s$ exceeds the
threshold $2^{ks}$. Every cube in $\mathcal{S}_k$ is contained in a
maximal one, since it has only finitely many dyadic ancestors in $Q$.
The dyadic nesting property then shows that the cubes at each fixed
level have pairwise disjoint interiors. For every
$k\in\mathbb{Z}_+$,
\begin{align*}
\fint_Q|F(x)|^s\,dx\leq1\leq2^{ks},
\end{align*}
so $Q\notin\mathcal{S}_k$. Consequently, for every
$k\in\mathbb{Z}_+$ and $j\in J_k$, the cube $R_{k,j}$ has a dyadic parent $\widehat{R_{k,j}}$, which does not
belong to $\mathcal{S}_k$ by maximality. It follows that, for every
$k\in\mathbb{Z}_+$ and $j\in J_k$,
\begin{align}\label{eq-new-dyadic-stopping-averages}
2^{ks}
<\fint_{R_{k,j}}|F(x)|^s\,dx
\leq2^n2^{ks}.
\end{align}
For every $k\in\mathbb{Z}_+$, the inequality
$2^{(k+1)s}>2^{ks}$ implies
$\mathcal{S}_{k+1}\subset\mathcal{S}_k$. Fix
$k\in\mathbb{Z}_+$ and $i\in J_{k+1}$. Then
$R_{k+1,i}\in\mathcal{S}_k$, so the maximal-cube property established
above implies that there exists $j\in J_k$ such that
$R_{k+1,i}\subset R_{k,j}$. For each $k\in\mathbb{Z}_+$, let
$$
E_k:=\bigcup_{j\in J_k}R_{k,j}.
$$
The preceding containment shows that $E_{k+1}\subset E_k$ for every
$k\in\mathbb{Z}_+$.

Moreover, from the lower bound in
\eqref{eq-new-dyadic-stopping-averages},
the disjointness of cubes at level $k$,
and $\|F\|_{L^s(Q)}\leq|Q|^{\frac{1}{s}}$,
we infer that,
for every $k\in\mathbb{Z}_+$,
\begin{align*}
\sum_{j\in J_k}|R_{k,j}|
\leq2^{-ks}\sum_{j\in J_k}
\int_{R_{k,j}}|F(x)|^s\,dx
\leq2^{-ks}\int_Q|F(x)|^s\,dx
\end{align*}
and, for every $k\in\mathbb{Z}_+$ and $j\in J_k$,
\begin{align*}
2^{ks}|R_{k,j}|
<\int_{R_{k,j}}|F(x)|^s\,dx
\leq\int_Q|F(x)|^s\,dx
\leq|Q|.
\end{align*}
Consequently,
$
\ell(R_{k,j})\leq\ell(Q)2^{-\frac{s k}{n}}
$
for every $k\in\mathbb{Z}_+$ and $j\in J_k$. This proves
\eqref{eq-new-dyadic-volume}.

\textbf{Step 2.} \noindent\textit{In this step, we construct bounded approximants
that preserve the relevant integrals and satisfy
$G_k\to F$ in $L^s(Q)$ as $k\to\infty$.}
For each $k\in\mathbb{Z}_+$, define
\begin{align*}
G_k
:=F\mathbf{1}_{Q\setminus E_k}
+\sum_{j\in J_k}
F_{R_{k,j}}\mathbf{1}_{R_{k,j}}.
\end{align*}
Fix $k\in\mathbb{Z}_+$. If $x\in Q\setminus E_k$ is a dyadic
Lebesgue point of $|F|^s$, then no dyadic cube containing $x$ belongs
to $\mathcal{S}_k$. The dyadic differentiation theorem therefore implies
$|F(x)|\leq2^k$. Using \eqref{eq-new-dyadic-stopping-averages} and
H\"older's inequality, we find that, for every $j\in J_k$,
\begin{align*}
|F_{R_{k,j}}|
\leq\left[
\fint_{R_{k,j}}|F(x)|^s\,dx
\right]^{\frac{1}{s}}
\leq2^{\frac{n}{s}}2^k.
\end{align*}
Consequently, for every $k\in\mathbb{Z}_+$,
\begin{align}\label{eq-new-dyadic-Gk-bound}
\|G_k\|_{L^\infty(Q)}
\leq C_{(n,s)}2^k.
\end{align}
By the definition of $G_k$, the pairwise disjointness of the
level-$k$ stopping cubes,
and $\int_QF(x)\,dx=0$, we find that, for every $k\in\mathbb{Z}_+$,
\begin{align*}
\int_QG_k(x)\,dx
&=\int_{Q\setminus E_k}F(x)\,dx
+\sum_{j\in J_k}\int_{R_{k,j}}F_{R_{k,j}}(x)\,dx\\
&=\int_{Q\setminus E_k}F(x)\,dx
+\sum_{j\in J_k}\int_{R_{k,j}}F(x)\,dx\\
&=\int_QF(x)\,dx=0.
\end{align*}
Fix $k\in\mathbb{Z}_+$ and $j\in J_k$. For each
$i\in J_{k+1}$, the dyadic nesting property and the containment proved
in Step 1 imply that either $R_{k+1,i}\subset R_{k,j}$ or
$|R_{k+1,i}\cap R_{k,j}|=0$. Let
\begin{align*}
J_{k+1}(j)
:=\left\{i\in J_{k+1}:R_{k+1,i}\subset R_{k,j}\right\}.
\end{align*}
Up to a null set, the level-$(k+1)$ stopping cubes
form the pairwise
disjoint decomposition
\begin{align*}
R_{k,j}
=\left(R_{k,j}\setminus E_{k+1}\right)
\cup\bigcup_{i\in J_{k+1}(j)}R_{k+1,i}.
\end{align*}
By the definitions of $G_{k+1}$ and $G_k$, we obtain
\begin{align}\label{eq-new-dyadic-two-level-integrals}
\int_{R_{k,j}}G_{k+1}(x)\,dx
&=\int_{R_{k,j}\setminus E_{k+1}}F(x)\,dx
+\sum_{i\in J_{k+1}(j)}
|R_{k+1,i}|F_{R_{k+1,i}}\nonumber\\
&=\int_{R_{k,j}\setminus E_{k+1}}F(x)\,dx
+\sum_{i\in J_{k+1}(j)}
\int_{R_{k+1,i}}F(x)\,dx\nonumber\\
&=\int_{R_{k,j}}F(x)\,dx
=|R_{k,j}|F_{R_{k,j}}
=\int_{R_{k,j}}G_k(x)\,dx.
\end{align}
It remains to verify strong convergence. Since the level-$k$
stopping cubes have pairwise disjoint interiors, the definition of $E_k$
and the first estimate in \eqref{eq-new-dyadic-volume} give
\begin{align}\label{Ek-01}
|E_k|
=\sum_{j\in J_k}|R_{k,j}|
\leq 2^{-ks}\int_Q|F(x)|^s\,dx
\leq 2^{-ks}|Q|
\to 0
\end{align}
as $k\to\infty$.
The definition of $G_k$ also gives
$G_k=F$ almost everywhere on $Q\setminus E_k$. Hence, using
the definition of $G_k$, the pairwise
disjointness of the level-$k$ stopping cubes, H\"older's
inequality, and \eqref{Ek-01},
we obtain, for every $k\in\mathbb{Z}_+$,
\begin{align}
\|F-G_k\|_{L^s(Q)}^s
&=\int_{E_k}|F(x)-G_k(x)|^s\,dx\nonumber\\
&\leq2^{s-1}\left[
\int_{E_k}|F(x)|^s\,dx
+\int_{E_k}|G_k(x)|^s\,dx
\right]\nonumber\\
&=2^{s-1}\left[
\int_{E_k}|F(x)|^s\,dx
+\sum_{j\in J_k}|R_{k,j}|\,|F_{R_{k,j}}|^s
\right]\nonumber\\
&\leq2^{s-1}\left[
\int_{E_k}|F(x)|^s\,dx
+\sum_{j\in J_k}\int_{R_{k,j}}|F(x)|^s\,dx
\right]\nonumber\\
&=2^s\int_{E_k}|F(x)|^s\,dx
\to0
\label{eq-new-dyadic-Gk-convergence}
\end{align}
as $k\to\infty$, where the last convergence follows from the absolute
continuity of the integral.

\textbf{Step 3.} \noindent\textit{In this step, we decompose each
difference $G_{k+1}-G_k$ into mean-zero pieces supported on the level-$k$
stopping cubes and sum these pieces to recover $F-G$.}
Let $G:=G_0$ and, for every $k\in\mathbb{Z}_+$ and $j\in J_k$,
define
\begin{align*}
H_{k,j}:=(G_{k+1}-G_k)\mathbf{1}_{R_{k,j}}.
\end{align*}
For every $k\in\mathbb{Z}_+$, since $E_{k+1}\subset E_k$ and both
$G_k$ and $G_{k+1}$ equal $F$ outside $E_k$, the difference
$G_{k+1}-G_k$ is supported in $E_k$. Using the pairwise disjoint
half-open representatives of the level-$k$ cubes, we obtain, for every
$k\in\mathbb{Z}_+$, almost everywhere on $Q$,
\begin{align}\label{eq-new-dyadic-H-sum}
G_{k+1}-G_k
=(G_{k+1}-G_k)\mathbf{1}_{E_k}
=\sum_{j\in J_k}
(G_{k+1}-G_k)\mathbf{1}_{R_{k,j}}
=\sum_{j\in J_k}H_{k,j}.
\end{align}
The convergence in $L^s(Q)$ is immediate when $J_k$ is finite.
Otherwise, write $J_k=\{j_m:m\in\mathbb{N}\}$.
Using the pairwise
disjointness of the level-$k$ cubes
and the absolute continuity of the integral, we find that,
for every $k\in\mathbb{Z}_+$,
\begin{align*}
\left\|
G_{k+1}-G_k-\sum_{m=1}^{N}H_{k,j_m}
\right\|_{L^s(Q)}^s
=\int_{E_k\setminus\bigcup_{m=1}^{N}R_{k,j_m}}
|G_{k+1}(x)-G_k(x)|^s\,dx
\to0
\end{align*}
as $N\to\infty$, by the absolute continuity of the integral.
In particular, for every $k\in\mathbb{Z}_+$ and $j\in J_k$,
$\operatorname{supp}H_{k,j}\subset R_{k,j}$. For the same $k$ and $j$,
\eqref{eq-new-dyadic-two-level-integrals} implies that
\begin{align*}
\int_{R_{k,j}}H_{k,j}(x)\,dx=0,
\end{align*}
while \eqref{eq-new-dyadic-Gk-bound} implies
\begin{align*}
\|H_{k,j}\|_{L^\infty(R_{k,j})}
\leq\|G_{k+1}\|_{L^\infty(Q)}
+\|G_k\|_{L^\infty(Q)}
\leq C_{(n,s)}2^k.
\end{align*}
Thus \eqref{eq-new-dyadic-good} and
\eqref{eq-new-dyadic-bad} hold. Finally, for every
$N\in\mathbb{Z}_+$, summing \eqref{eq-new-dyadic-H-sum} from $k=0$ to
$k=N$ implies
\begin{align*}
G+\sum_{k=0}^{N}\sum_{j\in J_k}H_{k,j}
=G_0+\sum_{k=0}^{N}(G_{k+1}-G_k)
=G_{N+1}.
\end{align*}
Letting $N\to\infty$ and using
\eqref{eq-new-dyadic-Gk-convergence}, we obtain
\eqref{eq-new-dyadic-expansion}.
This completes the proof of Lemma \ref{lem-new-dyadic-decomposition}.
\end{proof}

The following lemma assembles the
above single-cube decompositions across
an admissible family and shows that every
$s$-block with $s\in(n,\infty)$ admits an $\infty$-block expansion that
converges absolutely in $L^1(Q_0)$ and whose partial sums through the first
$M$ terms converge in $L^s(Q_0)$ as $M\to\infty$. This
decomposition will
be used below to prove
$\mathcal{H}_s(Q_0)=\mathcal{H}_\infty(Q_0)$ and,
in the next subsection, to provide a new proof of
$\|f\|_{\mathcal{B}_q(Q_0)}
\leq C_{(n,q)}\|f\|_{\mathcal{B}(Q_0)}$ with
$q\in(1,\frac{n}{n-1})$.

\begin{lemma}\label{lem-new-block-decomposition}
Let $s\in(n,\infty)$. There exists a positive constant $C_{(n,s)}$ such
that every Bourgain--Brezis--Mironescu $s$-block
$b$ admits a representation
\begin{align}\label{eq-new-block-expansion}
\sum_{m=1}^{M}\lambda_mB_m\to b
\quad\hbox{in both }L^1(Q_0)\hbox{ and }L^s(Q_0)
\end{align}
as $M\to\infty$,
where every $B_m$ is a Bourgain--Brezis--Mironescu $\infty$-block
and $\{\lambda_m\}_{m=1}^{\infty}\subset \mathbb{K}$ satisfies
\begin{align}\label{eq-new-block-coefficients}
\sum_{m=1}^{\infty}|\lambda_m|\leq C_{(n,s)}.
\end{align}
Moreover, the series
$\sum_{m=1}^{\infty}\lambda_mB_m$ converges absolutely in $L^1(Q_0)$,
that is,
$\sum_{m=1}^{\infty}
\|\lambda_mB_m\|_{L^1(Q_0)}<\infty.$
When $b$ is real-valued, all the blocks and coefficients in the
representation may be chosen real-valued.
\end{lemma}

\begin{proof}
We divide the proof into three steps.

\textbf{Step 1.} \noindent\textit{In this step, we apply the dyadic
decomposition inside each initial cube.}
Let $\varepsilon\in(0,1]$,
$\{Q_j\}_{j\in J}\in\mathscr{F}_{\varepsilon}(Q_0)$, and
$b=\sum_{j\in J}b_j$ be as in Definition \ref{def-new-blocks}. Using
the cancellation and the size condition of $b_j$, we find that, for every
$j\in J$,
\begin{align*}
\int_{Q_j}\varepsilon b_j(x)\,dx=0
\quad \text{and}\quad
\|\varepsilon b_j\|_{L^s(Q_j)}
\leq|Q_j|^{\frac{1}{s}}.
\end{align*}
Thus Lemma \ref{lem-new-dyadic-decomposition} applies to
$\varepsilon b_j$ for every $j\in J$. Denote the corresponding good
part by $G_j$, the level-$k$ stopping cubes by
$\{R_{k,\ell}^{(j)}\}_{\ell\in J_k^{(j)}}$, and the associated bad
functions by
$\{H_{k,\ell}^{(j)}\}_{\ell\in J_k^{(j)}}$. Let
\begin{align*}
G:=\sum_{j\in J}G_j
\quad\text{and}\quad
\mathcal{R}_k
:=\{R_{k,\ell}^{(j)}:
j\in J,\ \ell\in J_k^{(j)}\}.
\end{align*}
Since the initial cubes have pairwise disjoint interiors and, inside
each initial cube, the level-$k$ stopping cubes are disjoint in their
half-open representatives, every $R\in\mathcal{R}_k$ has a unique
representation $R=R_{k,\ell}^{(j)}$. Define
\begin{align*}
H_R:=H_{k,\ell}^{(j)}
\quad\text{whenever}\quad
R=R_{k,\ell}^{(j)}.
\end{align*}
In particular, all cubes in $\mathcal{R}_k$ have pairwise disjoint
interiors. The admissibility bound makes $J$ finite. Adding the finitely
many decompositions and then multiplying by $\varepsilon^{-1}$, we
obtain
\begin{align*}
b
&=\lim_{N\to\infty}\left[
\varepsilon^{-1}G
+\sum_{k=0}^{N}\sum_{R\in\mathcal{R}_k}
\varepsilon^{-1}H_R\right]
\quad\hbox{in }L^s(Q_0).
\end{align*}
In addition, by \eqref{eq-new-dyadic-volume}
and the admissibility condition
$\#J\leq\varepsilon^{1-n}$, we find that
\begin{align}
\sum_{R\in\mathcal{R}_k}|R|
&\leq2^{-ks}\sum_{j\in J}
\int_{Q_j}|\varepsilon b_j(x)|^s\,dx\notag\\
&\leq2^{-ks}(\#J)\varepsilon^n
\leq2^{-ks}\varepsilon^{1-n}\varepsilon^n
=\varepsilon2^{-ks}.\label{eq-new-block-volume-detail}
\end{align}
Moreover, again using \eqref{eq-new-dyadic-volume}, every
$R\in\mathcal{R}_k$ satisfies
\begin{align}\label{eq-new-block-length-detail}
\ell(R)\leq\varepsilon2^{-k\frac{s}{n}}.
\end{align}

\textbf{Step 2.} \noindent\textit{In this step, we enumerate the
$\infty$-block terms and prove convergence of their ordinary partial
sums in $L^s(Q_0)$.}
For every initial cube $Q_j$, the good part $G_j$ has integral
zero and, by \eqref{eq-new-dyadic-good}, satisfies
\begin{align*}
\left\|\varepsilon^{-1}G_j\right\|_{L^\infty(Q_j)}
\leq C_{(n,s)}\varepsilon^{-1}.
\end{align*}
Consequently, let
\begin{align}\label{l1}
B_0:=\frac{1}{C_{(n,s)}}\varepsilon^{-1}G
\quad\text{and}\quad
\lambda_0:=C_{(n,s)}.
\end{align}
Then $B_0$ is a Bourgain--Brezis--Mironescu $\infty$-block at scale
$\varepsilon$ and $\lambda_0B_0=\varepsilon^{-1}G$.

Next, fix $k\in\mathbb{Z}_+$. Since each cube in $\mathcal{R}_k$ is a
dyadic descendant of an initial cube of edge length $\varepsilon$, every
edge length occurring at this level has the form
$
r=2^{-j}\varepsilon
$
for some $j\in\mathbb{Z}_+$. For every occurring $r$, let
\begin{align*}
\mathcal{R}_{k,r}
:=\{R\in\mathcal{R}_k:\ell(R)=r\}
\quad\text{and}\quad
N_{k,r}:=\#\mathcal{R}_{k,r}.
\end{align*}
For each fixed $r=2^{-j}\varepsilon$, $N_{k,r}$ is finite.
Indeed, only
$2^{jn}$ dyadic descendants of edge length $r$
occur in each initial cube,
and the initial family is finite. The collection of occurring edge
lengths, and hence the collection of $\infty$-blocks at level $k$, is at
most countable.
Partition $\mathcal{R}_{k,r}$ into subcollections, each containing at
most $\lfloor r^{1-n}\rfloor$ cubes, and denote them by
$\{\mathcal{P}_{k,r,i}\}_{i=1}^{K_{k,r}}$. By
\eqref{eq-new-dyadic-bad}, every $H_R$ is supported in $R$, has integral
zero, and satisfies
$
\|\varepsilon^{-1}H_R\|_{L^\infty(R)}
\leq C_{(n,s)}2^k\varepsilon^{-1}.
$
For every such subcollection, let
\begin{align}\label{bl}
B_{k,r,i}
:=\frac{\varepsilon}{C_{(n,s)}2^kr}
\sum_{R\in\mathcal{P}_{k,r,i}}\varepsilon^{-1}H_R
\quad\text{and}\quad
\lambda_{k,r,i}
:=C_{(n,s)}\frac{2^kr}{\varepsilon}.
\end{align}
Thus every $B_{k,r,i}$ is a Bourgain--Brezis--Mironescu $\infty$-block
and
\begin{align*}
t_{k,r,i}
:=\lambda_{k,r,i}B_{k,r,i}
=\sum_{R\in\mathcal{P}_{k,r,i}}\varepsilon^{-1}H_R.
\end{align*}

For each $k$, enumerate the terms $t_{k,r,i}$ as
$\{t_{k,l}\}_{l\in\mathbb{N}}$, adjoining zero terms if necessary and
regarding each such term as $0\cdot B_0$.
Since the terms at each fixed level have pairwise disjoint supports,
we have
\begin{align*}
\sum_{l=1}^{\infty}t_{k,l}
=h_k:=\sum_{R\in\mathcal{R}_k}\varepsilon^{-1}H_R,
\end{align*}
here the convergence to $h_k$ is in $L^s(Q_0)$ and is
independent of the order in which the terms $t_{k,l}$ are summed. Indeed,
\begin{align*}
\|h_k\|_{L^s(Q_0)}^s
=\sum_{R\in\mathcal{R}_k}
\|\varepsilon^{-1}H_R\|_{L^s(R)}^s
\leq C_{(n,s)}2^{ks}\varepsilon^{-s}
\sum_{R\in\mathcal{R}_k}|R|
\leq C_{(n,s)}\varepsilon^{1-s}.
\end{align*}
Consequently,
\begin{align*}
b=\lambda_0B_0+\sum_{k=0}^{\infty}h_k
\quad\hbox{in }L^s(Q_0).
\end{align*}

We now flatten this iterated expansion into a single series whose
partial sums through the first $M$ terms converge in $L^s(Q_0)$ as
$M\to\infty$. Choose an increasing sequence of integers
$M_N\to\infty$ as $N\to\infty$, with
$M_N\geq N+1$, such that
$
\|h_k-\sum_{l=1}^{M_N}t_{k,l}\|_{L^s(Q_0)}
<2^{-N-k}
$ whenever $0\leq k\leq N$, and let
$
S_N:=\lambda_0B_0+
\sum_{k=0}^{N}\sum_{l=1}^{M_N}t_{k,l}.
$
Then
\begin{align*}
\left\|S_N-
\left(\lambda_0B_0+\sum_{k=0}^{N}h_k\right)
\right\|_{L^s(Q_0)}<2^{1-N},
\end{align*}
so $S_N\to b$ in $L^s(Q_0)$ as $N\to\infty$.
For $N\in\mathbb{Z}_+$, let
\begin{align*}
\Gamma_N:=\{(k,l):0\leq k\leq N,\ 1\leq l\leq M_N\}.
\end{align*}
Enumerate $\lambda_0B_0$ first, then the terms indexed by $\Gamma_0$,
and thereafter those indexed successively by the finite shells
$\Gamma_N\setminus\Gamma_{N-1}$, $N\geq1$, in any fixed order within
each shell. Since $M_N\to\infty$ as $N\to\infty$, the shells exhaust
$\mathbb{Z}_+\times\mathbb{N}$, so every indexed term occurs exactly
once. Relabel the resulting sequence as
$\{\lambda_mB_m\}_{m\in\mathbb{N}}$. The partial sum at the end of the
$N$-th shell is $S_N$. Hence these shell-end partial sums converge to
$b$ in $L^s(Q_0)$ as $N\to\infty$.

To prove the $L^s(Q_0)$ convergence in
\eqref{eq-new-block-expansion}, it remains only to control the partial
sums inside each shell, since the shell-end sums $S_N$ already converge
to $b$. Fix $N\in\mathbb{N}$. Write the terms in the $N$-th shell, in
their chosen order, as $u_{N,1},\ldots,u_{N,L_N}$, and,
for any integer $r\in[0, L_N]$, let
$
P_{N,r}:=S_{N-1}+\sum_{a=1}^{r}u_{N,a}.
$
Thus $P_{N,r}$ runs through all partial sums in the $N$-th shell.
Since $\sum_{k=0}^{\infty}h_k$ converges in $L^s(Q_0)$, we have
$\|h_N\|_{L^s(Q_0)}\to0$ as $N\to\infty$.
By this, the triangle inequality, the pairwise
disjointness of the terms within each row, and the choice of
$M_{N-1}$,
\begin{align*}
\sup_{0\leq r\leq L_N}
\|P_{N,r}-b\|_{L^s(Q_0)}
&\leq
\sup_{0\leq r\leq L_N}
\|P_{N,r}-S_{N-1}\|_{L^s(Q_0)}
+\|S_{N-1}-b\|_{L^s(Q_0)}\\
&\leq
\sum_{k=0}^{N-1}
\left\|h_k-\sum_{l=1}^{M_{N-1}}t_{k,l}\right\|_{L^s(Q_0)}
+\|h_N\|_{L^s(Q_0)}+\|S_{N-1}-b\|_{L^s(Q_0)}\\
&<2^{2-N}+\|h_N\|_{L^s(Q_0)}
+\|S_{N-1}-b\|_{L^s(Q_0)}\to0
\end{align*}
as $N\to\infty$.
Hence all partial sums of the relabeled series converge to $b$ in
$L^s(Q_0)$, proving the $L^s(Q_0)$ convergence in
\eqref{eq-new-block-expansion}.

\textbf{Step 3.} \noindent\textit{In this step, we
establish the summability of the coefficients and the absolute
convergence of the resulting series in $L^1(Q_0)$.}
For every occurring $r\in(0,1]$, we have $r^{1-n}\geq1$ and hence
$\lfloor r^{1-n}\rfloor\geq \frac{r^{1-n}}{2}$.
In what follows, for any $a\in\mathbb{R}$,
$\lceil a\rceil$ denotes the least integer not
less than $a$. Therefore
\begin{align}\label{eq-new-number-groups-detail}
K_{k,r}
&\leq\left\lceil\frac{N_{k,r}}
{\lfloor r^{1-n}\rfloor}\right\rceil
\leq1+\frac{N_{k,r}}{\lfloor r^{1-n}\rfloor}
\leq1+2N_{k,r}r^{n-1}.
\end{align}
Using \eqref{eq-new-number-groups-detail}, we estimate the sum of all
coefficients at level $k$ as follows:
\begin{align}\label{k1}
\sum_rK_{k,r}C_{(n,s)}\frac{2^kr}{\varepsilon}
\leq C_{(n,s)}\frac{2^k}{\varepsilon}
\sum_r r
+C_{(n,s)}\frac{2^k}{\varepsilon}
\sum_rN_{k,r}r^n.
\end{align}
By \eqref{eq-new-block-length-detail}, all occurring dyadic values of
$r$ are no larger than $\varepsilon2^{-k\frac{s}{n}}$. Using the
geometric-series formula, we obtain
\begin{align}\label{k2}
\sum_r r
\leq\sum_{\{j:\,2^{-j}\varepsilon\leq\varepsilon2^{-k\frac{s}{n}}\}}
2^{-j}\varepsilon
\leq2\varepsilon2^{-k\frac{s}{n}}.
\end{align}
Moreover, from \eqref{eq-new-block-volume-detail}, we obtain
\begin{align}\label{k3}
\sum_rN_{k,r}r^n
&=\sum_{R\in\mathcal{R}_k}|R|
\leq\varepsilon2^{-ks}.
\end{align}
Combining \eqref{k1}, \eqref{k2}, and \eqref{k3}, we infer that
\begin{align}\label{eq-new-level-final-detail}
\sum_rK_{k,r}C_{(n,s)}\frac{2^kr}{\varepsilon}
&\leq C_{(n,s)}\frac{2^k}{\varepsilon}
\varepsilon2^{-k\frac{s}{n}}
+C_{(n,s)}\frac{2^k}{\varepsilon}
\varepsilon2^{-ks}\nonumber\\
&=C_{(n,s)}
\left[2^{-k(\frac{s}{n}-1)}+2^{-k(s-1)}\right].
\end{align}
The enumeration in Step 2 exhausts $\lambda_0B_0$ and all terms
$\lambda_{k,r,i}B_{k,r,i}$; the adjoined zero terms contribute nothing.
From \eqref{l1}, \eqref{bl},
\eqref{eq-new-level-final-detail}, the preceding geometric-series
estimates, and $s\in(n,\infty)$, we infer that
\begin{align*}
\sum_{m=1}^{\infty}|\lambda_m|
&=|\lambda_0|+
\sum_{k=0}^{\infty}\sum_r\sum_{i=1}^{K_{k,r}}
|\lambda_{k,r,i}|\\
&=C_{(n,s)}+
\sum_{k=0}^{\infty}\sum_r
K_{k,r}C_{(n,s)}\frac{2^kr}{\varepsilon}\\
&\leq C_{(n,s)}+
C_{(n,s)}\sum_{k=0}^{\infty}
\left[2^{-k(\frac{s}{n}-1)}+2^{-k(s-1)}\right]\\
&\leq C_{(n,s)},
\end{align*}
which proves \eqref{eq-new-block-coefficients}. By
\eqref{eq-new-block-L1-bound},
\begin{align*}
\sum_{m=1}^{\infty}
\|\lambda_mB_m\|_{L^1(Q_0)}
\leq\sum_{m=1}^{\infty}|\lambda_m|<\infty.
\end{align*}

Thus the series converges absolutely in $L^1(Q_0)$. Moreover, by the
$L^s(Q_0)$ convergence of its partial sums to $b$ proved in Step 2
and H\"older's inequality,

we have
\begin{align*}
\left\|\sum_{m=1}^{N}\lambda_mB_m-b\right\|_{L^1(Q_0)}
\leq |Q_0|^{1-\frac{1}{s}}
\left\|\sum_{m=1}^{N}\lambda_mB_m-b\right\|_{L^s(Q_0)}
\to0
\end{align*}
as $N\to\infty$.
Hence the series converges to $b$ in $L^1(Q_0)$, and
this convergence is absolute.
Finally, if $b$ is real-valued, the construction in Lemma
\ref{lem-new-dyadic-decomposition} preserves real values, while all
normalizing coefficients in \eqref{l1} and \eqref{bl} are real.
Thus all $B_m$ and $\lambda_m$ may be chosen real-valued. This
completes the proof of Lemma \ref{lem-new-block-decomposition}.
\end{proof}

We now prove Theorem \ref{the-new-Hs-self}.

\begin{proof}[Proof of Theorem \ref{the-new-Hs-self}]
We first prove the straightforward inclusion
$\mathcal{H}_\infty(Q_0)\subset \mathcal{H}_s(Q_0)$,
which holds for every
$s\in(1,\infty)$. If
$B=\sum_{j\in J}B_j$ is an $\infty$-block at scale $\varepsilon$,
then, for every $j\in J$,
\begin{align*}
\|B_j\|_{L^s(Q_j)}
\leq |Q_j|^{\frac{1}{s}}\|B_j\|_{L^\infty(Q_j)}
\leq\varepsilon^{\frac{n}{s}-1}.
\end{align*}
Thus every $\infty$-block is an $s$-block. Hence every
$\infty$-block representation is also an $s$-block representation
of the same $L^1(Q_0)$ function,
$\mathcal{H}_\infty(Q_0)\subset\mathcal{H}_s(Q_0)$,
and, for every $g\in\mathcal{H}_\infty(Q_0)$,
\begin{align}\label{eq-new-Hinfty-into-Hs}
\|g\|_{\mathcal{H}_s(Q_0)}
\leq\|g\|_{\mathcal{H}_\infty(Q_0)}.
\end{align}

We now prove the reverse inclusion when $s\in(n,\infty)$. Let
$h\in\mathcal{H}_s(Q_0)$ and choose an $s$-block representation
\begin{align*}
h=\sum_{k=1}^{\infty}\mu_kb_k
\quad\text{with }
\sum_{k=1}^{\infty}|\mu_k|<\infty.
\end{align*}
By Definition \ref{def-new-Hs},
\begin{align*}
\sum_{k=1}^{K}\mu_kb_k\to h
\quad\text{in }L^1(Q_0)
\end{align*}
as $K\to\infty$.
For every $k$, Lemma
\ref{lem-new-block-decomposition} provides
$\infty$-blocks
$\{B_{k,m}\}_{m\in\mathbb{N}}$ and coefficients
$\{\lambda_{k,m}\}_{m\in\mathbb{N}}$ such that
\begin{align*}
b_k=\sum_{m=1}^{\infty}\lambda_{k,m}B_{k,m}
\quad\hbox{in }L^1(Q_0)
\quad\text{and }
\sum_{m=1}^{\infty}|\lambda_{k,m}|\leq C_{(n,s)}.
\end{align*}
Consequently,
\begin{align*}
\sum_{k=1}^{\infty}\sum_{m=1}^{\infty}
|\mu_k\lambda_{k,m}|
\leq C_{(n,s)}\sum_{k=1}^{\infty}|\mu_k|<\infty.
\end{align*}
The one-block representation implies
$\|B_{k,m}\|_{\mathcal{H}_\infty(Q_0)}\leq1$, while
\eqref{eq-new-block-L1-bound} implies
$\|B_{k,m}\|_{L^1(Q_0)}\leq1$. Hence the double series
$
\sum_{k=1}^{\infty}\sum_{m=1}^{\infty}
\mu_k\lambda_{k,m}B_{k,m}
$
converges absolutely, and hence unconditionally under any enumeration,
both in $\mathcal{H}_\infty(Q_0)$ and in $L^1(Q_0)$. Let $g$ denote
its sum. For every
$K\in\mathbb{N}$, the $L^1(Q_0)$ identities supplied
by Lemma \ref{lem-new-block-decomposition} imply
\begin{align*}
\sum_{k=1}^{K}\sum_{m=1}^{\infty}
\mu_k\lambda_{k,m}B_{k,m}
&=\sum_{k=1}^{K}\mu_kb_k
\quad\hbox{in }L^1(Q_0).
\end{align*}
As $K\to\infty$, the left-hand side converges to $g$ by the absolute
convergence of the double series, whereas the right-hand side converges
to $h$. Thus $g=h$ in $L^1(Q_0)$.
Consequently, $h$ belongs to $\mathcal{H}_\infty(Q_0)$, and
\begin{align*}
\|h\|_{\mathcal{H}_\infty(Q_0)}
\leq C_{(n,s)}\sum_{k=1}^{\infty}|\mu_k|.
\end{align*}
Taking the infimum over the $s$-block representations of $h$ proves the
second inequality in \eqref{eq-new-Hs-self-norms}.
This, together with
\eqref{eq-new-Hinfty-into-Hs},
implies the forward implication. The
case $s=\infty$ is immediate.

It remains to prove sharpness. Suppose, to the contrary, that the two
subspaces of $L_0^1(Q_0)$ coincide for some $s\in(1,n]$, and let
$q:=s'$.
Then
$
q\in[\frac{n}{n-1},\infty).
$
By \eqref{eq-new-Hinfty-into-Hs},
we find that the natural inclusion
\begin{align*}
J_s:
\begin{cases}
\displaystyle
\mathcal{H}_\infty(Q_0)\longrightarrow\mathcal{H}_s(Q_0)\\
\displaystyle
h\longmapsto h\in L_0^1(Q_0)
\end{cases}
\end{align*}
is then a bounded bijection between Banach spaces. The open mapping
theorem therefore implies that $J_s^{-1}$ is bounded as well.
By Theorem \ref{the-new-self}, choose
$f\in\mathcal{B}(Q_0)\setminus\mathcal{B}_q(Q_0)$. By Lemma
\ref{lem-new-atomic-core-properties}(ii), the canonical endpoint action
\begin{align*}
\mathcal{L}_f^\infty(h):=\langle f,h\rangle_\infty,
\quad h\in\mathcal{H}_\infty(Q_0)
\end{align*}
is continuous. Define
\begin{align*}
\mathcal{L}
:=\mathcal{L}_f^\infty\circ J_s^{-1}
\in[\mathcal{H}_s(Q_0)]^*.
\end{align*}
Theorem \ref{the-new-Hs-duality} would therefore
produce
$\widetilde f\in\mathcal{B}_q(Q_0)$ representing $\mathcal{L}$.
Choose representatives of $f$ and $\widetilde f$. For every
$g\in L_0^\infty(Q_0)$, one has $J_s^{-1}g=g$, and hence
\begin{align}\label{q0fg}
\int_{Q_0}f(x)g(x)\,dx
=\mathcal{L}_f^\infty(g)
=\mathcal{L}(g)
=\int_{Q_0}\widetilde f(x)g(x)\,dx.
\end{align}

Thus, for every measurable $E\subset Q_0$, let
$g=\mathbf{1}_E-\frac{|E|}{|Q_0|}$ in
\eqref{q0fg}. Then we have
\begin{align*}
\int_E\left\{\left[f(x)-\widetilde f(x)\right]
-(f-\widetilde f)_{Q_0}\right\}\,dx=0.
\end{align*}
Hence $f-\widetilde f$ is constant almost everywhere,

so $f$ and $\widetilde f$ are the same element of
$\mathcal{B}(Q_0)$. Since
$\mathcal{B}_q(Q_0)\subset\mathcal{B}(Q_0)$ and membership in
$\mathcal{B}_q(Q_0)$ is unchanged by adding a constant, this contradicts
$f\notin\mathcal{B}_q(Q_0)$. Hence the two subspaces of
$L_0^1(Q_0)$ cannot coincide for any $s\in(1,n]$.
This completes the proof of Theorem
\ref{the-new-Hs-self}.
\end{proof}

Theorem \ref{the-new-Hs-self} determines exactly when the natural
inclusion $J_s:\mathcal{H}_\infty(Q_0)\to\mathcal{H}_s(Q_0)$ is onto.
We next make two further properties of this same inclusion explicit.
For every $s\in(1,\infty)$, the map $J_s$ is contractive.
Moreover, under the canonical block pairings, it is compatible with the
natural inclusion
$\iota_s:\mathcal{B}_{s'}(Q_0)\to\mathcal{B}(Q_0)$, in the precise
sense stated below.

\begin{proposition}
\label{prop-new-Js-contraction}
For every $s\in(1,\infty)$, the natural inclusion
\begin{align*}
J_s:
\begin{cases}
\displaystyle
\mathcal{H}_\infty(Q_0)\longrightarrow\mathcal{H}_s(Q_0)\\
\displaystyle
h\longmapsto h\in L_0^1(Q_0)
\end{cases}
\end{align*}
is a contraction. Moreover, the canonical inclusion
\begin{align*}
\iota_s:\mathcal{B}_{s'}(Q_0)\longrightarrow\mathcal{B}(Q_0)
\end{align*}
is compatible with $J_s$ under the canonical block pairings from
Lemma \ref{lem-new-atomic-core-properties}(ii):
\begin{align}\label{eq-new-Js-ambient-compatibility}
\langle f,J_sh\rangle_s
=\langle\iota_sf,h\rangle_\infty
\end{align}
for every $f\in\mathcal{B}_{s'}(Q_0)$ and
$h\in\mathcal{H}_\infty(Q_0)$.
In particular, $J_s$ is the identity on $L_0^\infty(Q_0)$.
If $n\in\mathbb{N}\cap[2,\infty)$, then $J_s$ is onto if and only if
$s\in(n,\infty)$. We let $J_\infty$
denote the identity on $\mathcal{H}_\infty(Q_0)$.
\end{proposition}

\begin{proof}
Every $\infty$-block is an $s$-block by the estimate used to prove
\eqref{eq-new-Hinfty-into-Hs}. Hence the same $\infty$-block
representation is an $s$-block representation of the same
$L^1(Q_0)$ function. Taking the infimum over all such representations
shows that $J_s$ is a contraction.

The inequality
$
\|f\|_{\mathcal{B}(Q_0)}
\leq\|f\|_{\mathcal{B}_{s'}(Q_0)}
$ for every $f\in \mathcal{B}_{s'}(Q_0)$
shows that $\iota_s$ is a contraction. Let
$h\in\mathcal{H}_\infty(Q_0)$ and fix an $\infty$-block
representation
\begin{align*}
h=\sum_{m=1}^{\infty}\lambda_mB_m
\quad\hbox{in }L^1(Q_0),
\quad\text{and}\quad
\sum_{m=1}^{\infty}|\lambda_m|<\infty.
\end{align*}
For every $f\in\mathcal{B}_{s'}(Q_0)$,
the canonical block pairings satisfy
\begin{align*}
\langle\iota_sf,h\rangle_\infty
=\sum_{m=1}^{\infty}\lambda_m
\int_{Q_0}f(x)B_m(x)\,dx
=\langle f,J_sh\rangle_s.
\end{align*}
This proves \eqref{eq-new-Js-ambient-compatibility}. The remaining
assertions follow from Theorem \ref{the-new-Hs-self} and the
scale-one inclusion of $L_0^\infty(Q_0)$.
This completes the proof of Proposition
\ref{prop-new-Js-contraction}.
\end{proof}

Combining the two sharp
self-improvement results and the finite-exponent duality
in Theorems \ref{the-new-self},
\ref{the-new-Hs-duality}, and \ref{the-new-Hs-self},
we obtain the following endpoint duality result.

\begin{corollary}
\label{cor-new-Hinfty-endpoint-duality}
Let $n\in\mathbb{N}\cap[2,\infty)$.
With respect to the canonical block pairing
$\langle f,h\rangle_\infty$ from Lemma
\ref{lem-new-atomic-core-properties}(ii), one has
\begin{align}\label{eq-new-Hinfty-known-predual}
\left[\mathcal{H}_\infty(Q_0)\right]^*
=\mathcal{B}(Q_0)
\end{align}
with equivalent norms. Consequently,
$\mathcal{H}_\infty(Q_0)$ is an atomic predual of
$\mathcal{B}(Q_0)$.
\end{corollary}

\begin{proof}
Fix any $s_0\in(n,\infty)$. By Theorems
\ref{the-new-Hs-self},
\ref{the-new-Hs-duality}, and \ref{the-new-self}, we have
\begin{align*}
\left[\mathcal{H}_\infty(Q_0)\right]^*
=
\left[\mathcal{H}_{s_0}(Q_0)\right]^*
=\mathcal{B}_{s_0'}(Q_0)
=\mathcal{B}(Q_0).
\end{align*}
The resulting action is the canonical one: on the dense subspace
$L_0^\infty(Q_0)$ it is the integral pairing, and hence it agrees
with the canonical block pairing on all of
$\mathcal{H}_\infty(Q_0)$. This completes the proof of
\eqref{eq-new-Hinfty-known-predual}
and hence Corollary \ref{cor-new-Hinfty-endpoint-duality}.
\end{proof}

\begin{remark}\label{rem-new-endpoint-predual-known}
Corollary \ref{cor-new-Hinfty-endpoint-duality} recovers, in the present
notation, the atomic predual identification in
\cite[Theorem 1]{dgp2020}.
The distinction is in the route to the result:
\cite[Theorem 1]{dgp2020} obtains
the predual of $\mathcal{B}(Q_0)$
from a general atomic predual construction, whereas the
endpoint duality here follows by combining the finite-exponent duality
with the sharp self-improvements of both the $\mathcal{B}_q(Q_0)$ and
Hardy-type scales.
\end{remark}

\subsection{An Alternative Block Proof of
$\mathcal{B}_q(Q_0)=\mathcal{B}(Q_0)$\label{3.3}}

We finish this section by
using the block decomposition in
Lemma \ref{lem-new-block-decomposition} to recover the result
$\mathcal{B}_q(Q_0)
=\mathcal{B}(Q_0)$
for $q\in(1,\frac{n}{n-1})$ without using the sparse argument,
a John--Nirenberg-type inequality, or duality.
Indeed, combining the endpoint estimate
in Proposition \ref{prop-new-block-functional}
with Lemmas \ref{lem-new-block-pairing} and
\ref{lem-new-block-decomposition},
we obtain the alternative proof.

\begin{proposition}\label{prop-new-block-alternative}
Let $n\in\mathbb{N}\cap[2,\infty)$ and
$q\in(1,\frac{n}{n-1})$. Then
$\mathcal{B}_q(Q_0)=\mathcal{B}(Q_0)$
with equivalent norms. More precisely, there exists a positive constant
$C_{(n,q)}$ such that, for every
$f\in L^1(Q_0)$,
\begin{align}\label{eq-new-block-alternative-norms}
\|f\|_{\mathcal{B}(Q_0)}
\leq\|f\|_{\mathcal{B}_q(Q_0)}
\leq C_{(n,q)}\|f\|_{\mathcal{B}(Q_0)}.
\end{align}
\end{proposition}

\begin{proof}
By H\"older's inequality and
the definition in \eqref{eq-new-Bq},
for every
$f\in L^1(Q_0)$, we have
$
\|f\|_{\mathcal{B}(Q_0)}
\leq\|f\|_{\mathcal{B}_q(Q_0)}.
$
Thus $\mathcal{B}_q(Q_0)\subset\mathcal{B}(Q_0)$ continuously. It remains
to prove the reverse inclusion. Since
$q\in(1,\frac{n}{n-1})$, its conjugate exponent satisfies
$
q'=\frac{q}{q-1}\in(n,\infty).
$
We divide the proof into two steps.

\textbf{Step 1.} \noindent\textit{In this step, we prove the second
inequality in \eqref{eq-new-block-alternative-norms} for bounded
functions.}
Let $f\in L^\infty(Q_0)$ be a representative of an element of
$\mathcal{B}(Q_0)$. Fix $\varepsilon\in(0,1]$ and
$\{Q_j\}_{j\in J}\in\mathscr{F}_{\varepsilon}(Q_0)$.
Since $f\in L^\infty(Q_0)\subset L^q(Q_0)$, Lemma
\ref{lem-new-block-pairing} implies that there exists a
$q'$-block $b$ satisfying
\eqref{eq-new-norming-block-pairing}. Applying Lemma
\ref{lem-new-block-decomposition} with $s=q'$, we find that there exist
$\infty$-blocks
$\{B_m\}_{m\in\mathbb{N}}$ and
$\{\lambda_m\}_{m\in\mathbb{N}}\subset \mathbb{K}$
with $\sum_{m=1}^{\infty}|\lambda_m|
\leq C_{(n,q)}$ such that
\begin{align*}
b=\sum_{m=1}^{\infty}\lambda_mB_m
\quad\hbox{in both }L^1(Q_0)\hbox{ and }L^{q'}(Q_0).
\end{align*}
By the $L^{q'}(Q_0)$ convergence above and H\"older's inequality, we
find that
\begin{align*}
\left|\int_{Q_0}f(x)
\left[\sum_{m=1}^{N}\lambda_mB_m(x)-b(x)\right]\,dx\right|
\leq\|f\|_{L^q(Q_0)}
\left\|\sum_{m=1}^{N}\lambda_mB_m-b\right\|_{L^{q'}(Q_0)}
\to0
\end{align*}
as $N\to\infty$. Using this limit,
\eqref{eq-new-norming-block-pairing}, Proposition
\ref{prop-new-block-functional} with $s=\infty$, and
the coefficient estimate above, we conclude that
\begin{align*}
\frac{1}{2}\varepsilon^{n-1}\sum_{j\in J}
\operatorname{osc}_q(f;Q_j)
&=\left|\int_{Q_0}f(x)b(x)\,dx\right|\\
&=\lim_{N\to\infty}
\left|\sum_{m=1}^{N}\lambda_m
\int_{Q_0}f(x)B_m(x)\,dx\right|\\
&\leq\sum_{m=1}^{\infty}|\lambda_m|
\left|\int_{Q_0}f(x)B_m(x)\,dx\right|\\
&\leq\|f\|_{\mathcal{B}(Q_0)}
\sum_{m=1}^{\infty}|\lambda_m|
\leq C_{(n,q)}\|f\|_{\mathcal{B}(Q_0)}.
\end{align*}
This further implies that
\begin{align}\label{eq-new-bounded-self-detail}
\|f\|_{\mathcal{B}_q(Q_0)}
\leq C_{(n,q)}\|f\|_{\mathcal{B}(Q_0)}.
\end{align}

\textbf{Step 2.} \noindent\textit{In this step, we show that
the second
inequality in \eqref{eq-new-block-alternative-norms} holds for
any $f\in L^1(Q_0)$.}
Let $f\in L^1(Q_0)$ be a $\mathbb{K}$-valued representative of an
arbitrary element of $\mathcal{B}(Q_0)$. For every
$N\in\mathbb{N}$, define the \emph{radial projection}
onto the closed ball of
radius $N$ by
\begin{align*}
P_N(z)
:=
\begin{cases}
z,& |z|\leq N,\\[1mm]
\displaystyle N\frac{z}{|z|},& |z|>N,
\end{cases}
\qquad z\in\mathbb{K},
\qquad
f_N:=P_N\circ f.
\end{align*}
Then $f_N$ is bounded. By the first inequality in
\eqref{eq-new-pairing-truncation-bounds} with $s=\infty$,
we have $f_N\in\mathcal{B}(Q_0)$. Applying
\eqref{eq-new-bounded-self-detail} to $f_N$ and then using the same
truncation inequality, we obtain
\begin{align}\label{eq-new-truncation-q-detail}
\|f_N\|_{\mathcal{B}_q(Q_0)}
&\leq C_{(n,q)}\|f_N\|_{\mathcal{B}(Q_0)}
\leq2C_{(n,q)}\|f\|_{\mathcal{B}(Q_0)}.
\end{align}
Since $f_N\to f$ almost everywhere on $Q_0$ as $N\to\infty$ and
$|f_N-f|\leq2|f|\in L^1(Q_0)$, the dominated convergence theorem implies
that
\begin{align*}
\lim_{N\to\infty}\|f_N-f\|_{L^1(Q_0)}=0.
\end{align*}
From this, for every fixed cube $Q\subset Q_0$, we deduce that
\begin{align*}
|(f_N)_Q-f_Q|
&\leq\fint_Q|f_N(x)-f(x)|\,dx\to0
\end{align*}
as $N\to\infty$. Thus
$f_N-(f_N)_Q\to f-f_Q$ almost everywhere on $Q$ as $N\to\infty$.
Using this convergence,
Fatou's lemma, and \eqref{eq-new-truncation-q-detail}, for every
$\varepsilon\in(0,1]$ and every
$\{Q_j\}_{j\in J}\in\mathscr{F}_{\varepsilon}(Q_0)$,
we conclude that
\begin{align*}
&\varepsilon^{n-1}\sum_{j\in J}
\left[\fint_{Q_j}|f(x)-f_{Q_j}|^q\,dx\right]^{\frac{1}{q}}\\
&\quad\leq\liminf_{N\to\infty}
\varepsilon^{n-1}\sum_{j\in J}
\left[\fint_{Q_j}|f_N(x)-(f_N)_{Q_j}|^q\,dx\right]^{\frac{1}{q}}\\
&\quad\leq2C_{(n,q)}\|f\|_{\mathcal{B}(Q_0)},
\end{align*}
which further implies that
\begin{align*}
\|f\|_{\mathcal{B}_q(Q_0)}
\leq C_{(n,q)}\|f\|_{\mathcal{B}(Q_0)}.
\end{align*}
This proves
$\mathcal{B}(Q_0)\subset\mathcal{B}_q(Q_0)$
and hence completes the proof of
Proposition \ref{prop-new-block-alternative}.
\end{proof}

\begin{remark}
Proposition \ref{prop-new-block-alternative}
provides an independent proof
of Theorem \ref{the-new-self} for
$q\in(1,\frac{n}{n-1})$. The two proofs use different methods. The sparse
proof in Section \ref{sec-new-self} works directly with local
oscillations and covers the full range
$q\in(0,\frac{n}{n-1})$, including $q\in(0,1]$. In that proof, the
condition $q<\frac{n}{n-1}$ ensures the convergence of the geometric
series obtained by summing the contributions over dyadic scales.
In contrast, the proof of Proposition
\ref{prop-new-block-alternative} first represents the $q$-oscillation
by pairing with a $q'$-block and then decomposes this block into
$\infty$-blocks. This argument applies when $q>1$, and the required
condition $q'>n$ is equivalent to
$q<\frac{n}{n-1}$. Thus the sparse proof
establishes the result in the wider
range, while the block proof explains the same critical exponent
through the explicit relation between $q'$-blocks and
$\infty$-blocks.
\end{remark}

\section{Endpoint Self-Improvement
for Characteristic Functions\label{sec-new-fixed-perimeter}}

We now specialize to characteristic functions.
Indeed, for a characteristic function $f=\mathbf{1}_E$, the oscillation sum
in the definition of $\mathcal{B}(Q_0)$ measures how cubes at one common
scale meet the boundary of $E$.
Moreover, the
defining functionals $\|\cdot\|_{\mathcal{B}(Q_0)}$ and
$\|\cdot\|_{\mathcal{B}_q(Q_0)}$ are uniformly equivalent
precisely for $q\in(0,\frac{n}{n-1}]$. Theorem
\ref{the-new-fixed-perimeter} also determines their sharp dependence on
a perimeter bound and the failure above the critical exponent.

Before stating the main result,
we note that
the $\mathcal{B}_q(Q_0)$ functional is always finite in this setting.
Indeed, since $\mathbf{1}_E(x)\in\{0,1\}$ and
$(\mathbf{1}_E)_{Q_j}\in[0,1]$, we have
$|\mathbf{1}_E(x)-(\mathbf{1}_E)_{Q_j}|\leq1$, and every
$\{Q_j\}_{j\in J}\in\mathscr{F}_{\varepsilon}(Q_0)$ satisfies
\begin{align*}
\varepsilon^{n-1}\sum_{j\in J}
\operatorname{osc}_q(\mathbf{1}_E;Q_j)
\leq\varepsilon^{n-1}\sum_{j\in J}1
=\varepsilon^{n-1}\#J
\leq1.
\end{align*}
Taking the two suprema in
the definition in \eqref{eq-new-Bq}, we obtain
\begin{align}\label{eq-new-characteristic-universal-bound}
\|\mathbf{1}_E\|_{\mathcal{B}_q(Q_0)}\leq1
\end{align}
for every measurable $E\subset Q_0$ and
every $q\in(0,\infty)$.
Thus membership is automatic, and the issue is quantitative: for every
$q\in(0,\infty)$, we determine the sharp dependence on $\Lambda$ among
sets satisfying $P(E,Q_0)\leq\Lambda$.

We first recall the perimeter notation used below; see,
for instance, \cite[Definition 3.35]{AFP2000}.

\begin{definition}\label{def-new-perimeter}
Let $U\subset\mathbb{R}^n$ be open and
$E\subset\mathbb{R}^n$ be measurable. The \emph{perimeter of $E$
relative to $U$} is defined by
\begin{align*}
P(E,U):=|D\mathbf{1}_E|(U)
=\sup\left\{
\int_E\operatorname{div}\varphi(x)\,dx:
\varphi\in C_{\mathrm{c}}^1(U;\mathbb{R}^n),\
\|\varphi\|_{L^\infty(U)}\leq1
\right\}.
\end{align*}
$P(E,U)=\infty$ is allowed. For a cube $Q$, as stipulated above,
$P(E,Q)$ means $P(E,Q^\circ)$.
\end{definition}

We now state the main theorem of this section.

\begin{theorem}\label{the-new-fixed-perimeter}
Let $n\in\mathbb{N}\cap[2,\infty)$ and $q\in(0,\infty)$.
\begin{itemize}
\item[{\rm(i)}]
For every $\Lambda\in(0,1]$,
\begin{align}\label{eq-new-fixed-main}
\sup_{\{E\subset Q_0:\ P(E,Q_0)\leq \Lambda\}}
\|\mathbf{1}_E\|_{\mathcal{B}_q(Q_0)}\sim_{(n,q)}
\begin{cases}
\Lambda,
& q\in\left(0,\frac{n}{n-1}\right],\\[2mm]
\Lambda^{\frac{n}{(n-1)q}},
& q\in\left(\frac{n}{n-1},\infty\right).
\end{cases}
\end{align}
Both powers of $\Lambda$ in \eqref{eq-new-fixed-main} are optimal.

\item[{\rm(ii)}]
The following equivalence holds uniformly over all measurable sets
$E\subset Q_0$:
\begin{align}\label{eq-new-characteristic-self-range}
\|\mathbf{1}_E\|_{\mathcal{B}(Q_0)}
\sim_{(n,q)}\|\mathbf{1}_E\|_{\mathcal{B}_q(Q_0)}
\Longleftrightarrow
q\in\left(0,\frac{n}{n-1}\right].
\end{align}
More precisely, whenever $q$ lies in this range, uniformly over all measurable
sets $E\subset Q_0$,
\begin{align}\label{eq-new-characteristic-self-equivalence}
\|\mathbf{1}_E\|_{\mathcal{B}_q(Q_0)}
\sim_{(n,q)}
\|\mathbf{1}_E\|_{\mathcal{B}(Q_0)}
\sim_{(n,q)}
\min\{1,P(E,Q_0)\}.
\end{align}
\end{itemize}
\end{theorem}

\begin{remark}\label{rem-new-fixed-relation}
\begin{itemize}
\item[{\rm(i)}] For $q\in(0,\frac{n}{n-1})$, the linear upper bound
agrees with Theorem \ref{the-new-self}; at $q=1$, it also agrees with the
perimeter estimate in \cite{abbf2016}. The argument below reaches the
endpoint and identifies the exact supercritical power of $\Lambda$.

\item[{\rm(ii)}]\label{rem-new-fixed-components} The critical exponent
$\frac{n}{n-1}$ is the point at which the small-volume-fraction exponent
$\frac{1}{q}$, identified by Lemma
\ref{lem-new-characteristic-formula} when $q\geq1$, agrees with the
isoperimetric exponent $\frac{n-1}{n}$ in Lemma
\ref{lem-new-relative-isoperimetric}. At and below the endpoint, this
matching establishes linear perimeter control. Above it,
a ball with perimeter of order $\Lambda$ has volume
of order $\Lambda^{\frac{n}{n-1}}$, and
the scale-one test shows oscillation of order
$\Lambda^{\frac{n}{(n-1)q}}$. Thus a single small component already
produces the sharp supercritical power.
\end{itemize}
\end{remark}

To show Theorem \ref{the-new-fixed-perimeter},
we need the following two lemmas,
which are also used in Section \ref{sec-new-exact}. The
first lemma is a generalization of
\cite[(3.1)]{abbf2016}.

\begin{lemma}\label{lem-new-characteristic-formula}
Let $q\in(0,\infty)$, $E\subset\mathbb{R}^n$ be measurable,
$Q$ be a cube with arbitrary orientation,
\begin{align*}
\theta:=\frac{|E\cap Q|}{|Q|}=(\mathbf{1}_E)_Q,
\quad\text{and}\quad
m_Q:=\min\{\theta,1-\theta\}.
\end{align*}
Then
\begin{align}\label{eq-new-characteristic-formula}
\operatorname{osc}_q(\mathbf{1}_E;Q)
=\left[\theta(1-\theta)^q+(1-\theta)\theta^q\right]^{\frac{1}{q}}.
\end{align}
If $\theta=\frac{1}{2}$, then
$\operatorname{osc}_q(\mathbf{1}_E;Q)=\frac{1}{2}$.
Moreover,
\begin{align}\label{eq-new-volume-fraction-ranges}
\operatorname{osc}_q(\mathbf{1}_E;Q)
\sim_q
\begin{cases}
m_Q, & q\in(0,1],\\[1mm]
m_Q^{\frac{1}{q}}, & q\in[1,\infty).
\end{cases}
\end{align}
\end{lemma}

\begin{proof}
By the definitions of $\theta$ and the local oscillation, we find that
\begin{align*}
\operatorname{osc}_q(\mathbf{1}_E;Q)^q
&=\frac{1}{|Q|}
\left[
\int_{E\cap Q}(1-\theta)^q\,dx
+\int_{Q\setminus E}\theta^q\,dx
\right]\\
&=\theta(1-\theta)^q+(1-\theta)\theta^q.
\end{align*}
This proves
\eqref{eq-new-characteristic-formula}, which, with
$\theta=\frac{1}{2}$, also implies the asserted special case.

To show \eqref{eq-new-volume-fraction-ranges},
we may assume that
$\theta\in[0,\frac{1}{2}]$. Thus $m_Q=\theta$ and
$\frac{1}{2}\leq1-\theta\leq1$. According as $q\geq1$ or $q\leq1$, we
have
\begin{align*}
\begin{cases}
\displaystyle
2^{-q}\theta
\leq\theta(1-\theta)^q+(1-\theta)\theta^q
\leq\theta+\theta^q\leq2\theta,
& q\in[1,\infty),\\
\displaystyle
\frac{1}{2}\theta^q
\leq\theta(1-\theta)^q+(1-\theta)\theta^q
\leq\theta+\theta^q\leq2\theta^q,
& q\in(0,1].
\end{cases}
\end{align*}
This proves
\eqref{eq-new-volume-fraction-ranges}
and hence completes the proof of Lemma
\ref{lem-new-characteristic-formula}.
\end{proof}

The second lemma is the scaled relative isoperimetric inequality. Its
short proof records the invariance needed for arbitrarily oriented
cubes.

\begin{lemma}\label{lem-new-relative-isoperimetric}
Let $n\in\mathbb{N}\cap[2,\infty)$.
There exists a positive constant
$C_{(n)}$, depending only on $n$, such that,
for every cube $Q\subset\mathbb{R}^n$ with arbitrary
orientation and every
measurable set $E$ of
finite perimeter in $Q$,
\begin{align}\label{eq-new-relative-isoperimetric}
\min\{|E\cap Q|,|Q\setminus E|\}^{\frac{n-1}{n}}
\leq C_{(n)}P(E,Q).
\end{align}
\end{lemma}

\begin{proof}
The endpoint $\mathrm{BV}$ Poincar\'e inequality on $(0,1)^n$
\cite[Remark 3.50, (3.41)]{AFP2000} is invariant under translations,
rotations, and dilations. Hence, for every
$u\in\mathrm{BV}(Q^\circ)$,
\begin{align}\label{eq-BVP}
\left\|u-\fint_Q u(z)\,dz
\right\|_{L^{\frac{n}{n-1}}(Q)}
\leq C_{(n)}|Du|(Q^\circ).
\end{align}

Since relative perimeter is unchanged by taking complements, we may
assume that $|E\cap Q|\leq|Q\setminus E|$. Thus
$(\mathbf{1}_E)_Q\leq\frac{1}{2}$. Applying \eqref{eq-BVP} to
$u=\mathbf{1}_E$ and using Definition \ref{def-new-perimeter}, we find
\begin{align*}
&\min\{|E\cap Q|,|Q\setminus E|\}^{\frac{n-1}{n}}\\
&\quad=|E\cap Q|^{\frac{n-1}{n}}
\leq2\left\|\mathbf{1}_E-(\mathbf{1}_E)_Q
\right\|_{L^{\frac{n}{n-1}}(Q)}\\
&\quad\leq2C_{(n)}|D\mathbf{1}_E|(Q^\circ)
=2C_{(n)}P(E,Q).
\end{align*}
Absorbing the factor $2$ into $C_{(n)}$ proves
\eqref{eq-new-relative-isoperimetric}.
This completes the proof of Lemma
\ref{lem-new-relative-isoperimetric}.
\end{proof}

\begin{proof}[Proof of Theorem \ref{the-new-fixed-perimeter}]
Fix $\Lambda\in(0,1]$.

\medskip
\textbf{Step 1.} \noindent\textit{In this step, we prove the upper
estimate in part {\rm(i)}.}
Let $E\subset Q_0$ be measurable and satisfy $P(E,Q_0)\leq \Lambda$.
We claim that
\begin{align}\label{eq-new-fixed-upper}
\|\mathbf{1}_E\|_{\mathcal{B}_q(Q_0)}
\leq C_{(n,q)}
\begin{cases}
\Lambda,
& q\in\left(0,\frac{n}{n-1}\right],\\[2mm]
\Lambda^{\frac{n}{(n-1)q}},
& q\in\left(\frac{n}{n-1},\infty\right).
\end{cases}
\end{align}
Fix $\varepsilon\in(0,1]$ and
$\{Q_j\}_{j\in J}\in\mathscr{F}_{\varepsilon}(Q_0)$.
We first prove the endpoint estimate.
Using Lemmas
\ref{lem-new-characteristic-formula} and
\ref{lem-new-relative-isoperimetric}, the identity
$|Q_j|=\varepsilon^n$, and the pairwise disjointness of the interiors
of the cubes, we have
\begin{align*}
&\varepsilon^{n-1}\sum_{j\in J}
\operatorname{osc}_{\frac{n}{n-1}}(\mathbf{1}_E;Q_j)\\
&\quad\leq C_{(n)}\sum_{j\in J}
\min\{|E\cap Q_j|,|Q_j\setminus E|\}^{\frac{n-1}{n}}\\
&\quad\leq C_{(n)}\sum_{j\in J}P(E,Q_j)
\leq C_{(n)}P(E,Q_0)
\leq C_{(n)}\Lambda.
\end{align*}
Taking the supremum over all admissible collections and then over
$\varepsilon\in(0,1]$, we conclude that
\begin{align}\label{eq-new-fixed-endpoint}
\|\mathbf{1}_E\|_{\mathcal{B}_{\frac{n}{n-1}}(Q_0)}
\leq C_{(n)}P(E,Q_0)
\leq C_{(n)}\Lambda.
\end{align}

If $q\in(0,\frac{n}{n-1}]$, then H\"older's inequality and
\eqref{eq-new-fixed-endpoint} imply
\begin{align*}
\|\mathbf{1}_E\|_{\mathcal{B}_q(Q_0)}
\leq\|\mathbf{1}_E\|_{\mathcal{B}_{\frac{n}{n-1}}(Q_0)}
\leq C_{(n)}\Lambda.
\end{align*}

Now assume that $q\in(\frac{n}{n-1},\infty)$. Since
$|\mathbf{1}_E-(\mathbf{1}_E)_{Q_j}|\leq1$, we have
\begin{align*}
\operatorname{osc}_q(\mathbf{1}_E;Q_j)
\leq
\operatorname{osc}_{\frac{n}{n-1}}(\mathbf{1}_E;Q_j)^{
\frac{n}{(n-1)q}}.
\end{align*}
Using this inequality, H\"older's inequality for finite sums, and
$\varepsilon^{n-1}\#J\leq1$, we find that
\begin{align*}
&\varepsilon^{n-1}\sum_{j\in J}
\operatorname{osc}_q(\mathbf{1}_E;Q_j)\\
&\quad\leq
\left(\varepsilon^{n-1}\#J\right)^{1-\frac{n}{(n-1)q}}
\left[
\varepsilon^{n-1}\sum_{j\in J}
\operatorname{osc}_{\frac{n}{n-1}}(\mathbf{1}_E;Q_j)
\right]^{\frac{n}{(n-1)q}}\\
&\quad\leq
\|\mathbf{1}_E\|_{\mathcal{B}_{\frac{n}{n-1}}(Q_0)}^{
\frac{n}{(n-1)q}}
\leq C_{(n,q)}P(E,Q_0)^{\frac{n}{(n-1)q}}
\leq C_{(n,q)}\Lambda^{\frac{n}{(n-1)q}}.
\end{align*}
Taking the two suprema in
the definition in \eqref{eq-new-Bq} proves
\eqref{eq-new-fixed-upper}; taking the supremum over $E$
establishes the upper
estimate in \eqref{eq-new-fixed-main}.

\medskip
\textbf{Step 2.} \noindent\textit{In this step, we prove the lower
estimate and its sharpness in part {\rm(i)}.}
We construct examples attaining the two powers in
\eqref{eq-new-fixed-upper}. Choose $c_{(n)}\in(0,\frac{1}{4}]$ sufficiently
small that
\begin{align*}
(n+1)c_{(n)}^{n-1}\leq1,\quad
n|B(0,1)|c_{(n)}^{n-1}\leq1,\quad
|B(0,1)|c_{(n)}^n\leq\frac{1}{2},
\end{align*}
and let $r:=c_{(n)}\Lambda^{\frac{1}{n-1}}\leq\frac{1}{4}$.

\textbf{Case 1.} Assume that
$q\in(0,\frac{n}{n-1}]$.
Choose a cube
$Q=x_0+[0,r]^n\subset Q_0$ with $\overline{Q}\subset Q_0$, and
define the rectangular box
\begin{align*}
E:=x_0+\left(\left(0,\frac{r}{2}\right)\times(0,r)^{n-1}\right).
\end{align*}
Since $\overline E\subset Q_0$, its relative perimeter is
\begin{align*}
P(E,Q_0)
=2r^{n-1}+2(n-1)\frac{r}{2}r^{n-2}
=(n+1)c_{(n)}^{n-1}\Lambda
\leq\Lambda.
\end{align*}
The singleton family $\{Q\}$ is admissible at scale $r$, and
$|E\cap Q|=\frac{|Q|}{2}$. Hence
Lemma \ref{lem-new-characteristic-formula} implies
\begin{align*}
\|\mathbf{1}_E\|_{\mathcal{B}_q(Q_0)}
\geq r^{n-1}\operatorname{osc}_q(\mathbf{1}_E;Q)
=\frac{1}{2}r^{n-1}
=\frac{1}{2}c_{(n)}^{n-1}\Lambda
\geq C_{(n)}^{-1}\Lambda.
\end{align*}

\textbf{Case 2.} Assume that
$q\in(\frac{n}{n-1},\infty)$, and define
\begin{align*}
E_\Lambda:=B\left(\left(\frac{1}{2},\ldots,\frac{1}{2}\right),r\right).
\end{align*}
Since $\overline E_\Lambda\subset Q_0$, the choice of $c_{(n)}$ implies
\begin{align*}
P(E_\Lambda,Q_0)
=n|B(0,1)|r^{n-1}
=n|B(0,1)|c_{(n)}^{n-1}\Lambda
\leq\Lambda.
\end{align*}
Moreover,
\begin{align*}
\theta:=\frac{|E_\Lambda\cap Q_0|}{|Q_0|}
=|B(0,1)|r^n
=|B(0,1)|c_{(n)}^n\Lambda^{\frac{n}{n-1}}
\in\left(0,\frac{1}{2}\right].
\end{align*}
Testing with the scale-one family $\{Q_0\}$ and using Lemma
\ref{lem-new-characteristic-formula}, we obtain
\begin{align*}
\|\mathbf{1}_{E_\Lambda}\|_{\mathcal{B}_q(Q_0)}
&\geq\operatorname{osc}_q(\mathbf{1}_{E_\Lambda};Q_0)\\
&=\left[\theta(1-\theta)^q+(1-\theta)\theta^q\right]^{\frac{1}{q}}\\
&\geq(1-\theta)\theta^{\frac{1}{q}}
\geq\frac{1}{2}\left[|B(0,1)|c_{(n)}^n
\Lambda^{\frac{n}{n-1}}\right]^{\frac{1}{q}}\\
&\geq C_{(n,q)}^{-1}
\Lambda^{\frac{n}{(n-1)q}}.
\end{align*}
The two cases establish the lower estimate in
\eqref{eq-new-fixed-main}.
Together with Step 1, they also show that neither power of $\Lambda$ can
be replaced uniformly by a larger one. This proves part {\rm(i)}.

\medskip
\textbf{Step 3.} \noindent\textit{In this step, we prove part {\rm(ii)}.}
Let $E\subset Q_0$ be measurable. Assume first that
$q\in(0,\frac{n}{n-1}]$. If
$P(E,Q_0)\in(0,1]$, we use \eqref{eq-new-fixed-upper} with
$\Lambda=P(E,Q_0)$; if $P(E,Q_0)=0$, we apply that estimate for every
$\Lambda\in(0,1]$ and let $\Lambda\to0^+$. Together with
\eqref{eq-new-characteristic-universal-bound} when
$P(E,Q_0)\in(1,\infty]$, this implies
\begin{align}\label{eq-new-characteristic-self-upper}
\|\mathbf{1}_E\|_{\mathcal{B}_q(Q_0)}
\leq C_{(n,q)}\min\{1,P(E,Q_0)\}.
\end{align}

The axis-parallel families used in the infinitesimal functional of
\cite[Corollary 4.2]{abbf2016} form a subclass of the testing families
in the definition in \eqref{eq-new-Bq}.
Consequently, that result implies
\begin{align}\label{eq-new-characteristic-self-B-lower}
\min\{1,P(E,Q_0)\}
\leq C_{(n)}\|\mathbf{1}_E\|_{\mathcal{B}(Q_0)}.
\end{align}
In particular, \eqref{eq-new-characteristic-self-upper} with $q=1$ and
\eqref{eq-new-characteristic-self-B-lower} show that
\begin{align}\label{eq-new-characteristic-B-perimeter}
\|\mathbf{1}_E\|_{\mathcal{B}(Q_0)}
\sim_n\min\{1,P(E,Q_0)\}.
\end{align}

It remains to compare the two functionals. If
$q\in[1,\frac{n}{n-1}]$, H\"older's inequality on each cube implies
$\operatorname{osc}_1(\mathbf{1}_E;Q)
\leq\operatorname{osc}_q(\mathbf{1}_E;Q)$. If $q\in(0,1)$, Lemma
\ref{lem-new-characteristic-formula}, applied with $q$ and with $1$,
implies
\begin{align*}
\operatorname{osc}_1(\mathbf{1}_E;Q)
\leq C_{(q)}\operatorname{osc}_q(\mathbf{1}_E;Q)
\end{align*}
for every cube $Q$.
Taking the sums and suprema in \eqref{eq-new-Bq}
and then using \eqref{eq-new-characteristic-self-upper} and
\eqref{eq-new-characteristic-B-perimeter}, we obtain
\eqref{eq-new-characteristic-self-equivalence}.

Finally, let $q\in(\frac{n}{n-1},\infty)$. For every
$\Lambda\in(0,1]$, Case 2 of Step 2 provides a
measurable set $E_{\Lambda}\subset Q_0$
such that
\begin{align*}
P(E_{\Lambda},Q_0)\leq \Lambda
\quad\hbox{and}\quad
\|\mathbf{1}_{E_{\Lambda}}\|_{\mathcal{B}_q(Q_0)}
\geq C_{(n,q)}^{-1}
\Lambda^{\frac{n}{(n-1)q}}.
\end{align*}
Suppose, to the contrary, that there existed
$K_{(n,q)}\in[1,\infty)$ such that
$
\|\mathbf{1}_E\|_{\mathcal{B}_q(Q_0)}
\leq K_{(n,q)}\|\mathbf{1}_E\|_{\mathcal{B}(Q_0)}
$
for every measurable set $E\subset Q_0$. Applying this to $E_{\Lambda}$
and using \eqref{eq-new-fixed-upper} with $q=1$, we obtain
\begin{align*}
0&<C_{(n,q)}^{-1}
\leq\Lambda^{-\frac{n}{(n-1)q}}
\|\mathbf{1}_{E_{\Lambda}}\|_{\mathcal{B}_q(Q_0)}\\
&\leq K_{(n,q)}\Lambda^{-\frac{n}{(n-1)q}}
\|\mathbf{1}_{E_{\Lambda}}\|_{\mathcal{B}(Q_0)}
\leq K_{(n,q)}C_{(n)}
\Lambda^{1-\frac{n}{(n-1)q}}
\to0
\end{align*}
as $\Lambda\to0^+$. This leads to
a contradiction. Hence \eqref{eq-new-characteristic-self-range}
cannot hold uniformly for $q\in(\frac{n}{n-1},\infty)$, which proves
the sharpness of the range
$q\in(0,\frac{n}{n-1}]$. This completes the proof of (ii)
and hence Theorem \ref{the-new-fixed-perimeter}.
\end{proof}

\section{Exact Perimeter Formulas\label{sec-new-exact}}

The preceding section related the quantities
$\|\mathbf{1}_E\|_{\mathcal{B}_q(Q_0)}$, for measurable sets
$E\subset Q_0$, to the perimeter $P(E,Q_0)$ and showed that
the corresponding self-improvement range for characteristic functions reaches
the endpoint $q=\frac{n}{n-1}$. In this section, we pass from these bounds
to an exact perimeter formula by taking the small-scale limit of BBM-type
seminorms defined using cubes with arbitrary
orientations. For $q=1$, Ambrosio et al.
\cite[Theorem 4.1]{abbf2016} proved that this limit is
$\frac{1}{2}\min\{1,P(E,Q_0)\}$. We prove that the same formula holds
in every dimension $n\geq2$ for all $q\in(0,\frac{n}{n-1}]$.
In the case $n=2$, it
continues to hold up to some universal exponent
$q_*>\frac{19}{9}$.

Neither the formula for $q=1$ nor its proof extends directly to
$q\ne1$. Indeed, comparison of the local $L^q$ and $L^1$ oscillations implies
only the upper bound when $q<1$, and only the lower bound when $q>1$. For
$q<1$, we recover the missing lower bound from families that nearly attain
the $q=1$ limit, on which the two oscillation sums have the same limit. For
$q\in(1,\frac{n}{n-1}]$, the missing upper bound must hold uniformly over
cube families that may change with the scale. Uniform approximation by
half-spaces reduces it to a geometric inequality for a hyperplane cutting
a unit cube, which we prove using the concavity of parallel sections and a
second-moment estimate. In the case $n=2$, this pointwise inequality fails
above $q=2$, but averaging over small common translations of the square
first establishes the explicit range $q\leq\frac{19}{9}$,
and a stability argument extends it slightly beyond $\frac{19}{9}$.

\subsection{The Isotropic Functional and the Main Theorem}

We begin by defining the \emph{isotropic functional}
based on arbitrarily
oriented cubes and stating the exact perimeter formula proved in this
section. We then compare this functional with its axis-parallel version.

\begin{definition}\label{def-new-perimeter-isotropic}
Let $U\subset\mathbb{R}^n$ be open, $M\in(0,\infty)$,
and $\varepsilon\in(0,1)$. We denote by
$\mathscr{F}_{\varepsilon,M}^{*}(U)$
the set of all finite collections
$\{Q_j\}_{j\in J}$ of cubes with pairwise
disjoint interiors contained in
$U$ such that every $Q_j$ may have arbitrary orientation,
$\ell(Q_j)=\varepsilon$ for every $j\in J$,
and $\#J\leq M\varepsilon^{1-n}$.
For every $q\in(0,\infty)$ and every measurable $E\subset U$, define
\begin{align}\label{eq-new-I-general}
I_{\varepsilon,q}^{*}(E;U,M)
:=\varepsilon^{n-1}
\sup_{\{Q_j\}_{j\in J}\in\mathscr{F}_{\varepsilon,M}^{*}(U)}
\sum_{j\in J}\operatorname{osc}_q(\mathbf{1}_E;Q_j).
\end{align}
When $U=Q_0$ and $M=1$, write
\begin{align}\label{eq-new-I-unit}
I_{\varepsilon,q}^{*}(E):=I_{\varepsilon,q}^{*}(E;Q_0,1).
\end{align}
\end{definition}

\begin{remark}
Here and thereafter,
$M$ is fixed independently of $\varepsilon$. For every
$\varepsilon\in(0,1)$, the functional in \eqref{eq-new-I-general} is
finite; indeed,
$
I_{\varepsilon,q}^{*}(E;U,M)\leq M.
$
Moreover, when $U$ is the interior of a cube, the number of cubes of edge
length $\varepsilon$ with pairwise disjoint interiors that can be placed
in $U$ is of order $\varepsilon^{-n}$, whereas an admissible family
contains at most $M\varepsilon^{1-n}$ cubes. Since
$
\frac{M\varepsilon^{1-n}}{\varepsilon^{-n}}
=M\varepsilon\to0
$
as $\varepsilon\to0^+$, the cardinality restriction remains relevant,
even when the fixed constant $M$ is large.
\end{remark}

We now state the exact perimeter formulas, which are the main results
of this section.

\begin{theorem}\label{the-new-exact}
Let $n\in\mathbb{N}\cap[2,\infty)$ and let $E\subset Q_0$ be measurable.
Then
\begin{align}\label{eq-new-exact-all-dim}
\lim_{\varepsilon\to0^+}I_{\varepsilon,q}^{*}(E)
=\frac{1}{2}\min\{1,P(E,Q_0)\}
\end{align}
for every $q\in(0,\frac{n}{n-1}]$.
Moreover, there exists a universal constant
$q_*\in(\frac{19}{9},3)$ such that, when $n=2$,
\eqref{eq-new-exact-all-dim} also holds for every $q\in(2,q_*]$.
\end{theorem}

\begin{remark}\label{rem-new-exact-known}
When $q=1$, \eqref{eq-new-exact-all-dim} is precisely the isotropic
perimeter formula of Ambrosio et al. \cite[Theorem 4.1]{abbf2016}.
That result is therefore recovered here, whereas the new contribution is
the persistence of the same formula for the stated ranges of $q\ne1$.
The proof for general $q$ requires different arguments because the direct
comparison between the local $L^q$ and $L^1$ oscillations supplies only one
of the two required bounds on either side of $q=1$. Theorem
\ref{the-new-exact} shows that its right-hand side
$\frac{1}{2}\min\{1,P(E,Q_0)\}$ remains unchanged as $q$ varies throughout
the stated ranges. In dimension two, this remains true
up to the universal exponent $q_*>\frac{19}{9}$ in Theorem
\ref{the-new-exact}, so the exact small-scale formula persists beyond the endpoint
of the uniform characteristic-function comparison in Theorem
\ref{the-new-fixed-perimeter}. The remaining gap and the failure above an
explicit exponent less than $2.1541$ are discussed in Section
\ref{sec-new-open-problems}.
\end{remark}

To relate the isotropic functional to the
\emph{axis-parallel functional} used in
the preceding section, we record the following comparison. For
$q\in(0,\infty)$, $\varepsilon\in(0,1)$, and a measurable set
$E\subset Q_0$, define
\begin{align*}
I_{\varepsilon,q}(E)
:=\varepsilon^{n-1}
\sup_{\{Q_j\}_{j\in J}\in\mathscr{F}_{\varepsilon}(Q_0)}
\sum_{j\in J}\operatorname{osc}_q(\mathbf{1}_E;Q_j).
\end{align*}

\begin{proposition}\label{prop-new-axis-isotropic}
Let $q\in(0,\infty)$ and let $E\subset Q_0$ be measurable. Then
\begin{align}\label{eq-new-axis-isotropic-comparison}
I_{\varepsilon,q}(E)
\leq I_{\varepsilon,q}^{*}(E)
\leq C_{(n,q)}I_{\sqrt n\varepsilon,q}(E)
\end{align}
for every $\varepsilon\in(0,\frac{1}{\sqrt n})$.
\end{proposition}

\begin{proof}
The first inequality is immediate because every axis-parallel
admissible family is also admissible in \eqref{eq-new-I-unit}. For the
second, fix
$\{Q_j\}_{j\in J}\in\mathscr{F}_{\varepsilon,1}^{*}(Q_0)$.
As proved in \cite[Lemma 4.3]{abbf2016}, each $Q_j$ is contained in an
axis-parallel cube $Q_j'\subset Q_0$ of edge length
$\sqrt n\varepsilon$.

Let $x_j$ be the center of $Q_j$. The balls
$B(x_j,\frac{\varepsilon}{2})$ have pairwise disjoint interiors. If
$Q_j'\cap Q_k'\ne\emptyset$, then
$\operatorname{diam}(Q_j')=n\varepsilon$ and
$x_j\in Q_j'$, $x_k\in Q_k'$ imply
$
|x_j-x_k|\leq2n\varepsilon.
$
Hence all the balls $B(x_k,\frac{\varepsilon}{2})$ corresponding to cubes
$Q_k'$ that meet a fixed $Q_j'$ are contained in
$B(x_j,[2n+\frac12]\varepsilon)$. Comparing volumes, we obtain
\begin{align*}
\#\{k\in J:\ Q_j'\cap Q_k'\ne\emptyset\}
\leq(4n+1)^n=:N_n.
\end{align*}
Define a finite graph $G$ on $J$ by joining distinct $j,k\in J$
whenever $Q_j'\cap Q_k'\ne\emptyset$. Since the preceding count
includes $k=j$, every vertex of $G$ has at most $N_n-1$ neighbors.
Order the vertices arbitrarily and color them successively, assigning
to each vertex a color different from those of its previously colored
neighbors. At each step, at most $N_n-1$ colors are forbidden, so one
of $N_n$ colors remains available. Thus adjacent vertices receive
different colors. Consequently, the color classes partition
$\{Q_j'\}_{j\in J}$ into at most $N_n$ subfamilies of pairwise
disjoint cubes.

Let $\delta:=\sqrt n\varepsilon$ and
$K_\delta:=\lfloor\delta^{1-n}\rfloor$. Since $\delta\in(0,1)$,
$
K_\delta\geq\frac12\delta^{1-n}.
$
Each color class contains at most
$\#J\leq\varepsilon^{1-n}$ cubes. Splitting it into subfamilies of at
most $K_\delta$ cubes therefore produces no more than
\begin{align*}
\frac{\varepsilon^{1-n}}{K_\delta}+1
\leq2\frac{\varepsilon^{1-n}}{\delta^{1-n}}+1
=2n^{\frac{n-1}{2}}+1
\end{align*}
subfamilies. Consequently, there exists a partition
$J=\bigcup_{\nu=1}^{L}J_\nu$ such that
\begin{align}\label{eq-new-axis-isotropic-partition}
L\leq N_n\left(2n^{\frac{n-1}{2}}+1\right)
\quad\text{and}\quad
\{Q_j'\}_{j\in J_\nu}
\in\mathscr{F}_{\delta}(Q_0).
\end{align}
For $t\in[0,1]$, let
\begin{align*}
\Phi_q(t):=
\left[t(1-t)^q+(1-t)t^q\right]^{\frac1q}.
\end{align*}
For every $j\in J$, let
\begin{align*}
\theta_j:=\frac{|E\cap Q_j|}{|Q_j|}
\quad\text{and}\quad
\theta_j':=\frac{|E\cap Q_j'|}{|Q_j'|}.
\end{align*}
Then the inclusion $Q_j\subset Q_j'$ and
$|Q_j'|=n^{\frac{n}{2}}|Q_j|$ imply
\begin{align}\label{theta1}
\min\{\theta_j',1-\theta_j'\}
&=\frac{1}{|Q_j'|}
\min\{|E\cap Q_j'|,|Q_j'\setminus E|\}\nonumber\\
&\geq n^{-\frac n2}\min\{\theta_j,1-\theta_j\}.
\end{align}
By Lemma \ref{lem-new-characteristic-formula},
$\Phi_q(t)\sim_q\min\{t,1-t\}^{\min\{1,\frac1q\}}$.
From this and \eqref{theta1}, we infer that,
for any $j\in J$,
\begin{align}\label{eq-new-axis-isotropic-local-comparison}
\left[
\fint_{Q_j}|\mathbf{1}_E(x)
-(\mathbf{1}_E)_{Q_j}|^q\,dx
\right]^{\frac1q}
&=\Phi_q(\theta_j)
\leq C_{(n,q)}\Phi_q(\theta_j')\nonumber\\
&\leq C_{(n,q)}
\left[
\fint_{Q_j'}|\mathbf{1}_E(x)-(\mathbf{1}_E)_{Q_j'}|^q\,dx
\right]^{\frac1q}.
\end{align}
Combining \eqref{eq-new-axis-isotropic-partition} and
\eqref{eq-new-axis-isotropic-local-comparison}, and using
$\delta=\sqrt n\varepsilon$, we obtain
\begin{align*}
&\varepsilon^{n-1}\sum_{j\in J}
\operatorname{osc}_q(\mathbf{1}_E;Q_j)\\
&\quad\leq
C_{(n,q)}
\sum_{\nu=1}^{L}
\varepsilon^{n-1}\sum_{j\in J_\nu}
\operatorname{osc}_q(\mathbf{1}_E;Q_j')\\
&\quad\leq
C_{(n,q)}L
\left(\frac{\varepsilon}{\delta}\right)^{n-1}
I_{\delta,q}(E)
\leq C_{(n,q)}I_{\sqrt n\varepsilon,q}(E).
\end{align*}
Taking the supremum over
$\mathscr{F}_{\varepsilon,1}^{*}(Q_0)$ proves the second inequality in
\eqref{eq-new-axis-isotropic-comparison}
and completes the proof of Proposition
\ref{prop-new-axis-isotropic}.
\end{proof}

\begin{remark}\label{rem-new-axis-isotropic}
Proposition \ref{prop-new-axis-isotropic} when $q=1$ coincides with
\cite[Lemma 4.3]{abbf2016}. The extension to every $q\in(0,\infty)$ uses
the same geometric construction together with the volume-fraction estimates
for characteristic functions. The comparison constant depends on $n$ and
$q$, so \eqref{eq-new-axis-isotropic-comparison} does not preserve the exact
factor $\frac12$ in the isotropic perimeter formula in
Theorem \ref{the-new-exact}.
\end{remark}

\subsection{The Universal Lower Bound}

The aim of this subsection is to prove the universal lower bound
for every $q\in(0,\infty)$
stated in Proposition \ref{prop-new-universal-lower} below. We
begin with the $q=1$ formula for arbitrary $M$. When $M=1$, it is precisely
\cite[Theorem 4.1]{abbf2016} applied to $Q_0$.

\begin{lemma}\label{lem-new-known-q1}
Let $n\in\mathbb{N}\cap[2,\infty)$.
Let $M\in(0,\infty)$ and let $E\subset\mathbb{R}^n$ be measurable. Then
\begin{align}\label{eq-new-known-general-m}
\lim_{\varepsilon\to0^+}I_{\varepsilon,1}^{*}(E;Q_0,M)
=\frac{1}{2}\min\{M,P(E,Q_0)\}.
\end{align}
\end{lemma}

\begin{proof}
The case $M=1$ follows by applying \cite[Theorem 4.1]{abbf2016}
to the
open cube $Q_0$. The choice of open or closed representatives of the
testing cubes is immaterial here, since cube boundaries have zero Lebesgue
measure and admissibility is formulated in terms of disjoint interiors.

For general $M$, let $a:=M^{-\frac{1}{n-1}}$ and define
\begin{align*}
T_a:\mathbb{R}^n\to\mathbb{R}^n,
\quad T_a(x):=ax.
\end{align*}
Let
$Q_a:=T_a(Q_0)$ and $E_a:=T_a(E)$.
For every $\varepsilon\in(0,\min\{1,a^{-1}\})$, we have
$
M\varepsilon^{1-n}=(a\varepsilon)^{1-n}.
$
Hence the map
\begin{align*}
\begin{cases}
\displaystyle
\mathscr{F}_{\varepsilon,M}^{*}(Q_0)
\to \mathscr{F}_{a\varepsilon,1}^{*}(Q_a)\\
\displaystyle
\{Q_j\}_{j\in J}\mapsto\{T_a(Q_j)\}_{j\in J}
\end{cases}
\end{align*}
is a bijection. Moreover, a change of
variables implies
$
\operatorname{osc}_1(\mathbf{1}_{E_a};T_a(Q_j))
=\operatorname{osc}_1(\mathbf{1}_E;Q_j).
$
Consequently,
\begin{align*}
I^*_{a\varepsilon,1}(E_a;Q_a,1)
&=a^{n-1}I^*_{\varepsilon,1}(E;Q_0,M)
=M^{-1}I^*_{\varepsilon,1}(E;Q_0,M).
\end{align*}
Since
\begin{align*}
P(E_a,Q_a)=a^{n-1}P(E,Q_0)=M^{-1}P(E,Q_0),
\end{align*}
letting $\varepsilon\to0^+$ and applying the $M=1$ formula
to $Q_a$
obtained in \cite[Theorem 4.1]{abbf2016},
we obtain
\begin{align*}
\lim_{\varepsilon\to0^+}I^*_{\varepsilon,1}(E;Q_0,M)
&=\frac{M}{2}\min\{1,M^{-1}P(E,Q_0)\}\\
&=\frac12\min\{M,P(E,Q_0)\}.
\end{align*}
This completes the proof of Lemma \ref{lem-new-known-q1}.
\end{proof}

The following lemma plays a key role in the
proof of the lower bound for
$q\in(0,1)$ in Theorem \ref{the-new-exact}.
For $q\in(0,1)$, H\"older's inequality implies
$I_{\varepsilon,q}^{*}(E)\leq I_{\varepsilon,1}^{*}(E)$,
but this provides no lower bound.
To overcome this, the following lemma shows that if the $q=1$ sum over
a family tends to $\frac{M}{2}$ as $\varepsilon\to0^+$, then the $q$-sum
over the same family also tends to $\frac{M}{2}$ as $\varepsilon\to0^+$.

\begin{lemma}\label{lem-new-saturation}
Let $n\in\mathbb{N}\cap[2,\infty)$.
Let $q\in(0,1)$, $M\in(0,1]$, and let $E\subset Q_0$ be measurable.
For every $\varepsilon\in(0,M^{\frac{1}{n-1}})$, choose an admissible family
$
\mathcal{Q}_\varepsilon
:=\{Q_{\varepsilon,j}\}_{j\in J_\varepsilon}
\in\mathscr{F}_{\varepsilon,M}^{*}(Q_0).
$
The choice of $\mathcal{Q}_\varepsilon$ may depend on $\varepsilon$.
If
\begin{align}\label{eq-new-saturation-assumption}
\lim_{\varepsilon\to 0^+}\varepsilon^{n-1}\sum_{j\in J_\varepsilon}
\operatorname{osc}_1(\mathbf{1}_E;Q_{\varepsilon,j})=\frac{M}{2}.
\end{align}
Then
\begin{align}\label{eq-new-saturation-q}
\lim_{\varepsilon\to 0^+}\varepsilon^{n-1}\sum_{j\in J_\varepsilon}
\operatorname{osc}_q(\mathbf{1}_E;Q_{\varepsilon,j})=\frac{M}{2}.
\end{align}
\end{lemma}

\begin{proof}
We first claim that there exists $C_q\in(0,\infty)$ such that,
for every cube $Q$,
\begin{align}\label{eq-new-local-deficit-comparison}
0
\leq\operatorname{osc}_1(\mathbf{1}_E;Q)
-\operatorname{osc}_q(\mathbf{1}_E;Q)
\leq C_q\left[
\frac{1}{2}-\operatorname{osc}_1(\mathbf{1}_E;Q)
\right].
\end{align}
The first inequality follows from H\"older's inequality. To prove the
second, let $t:=\frac{|E\cap Q|}{|Q|}$. Then, by Lemma
\ref{lem-new-characteristic-formula}, the desired inequality becomes
\begin{align*}
2t(1-t)-
\left[t(1-t)^q+(1-t)t^q\right]^{\frac{1}{q}}
\leq C_q\left[\frac{1}{2}-2t(1-t)\right].
\end{align*}
It is an equality when $t=\frac12$, while
$\frac12-2t(1-t)>0$ for $t\neq\frac12$.
It therefore remains to show that
\begin{align*}
f(t):=\frac{
2t(1-t)-
\left[t(1-t)^q+(1-t)t^q\right]^{\frac{1}{q}}
}{
\frac{1}{2}-2t(1-t)
}
\end{align*}
is bounded on $[0,1]\setminus\{\frac12\}$. $f(t)$ is
continuous away from $\frac12$.
In addition, by Taylor's formula, we find that
\begin{align*}
\lim_{t\to\frac12}f(t)
&=\lim_{x\to0}\frac{
2\left(\frac12+x\right)\left(\frac12-x\right)-
\left[
\left(\frac12+x\right)\left(\frac12-x\right)^q
+\left(\frac12-x\right)\left(\frac12+x\right)^q
\right]^{\frac{1}{q}}
}{
\frac{1}{2}-2\left(\frac12+x\right)\left(\frac12-x\right)
}\\
&=\lim_{x\to0}
\frac{
\left(\frac12-2x^2\right)
-\left[\frac12-(3-q)x^2+O(x^4)\right]}{2x^2}\\
&=\lim_{x\to0}
\frac{(1-q)x^2+O(x^4)}{2x^2}
=\frac{1-q}{2}.
\end{align*}
Thus $f(t)$ is bounded
on $[0,1]\setminus\{\frac12\}$, which proves
\eqref{eq-new-local-deficit-comparison}.

Using \eqref{eq-new-local-deficit-comparison},
the cardinality bound in the
definition of $\mathscr{F}_{\varepsilon,M}^{*}(Q_0)$,
and \eqref{eq-new-saturation-assumption},
we conclude that
\begin{align*}
0
&\leq\varepsilon^{n-1}\sum_{j\in J_\varepsilon}
\left\{\operatorname{osc}_1(\mathbf{1}_E;Q_{\varepsilon,j})
-\operatorname{osc}_q(\mathbf{1}_E;Q_{\varepsilon,j})\right\}\\
&\leq C_q\left[
\frac{\varepsilon^{n-1}\#J_\varepsilon}{2}
-\varepsilon^{n-1}\sum_{j\in J_\varepsilon}
\operatorname{osc}_1(\mathbf{1}_E;Q_{\varepsilon,j})\right]\\
&\leq C_q\left[
\frac{M}{2}
-\varepsilon^{n-1}\sum_{j\in J_\varepsilon}
\operatorname{osc}_1(\mathbf{1}_E;Q_{\varepsilon,j})\right]
\to0
\end{align*}
as $\varepsilon\to0^+$,
which further implies
\eqref{eq-new-saturation-q}
and hence completes the proof of Lemma \ref{lem-new-saturation}.
\end{proof}

Combining the $q=1$ formula \eqref{eq-new-known-general-m} with Lemma
\ref{lem-new-saturation} and, when $q\geq1$, H\"older's inequality, we
obtain a lower bound valid for every $q\in(0,\infty)$.

\begin{proposition}\label{prop-new-universal-lower}
For every $n\in\mathbb{N}\cap[2,\infty)$,
every $q\in(0,\infty)$, and
every measurable
$E\subset Q_0$,
\begin{align}\label{eq-new-universal-lower}
\liminf_{\varepsilon\to0^+}I_{\varepsilon,q}^{*}(E)
\geq\frac{1}{2}\min\{1,P(E,Q_0)\}.
\end{align}
\end{proposition}

\begin{proof}
If $q\geq1$, then H\"older's inequality implies
$I_{\varepsilon,q}^{*}(E)\geq I_{\varepsilon,1}^{*}(E)$, and the
claim follows from Lemma \ref{lem-new-known-q1} with $M=1$. Suppose
$q\in(0,1)$. The assertion is immediate when $P(E,Q_0)=0$; otherwise
let $M_E:=\min\{1,P(E,Q_0)\}\in(0,1]$.

By the definition of the supremum, for every
$\varepsilon\in(0,M_E^{\frac{1}{n-1}})$, we can choose an admissible family
$\{Q_{\varepsilon,j}\}_{j\in J_\varepsilon}\in
\mathscr F^{*}_{\varepsilon,M_E}(Q_0)$ such that
\begin{align}\label{1=q2}
I^{*}_{\varepsilon,1}(E;Q_0,M_E)-\varepsilon
\leq\varepsilon^{n-1}\sum_{j\in J_\varepsilon}
\operatorname{osc}_1(\mathbf{1}_E;Q_{\varepsilon,j})
\leq I^{*}_{\varepsilon,1}(E;Q_0,M_E).
\end{align}
Since $M_E\leq1$, the same family is also admissible in the
definition of $I_{\varepsilon,q}^{*}(E)$. Using Lemmas
\ref{lem-new-known-q1} and \ref{lem-new-saturation}
and \eqref{1=q2}, we obtain
\begin{align*}
\liminf_{\varepsilon\to0^+}I_{\varepsilon,q}^{*}(E)
&\geq\lim_{\varepsilon\to0^+}\varepsilon^{n-1}
\sum_{j\in J_\varepsilon}
\operatorname{osc}_q(\mathbf{1}_E;Q_{\varepsilon,j})\\
&=\lim_{\varepsilon\to0^+}\varepsilon^{n-1}
\sum_{j\in J_\varepsilon}
\operatorname{osc}_1(\mathbf{1}_E;Q_{\varepsilon,j})\\
&=\lim_{\varepsilon\to0^+}
I^{*}_{\varepsilon,1}(E;Q_0,M_E)
=\frac{M_E}{2}
=\frac12\min\{1,P(E,Q_0)\}.
\end{align*}
This completes the proof of \eqref{eq-new-universal-lower}
and hence Proposition \ref{prop-new-universal-lower}.
\end{proof}

\subsection{Uniform Boundary Blow-Up for Cubes
with Varying Centers and Orientations\label{5.3}}

We next establish the blow-up statement needed to localize the upper-bound
argument. At each scale, the almost maximizing family may be different, so
the cubes selected in the argument may have scale-dependent centers and
orientations. De Giorgi's theorem states that, at almost every point of the
reduced boundary, the rescaled set converges to its tangent half-space as
$r\to0^+$; the corresponding perimeter-measure convergence as $r\to0^+$
follows from the blow-up theorem
for rectifiable measures; see \cite[Theorem 3.59, p.\,157]{AFP2000} and
\cite[Theorem 10.2]{Maggi2012}. This result is
pointwise in the base point and therefore does not directly control cubes
whose centers vary with the scale. The next lemma supplies the needed
uniformity on a compact set whose complement has arbitrarily small perimeter
measure and records the resulting stability when the
cubes move and rotate.

To this end, we first recall some notation.
For a measurable set $E\subset Q_0$ of finite perimeter, let $\sigma_E$
denote the Radon measure on $\mathbb{R}^n$ obtained by restricting
$|D\mathbf{1}_E|$ to $Q_0$ and extending this restriction by zero:
for every Borel set $A\subset\mathbb{R}^n$,
\begin{align}\label{eq-new-relative-variation-extension}
\sigma_E(A):=|D\mathbf{1}_E|(A\cap Q_0).
\end{align}
Since $U\cap Q_0$ is open whenever $U\subset\mathbb{R}^n$ is open, the
definition of relative perimeter implies
$\sigma_E(U)=P(E,U\cap Q_0)$. In particular, if $Q$ is a cube with
$Q^\circ\subset Q_0$, then $\sigma_E(Q^\circ)=P(E,Q)$. Moreover,
$\sigma_E(Q_0)=P(E,Q_0)$ and $\sigma_E$
does not charge
$\partial Q_0$. In general, it is not the variation measure of the zero
extension of $\mathbf{1}_E$, which may charge $\partial Q_0$.
In what follows, $\partial^*E$ denotes the
\emph{reduced boundary} (denoted by
$\mathcal{F}E$ in \cite[Definition 3.54]{AFP2000}), namely
\begin{align*}
\partial^*E:=\left\{x\in\operatorname{supp}|D\mathbf{1}_E|:
\lim_{\rho\to0^+}
\frac{D\mathbf{1}_E(B(x,\rho))}
{|D\mathbf{1}_E|(B(x,\rho))}\ \text{exists in }\mathbb{S}^{n-1}\right\}.
\end{align*}
At such a point, $\nu_E(x)$ is minus the limit above; thus our
measure-theoretic outer-normal convention is
$D\mathbf{1}_E=-\nu_E|D\mathbf{1}_E|$.
For a Borel map $T:\mathbb{R}^n\to\mathbb{R}^n$ and a Radon measure $\mu$,
$T_\#\mu$ denotes the \emph{push-forward measure},
defined by
$
(T_\#\mu)(A):=\mu(T^{-1}(A))
$
for every Borel set $A$; equivalently,
$\int_{\mathbb{R}^n} f(z)\,d(T_\#\mu)(z)=\int_{\mathbb{R}^n} f(T(y))\,d\mu(y)$ whenever the integrals
are defined.
For $x\in\partial^*E\cap Q_0$ and $r\in(0,\infty)$, let
\begin{align*}
E_{x,r}:=\frac{E-x}{r}
\quad\text{and}\quad
\mu_{x,r}:=r^{1-n}
\left[\left(y\mapsto\frac{y-x}{r}\right)_\#\sigma_E\right],
\end{align*}
and define $H_x$ and the Radon measure $\mu_x$ by
\begin{align*}
H_x:=\{z\in\mathbb{R}^n:z\cdot\nu_E(x)<0\}
\quad\text{and}\quad
\mu_x(A):=\mathcal{H}^{n-1}(A\cap\partial H_x)
\end{align*}
for every Borel subset $A\subset \mathbb{R}^n$.
Thus $\mu_x$ is the $(n-1)$-dimensional surface measure on the hyperplane
$\partial H_x$.
For $R\in(0,\infty)$ and
$\phi\in C_{\mathrm{c}}(B(0,R))$, viewed as zero outside $B(0,R)$, let
\begin{align*}
\operatorname{Lip}(\phi)
:=\sup_{\{z,w\in\mathbb{R}^n:\ z\neq w\}}
\frac{|\phi(z)-\phi(w)|}{|z-w|}.
\end{align*}
For locally finite Radon measures $\mu,\lambda$ on $\mathbb{R}^n$, define
their \emph{localized bounded-Lipschitz (Rubinstein)
distance} on $B(0,R)$ by
\begin{align*}
d_{\mathrm{R}}(\mu,\lambda)
:=\sup_{\{\phi\in C_{\mathrm{c}}(B(0,R)),\
\|\phi\|_{L^\infty(\mathbb{R}^n)}
+\operatorname{Lip}(\phi)\leq1\}}
\left|\int_{B(0,R)}\phi(x)\,d\mu(x)
-\int_{B(0,R)}\phi(x)\,d\lambda(x)\right|.
\end{align*}
For nonempty compact sets $A,B\subset\mathbb{R}^n$, define their
\emph{Hausdorff distance} by
\begin{align*}
d_{\mathrm{H}}(A,B)
:=\max\left\{
\sup_{a\in A}\operatorname{dist}(a,B),
\sup_{b\in B}\operatorname{dist}(b,A)
\right\}.
\end{align*}

We now state the required uniform blow-up lemma.

\begin{lemma}
\label{lem-new-uniform-blowup}
Let $E\subset Q_0$ be measurable with finite perimeter, and let
$\delta\in(0,1)$. There
exists a compact set $G\subset\partial^*E\cap Q_0$ such that
$
\sigma_E(Q_0\setminus G)<\delta,
$
$\nu_E|_G$ is continuous, and, for every $R\in(0,\infty)$,
\begin{align}\label{eq-new-uniform-set-blowup}
\lim_{r_0\to0^+}\sup_{x\in G,\ r\in(0,r_0)}
\int_{B(0,R)}
|\mathbf{1}_{E_{x,r}}(z)-\mathbf{1}_{H_x}(z)|\,dz
=0
\end{align}
and
\begin{align}\label{eq-new-uniform-measure-blowup}
\lim_{r_0\to0^+}\sup_{x\in G,\ r\in(0,r_0)}
d_{\mathrm{R}}(\mu_{x,r},\mu_x)
=0.
\end{align}
If $G=\emptyset$, each supremum in
\eqref{eq-new-uniform-set-blowup} and
\eqref{eq-new-uniform-measure-blowup} is understood to be zero.
Equivalently, once $G$ has been chosen, for every fixed
$R>0$ and every $\eta>0$ there exists $r_{R,\eta}>0$, independent of
$x\in G$, such that both errors in
\eqref{eq-new-uniform-set-blowup} and
\eqref{eq-new-uniform-measure-blowup} are smaller than $\eta$ whenever
$x\in G$ and $0<r<r_{R,\eta}$.
\end{lemma}

\begin{proof}
\textbf{Step 1.} \noindent\textit{In this step, we recall the pointwise
set and perimeter-measure blow-ups and prove the latter convergence in
$d_{\mathrm{R}}$.}
By De Giorgi's structure and set blow-up theorem
\cite[Theorem 3.59, p.\,157]{AFP2000},
together with the blow-up theorem for
rectifiable measures \cite[Theorem 10.2]{Maggi2012},
we find that, for $\sigma_E$-almost every
$x\in\partial^*E\cap Q_0$,
\begin{align}\label{eq-new-pointwise-set-blowup}
\mathbf{1}_{E_{x,r}}\to\mathbf{1}_{H_x}
\quad\text{in }L^1(B(0,R))
\end{align}
as $r\to0^+$ for every $R>0$, and
\begin{align}\label{eq-new-pointwise-measure-blowup}
\mu_{x,r}\stackrel{*}{\rightharpoonup}\mu_x
\quad\text{locally in }\mathbb{R}^n
\end{align}
as $r\to0^+$. For fixed $x\in Q_0$ and $R>0$, the
restriction defining $\sigma_E$ does not affect the restriction of
$\mu_{x,r}$ to $B(0,R)$ for all sufficiently small $r$.
Since every ball centered at the origin
is a $\mu_x$-continuity set, the portmanteau theorem implies
\begin{align*}
\mu_{x,r}(B(0,R))\to\mu_x(B(0,R))
\end{align*}
as $r\to0^+$. In particular, there exist constants
$r_{x,R}>0$ and $C_{x,R}<\infty$ such that
\begin{align*}
\mu_{x,r}(B(0,R))+\mu_x(B(0,R))\leq C_{x,R}
\end{align*}
for $0<r<r_{x,R}$.
Let $\mathcal{F}_R$ denote the test-function class in the definition of
$d_{\mathrm{R}}$. After extension by zero, $\mathcal{F}_R$ is uniformly
bounded and equi-Lipschitz on $\overline{B(0,R)}$; hence, by the
Arzel\`a--Ascoli theorem (see,
for instance, \cite[Theorem 7.25, p.\,158]{Rudin1976}), it is
totally bounded in the uniform norm. Given
$\eta>0$, choose a finite $\eta$-net
$\{\phi_1,\ldots,\phi_N\}\subset\mathcal{F}_R$. It follows that
\begin{align*}
d_{\mathrm{R}}(\mu_{x,r},\mu_x)
&\leq\max_{1\leq j\leq N}
\left|\int_{B(0,R)}\phi_j\,d(\mu_{x,r}-\mu_x)\right|\\
&\quad+\eta\bigl[\mu_{x,r}(B(0,R))+\mu_x(B(0,R))\bigr].
\end{align*}
The first term tends to zero as $r\to0^+$ by
\eqref{eq-new-pointwise-measure-blowup}, whereas the second is bounded by
$\eta C_{x,R}$. Taking the limit superior as $r\to0^+$ and then letting
$\eta\to0^+$, we conclude that
\begin{align}\label{eq-new-pointwise-dR-blowup}
d_{\mathrm{R}}(\mu_{x,r},\mu_x)\to0
\end{align}
as $r\to0^+$.

\textbf{Step 2.} \noindent\textit{In this step, we introduce the
blow-up errors and prove their measurability.}
Changing $E$ on a Lebesgue-null set if necessary, we assume that $E$ is
Borel; this does not change its $L^1$ class or its perimeter measure.
Since, for every $r>0$, the maps
$x\mapsto D\mathbf{1}_E(B(x,r))$ and
$x\mapsto |D\mathbf{1}_E|(B(x,r))$ are Borel,
the definition of $\partial^*E$ shows that
$\nu_E$ is Borel on $\partial^*E$.
Fix $R\in(0,\infty)$. For $x\in\partial^*E\cap Q_0$ and
$r\in(0,\infty)$, let
\begin{align*}
\alpha_R(x,r)
:=\int_{B(0,R)}
|\mathbf{1}_{E_{x,r}}(z)-\mathbf{1}_{H_x}(z)|\,dz
\quad\text{and}\quad
\beta_R(x,r)
:=d_{\mathrm{R}}(\mu_{x,r},\mu_x).
\end{align*}
For $k\in\mathbb{N}$, define
\begin{align*}
a_{k,R}(x)
:=\sup_{0<r<\frac{1}{k}}\alpha_R(x,r)
\quad\text{and}\quad
b_{k,R}(x)
:=\sup_{0<r<\frac{1}{k}}\beta_R(x,r).
\end{align*}

We first prove continuity in the scale variable. Fix
$x\in\partial^*E\cap Q_0$ and let $r_i\to r$ as $i\to\infty$, where
$r\in(0,\infty)$. For all
sufficiently large $i$,
$
\frac{r}{2}\leq r_i\leq2r.
$
Let $K:=\overline{B(x,2rR)}$. Given $\eta>0$, choose
$g\in C_{\mathrm{c}}(\mathbb{R}^n)$ such that
$
\int_K|\mathbf{1}_E(y)-g(y)|\,dy<\eta,
$
and let
\begin{align*}
\omega_g(t):=\sup_{\{u,v\in\mathbb{R}^n:\ |u-v|\leq t\}}|g(u)-g(v)|.
\end{align*}
Since $g$ is uniformly continuous, $\omega_g(t)\to0$ as $t\to0^+$.
Using $\mathbf{1}_{E_{x,s}}(z)=\mathbf{1}_E(x+sz)$, adding and
subtracting $g$, and changing variables, we obtain
\begin{align*}
\int_{B(0,R)}
|\mathbf{1}_{E_{x,r_i}}(z)-\mathbf{1}_{E_{x,r}}(z)|\,dz
\leq
\left[\left(\frac{2}{r}\right)^n+r^{-n}\right]\eta
+|B(0,R)|\,\omega_g(R|r_i-r|).
\end{align*}
Letting first $i\to\infty$ and then $\eta\to0^+$ implies
\begin{align}\label{alpha-1}
|\alpha_R(x,r_i)-\alpha_R(x,r)|
&\leq\int_{B(0,R)}
|\mathbf{1}_{E_{x,r_i}}-\mathbf{1}_{E_{x,r}}|\,dz
\to0
\end{align}
as $i\to\infty$.

We next treat the measure error. For every
$\phi\in\mathcal{F}_R$, viewed as zero outside $B(0,R)$,
by the definition of the push-forward, we obtain, for every
$s\in(0,\infty)$,
\begin{align*}
\int_{\mathbb{R}^n}\phi(z)\,d\mu_{x,s}(z)
=s^{1-n}\int_{\mathbb{R}^n}
\phi\left(\frac{y-x}{s}\right)\,d\sigma_E(y).
\end{align*}
For all sufficiently large $i$, both integrands corresponding to
$s=r_i$ and $s=r$ vanish outside $K$. Using
$\|\phi\|_{L^\infty(\mathbb{R}^n)}\leq1$ and
$\operatorname{Lip}(\phi)\leq1$, we obtain
\begin{align*}
&\left|\int_{\mathbb{R}^n}\phi\,d\mu_{x,r_i}
-\int_{\mathbb{R}^n}\phi\,d\mu_{x,r}\right|\\
&\quad\leq
|r_i^{1-n}-r^{1-n}|\,\sigma_E(K)
+r^{1-n}
\left|\frac{1}{r_i}-\frac{1}{r}\right|
\int_K|y-x|\,d\sigma_E(y)\\
&\quad\leq C_{(n,r,R)}\sigma_E(K)|r_i-r|.
\end{align*}
Taking the supremum over $\phi\in\mathcal{F}_R$, we obtain
\begin{align*}
|\beta_R(x,r_i)-\beta_R(x,r)|
\leq d_{\mathrm{R}}(\mu_{x,r_i},\mu_{x,r})
\leq C_{(n,r,R)}\sigma_E(K)|r_i-r|\to0
\end{align*}
as $i\to\infty$.
By this and \eqref{alpha-1}, we find that
both $r\mapsto\alpha_R(x,r)$ and
$r\mapsto\beta_R(x,r)$ are continuous on $(0,\infty)$. Consequently,
\begin{align*}
a_{k,R}(x)
=\sup_{r\in\mathbb{Q}\cap(0,\frac{1}{k})}\alpha_R(x,r)
\quad\text{and}\quad
b_{k,R}(x)
=\sup_{r\in\mathbb{Q}\cap(0,\frac{1}{k})}\beta_R(x,r).
\end{align*}

It remains to prove that, for every fixed $r>0$, the maps
$x\mapsto\alpha_R(x,r)$ and $x\mapsto\beta_R(x,r)$ are Borel
measurable on $\partial^*E\cap Q_0$.
Since $E$ is Borel and
$\nu_E$ is Borel measurable, the map
\begin{align*}
(x,z)\longmapsto
\left|\mathbf{1}_E(x+rz)
-\mathbf{1}_{\{z\cdot\nu_E(x)<0\}}\right|
\end{align*}
is Borel on $(\partial^*E\cap Q_0)\times B(0,R)$.
Therefore, by parameter-dependent integration,
$x\mapsto\alpha_R(x,r)$ is Borel.

For the measure error, recall from Step 1 that $\mathcal{F}_R$ is totally
bounded in the uniform norm. Choose a countable uniformly dense subset
$\{\phi_j\}_{j\in\mathbb{N}}\subset\mathcal{F}_R$. Since
\begin{align*}
\left|\int_{B(0,R)}(\phi-\psi)\,d(\mu_{x,r}-\mu_x)\right|
\leq\|\phi-\psi\|_{L^\infty(\mathbb{R}^n)}
\bigl[\mu_{x,r}(B(0,R))+\mu_x(B(0,R))\bigr],
\end{align*}
uniform density implies
\begin{align*}
\beta_R(x,r)
&=\sup_{j\in\mathbb{N}}\left|r^{1-n}\int_{\mathbb{R}^n}
\phi_j\left(\frac{y-x}{r}\right)\,d\sigma_E(y)
-\int_{\nu_E(x)^\perp}
\phi_j(z)\,d\mathcal{H}^{n-1}(z)\right|\\
&=:\sup_{j\in\mathbb{N}}|I_{j,r}(x)-J_j(x)|.
\end{align*}
Here and thereafter,
$\nu^\perp:=\{z\in\mathbb{R}^n:z\cdot\nu=0\}$ for
$\nu\in\mathbb{S}^{n-1}$; in particular,
$\nu_E(x)^\perp=\partial H_x$.
For fixed $j$ and $r$, the map $x\mapsto I_{j,r}(x)$ is continuous by
dominated convergence and the finiteness of $\sigma_E$. To treat $J_j$,
define, for every $\nu\in\mathbb{S}^{n-1}$,
\begin{align*}
\Gamma_j(\nu):=
\int_{\nu^\perp}\phi_j(z)\,d\mathcal{H}^{n-1}(z).
\end{align*}
Here $I$ denotes the identity map on $\mathbb{R}^n$, and, for
$A\in\mathbb{R}^{n\times n}$,
$$
\|A\|_{\mathrm{op}}:=\sup_{\{z\in\mathbb{R}^n:\ |z|=1\}}|Az|
$$
denotes the operator norm induced by the Euclidean norm.
Let $\{\nu_m\}_{m\in\mathbb{N}}\subset\mathbb{S}^{n-1}$ satisfy
$\nu_m\to\nu$ as $m\to\infty$, and choose
$\{V_m\}_{m\in\mathbb{N}}\subset O(n)$ such that
$V_m\nu=\nu_m$ and $V_m\to I$ as $m\to\infty$. Then
\begin{align*}
|\Gamma_j(\nu_m)-\Gamma_j(\nu)|
&\leq\int_{\nu^\perp\cap B(0,R)}
|\phi_j(V_mz)-\phi_j(z)|\,d\mathcal{H}^{n-1}(z)\\
&\leq C_{(n)}R^n\operatorname{Lip}(\phi_j)
\|V_m-I\|_{\mathrm{op}}\to0
\end{align*}
as $m\to\infty$.
Thus $\Gamma_j$ is continuous and
$J_j=\Gamma_j\circ\nu_E$ is Borel. It follows that
$x\mapsto\beta_R(x,r)$ is Borel as a countable supremum of Borel
functions. The preceding rational-scale formulas now show that
$a_{k,R}$ and $b_{k,R}$ are Borel.

Finally, \eqref{eq-new-pointwise-set-blowup} and
\eqref{eq-new-pointwise-dR-blowup} imply
$
a_{k,R}(x)\to0$ and
$
b_{k,R}(x)\to0
$
as $k\to\infty$ for $\sigma_E$-almost every
$x\in\partial^*E\cap Q_0$.

\textbf{Step 3.} \noindent\textit{In this step, we make the pointwise
convergence uniform on a large compact set.}
By Step 2 and the countability of $\mathbb{N}$, after discarding a
$\sigma_E$-null set we may assume that
$a_{k,R}(x)\to0$ and $b_{k,R}(x)\to0$ as $k\to\infty$ for every integer
$R\geq1$ and every remaining point $x$. By Lusin's theorem, we can
choose a measurable set
$L\subset\partial^*E\cap Q_0$ such that
$
\sigma_E(Q_0\setminus L)<\frac{\delta}{4},
$
and the restriction $\nu_E|_L$ is continuous. The preceding convergences
hold at every point of $L$.

For each $R\in\mathbb{N}$, apply Egorov's theorem on $L$ simultaneously
to the sequences $\{a_{k,R}\}_{k\in\mathbb{N}}$ and
$\{b_{k,R}\}_{k\in\mathbb{N}}$. We obtain a measurable set
$C_R\subset L$ such that both convergences are uniform on $C_R$ and
$
\sigma_E(L\setminus C_R)<\delta2^{-R-2}.
$
Let
$
L_0:=L\cap\bigcap_{R=1}^{\infty}C_R.
$
Then
\begin{align*}
\sigma_E(Q_0\setminus L_0)
&\leq\sigma_E(Q_0\setminus L)
+\sum_{R=1}^{\infty}\sigma_E(L\setminus C_R)\\
&<\frac{\delta}{4}
+\sum_{R=1}^{\infty}\delta2^{-R-2}
=\frac{\delta}{2}.
\end{align*}
By inner regularity, choose a compact set $G\subset L_0$ such that
$
\sigma_E(L_0\setminus G)<\frac{\delta}{2}.
$
It follows that
$\sigma_E(Q_0\setminus G)<\delta$, and $\nu_E|_G$ is continuous.

Moreover, for every $R\in\mathbb{N}$,
\begin{align*}
\sup_{x\in G}a_{k,R}(x)
+\sup_{x\in G}b_{k,R}(x)\to0
\end{align*}
as $k\to\infty$.
Here, as in the statement, both suprema are understood to be zero when
$G=\emptyset$.
By the definitions of $a_{k,R}$ and $b_{k,R}$, these are precisely
\eqref{eq-new-uniform-set-blowup} and
\eqref{eq-new-uniform-measure-blowup} along
$r_0=\frac{1}{k}$. Monotonicity in
$r_0$ implies the full limit as $r_0\to0^+$,
and monotonicity in $R$
extends the conclusions from integer radii
to every $R>0$. This completes the proof of Lemma
\ref{lem-new-uniform-blowup}.
\end{proof}

The localization arguments require these convergences on cubes that may
move and rotate with the scale. We record this consequence separately.

\begin{lemma}
\label{lem-new-moving-cube-blowup}
Let $E$, $\delta$, and $G$ be as in Lemma
\ref{lem-new-uniform-blowup}. Let
$\{x_m\}_{m\in\mathbb{N}}\subset G$ and $x\in G$ satisfy
$x_m\to x$ as $m\to\infty$, and let
$\{r_m\}_{m\in\mathbb{N}}\subset(0,\infty)$ satisfy
$r_m\to0^+$ as $m\to\infty$. Let
$\{S_m\}_{m\in\mathbb{N}}$ be a sequence of arbitrarily oriented unit
cubes in $\mathbb{R}^n$, and let $S$ be an arbitrarily oriented unit cube
in $\mathbb{R}^n$ such that
$d_{\mathrm{H}}(S_m,S)\to0$ as $m\to\infty$. Then
\begin{align}\label{eq-new-moving-cube-set}
\int_{S_m}
|\mathbf{1}_{E_{x_m,r_m}}(z)-\mathbf{1}_{H_x}(z)|\,dz
\to0
\end{align}
as $m\to\infty$.
If, in addition, $\mu_x(\partial S)=0$, then
\begin{align}\label{eq-new-moving-cube-measure}
\mu_{x_m,r_m}(S_m^\circ)\to\mu_x(S)
\end{align}
as $m\to\infty$.
\end{lemma}

\begin{proof}
Choose $R_0\in(0,\infty)$ such that
$S_m\cup S\subset B(0,R_0)$ for all sufficiently
large $m$. The continuity of $\nu_E$ on $G$ implies
$\mathbf{1}_{H_{x_m}}\to\mathbf{1}_{H_x}$ in
$L^1(B(0,R_0))$ as $m\to\infty$. Hence, by
\eqref{eq-new-uniform-set-blowup},
\begin{align*}
&\int_{S_m}|\mathbf{1}_{E_{x_m,r_m}}-
\mathbf{1}_{H_x}|\,dz\\
&\quad\leq
\int_{B(0,R_0)}|\mathbf{1}_{E_{x_m,r_m}}-
\mathbf{1}_{H_{x_m}}|\,dz
+\int_{B(0,R_0)}|\mathbf{1}_{H_{x_m}}-
\mathbf{1}_{H_x}|\,dz
\to0
\end{align*}
as $m\to\infty$,
which is \eqref{eq-new-moving-cube-set}.

Now assume that $\mu_x(\partial S)=0$.
For the measures, choose rotations $V_m\to I$ as $m\to\infty$ such that
\begin{align*}
V_m\nu_E(x)=\nu_E(x_m).
\end{align*}
Such rotations exist because $\nu_E(x_m)\to\nu_E(x)$. For all
large $m$ for which the two normals differ, take the rotation that acts in
the plane they generate, sends $\nu_E(x)$ to $\nu_E(x_m)$, and is the
identity on its orthogonal complement; take $V_m=I$ when the normals
agree, and choose the remaining finitely many rotations arbitrarily. The
rotation angles tend to zero, and hence $V_m\to I$.
Then $\mu_{x_m}=(V_m)_\#\mu_x$. For every test function $\phi$ in the
unit ball defining $d_{\mathrm{R}}$, extended by zero outside $B(0,R)$, we have
\begin{align*}
\left|\int_{\mathbb{R}^n}\phi(z)\,d\mu_{x_m}(z)
-\int_{\mathbb{R}^n}\phi(z)\,d\mu_x(z)\right|
&=\left|\int_{\partial H_x}
[\phi(V_mz)-\phi(z)]\,d\mathcal{H}^{n-1}(z)\right|\\
&\leq\|V_m-I\|_{\mathrm{op}}
\int_{\partial H_x\cap B(0,R)}
|z|\,d\mathcal{H}^{n-1}(z)\\
&\leq C_{(n)}R^n\|V_m-I\|_{\mathrm{op}}.
\end{align*}
Taking the supremum over $\phi$ implies
$
d_{\mathrm{R}}(\mu_{x_m},\mu_x)\to0
$
as $m\to\infty$ for every $R>0$. Since $x_m\in G$ and
$r_m\to0$ as $m\to\infty$,
\eqref{eq-new-uniform-measure-blowup} also implies
$
d_{\mathrm{R}}(\mu_{x_m,r_m},\mu_{x_m})\to0
$
as $m\to\infty$.
Consequently,
\begin{align*}
d_{\mathrm{R}}(\mu_{x_m,r_m},\mu_x)
&\leq d_{\mathrm{R}}(\mu_{x_m,r_m},\mu_{x_m})
+d_{\mathrm{R}}(\mu_{x_m},\mu_x)
\to0
\end{align*}
as $m\to\infty$ for every $R>0$. To obtain local weak-$*$ convergence, fix
$\psi\in C_{\mathrm c}(B(0,R))$ and choose
$\eta\in C_{\mathrm c}(B(0,R+1))$ such that
$0\leq\eta\leq1$ and $\eta=1$ on $B(0,R)$. The convergence in
$d_{\mathrm R}$, applied to a normalization of $\eta$, implies
$M_R:=\sup_m\mu_{x_m,r_m}(B(0,R))+\mu_x(B(0,R))<\infty$. If
$\{\psi_j\}_{j\in\mathbb{N}}
\subset C_{\mathrm c}(B(0,R))$
consists of Lipschitz functions and
$\|\psi_j-\psi\|_{L^\infty(\mathbb{R}^n)}\to0$
as $j\to\infty$, then
\begin{align*}
\limsup_{m\to\infty}
\left|\int_{\mathbb{R}^n}\psi\,d(\mu_{x_m,r_m}-\mu_x)\right|
&\leq\limsup_{m\to\infty}
\left|\int_{\mathbb{R}^n}\psi_j\,d(\mu_{x_m,r_m}-\mu_x)\right|
+M_R\|\psi_j-\psi\|_{L^\infty(\mathbb{R}^n)}\\
&=M_R\|\psi_j-\psi\|_{L^\infty(\mathbb{R}^n)}
\to0
\end{align*}
as $j\to\infty$.
Hence
$
\mu_{x_m,r_m}\stackrel{*}{\rightharpoonup}\mu_x$
locally in $\mathbb{R}^n$
as $m\to\infty$.

For $\rho>0$, let
\begin{align*}
S_{-\rho}
:=\{z\in S^\circ:
\operatorname{dist}(z,\partial S)>\rho\}
\quad\text{and}\quad
S^{+\rho}
:=\{z\in\mathbb{R}^n:
\operatorname{dist}(z,S)\leq\rho\}.
\end{align*}
Since $S_m$ and $S$ are congruent convex cubes, their Hausdorff
convergence implies that, for every fixed $\rho>0$ and all sufficiently
large $m$,
$
S_{-\rho}\subset S_m^\circ\subset S^{+\rho}.
$
The first set is open and relatively compact, while the second is
compact. The portmanteau theorem and the preceding inclusions therefore
imply
\begin{align*}
\mu_x(S_{-\rho})
&\leq\liminf_{m\to\infty}\mu_{x_m,r_m}(S_{-\rho})\\
&\leq\liminf_{m\to\infty}\mu_{x_m,r_m}(S_m^\circ)
\leq\limsup_{m\to\infty}\mu_{x_m,r_m}(S_m^\circ)\\
&\leq\limsup_{m\to\infty}\mu_{x_m,r_m}(S^{+\rho})
\leq\mu_x(S^{+\rho}).
\end{align*}
As $\rho\to0^+$, the sets $S_{-\rho}$ increase to $S^\circ$ and the
sets $S^{+\rho}$ decrease to $S$. Since $\mu_x(\partial S)=0$, both
outer terms converge to $\mu_x(S)$. Hence
\begin{align*}
\lim_{m\to\infty}\mu_{x_m,r_m}(S_m^\circ)=\mu_x(S),
\end{align*}
which proves \eqref{eq-new-moving-cube-measure} and completes the proof
of Lemma \ref{lem-new-moving-cube-blowup}.
\end{proof}

\subsection{A Sharp Cube-Section Inequality}

In this subsection, we consider
the global upper bound in Theorem \ref{the-new-exact},
which reduces to a sharp
cell estimate below for a half-space in a unit cube.
The difficulty is to make this bound uniform in both the position
and orientation of the cut while retaining the sharp factor $\frac12$ in
the cell inequality.
We first state
the result and then reduce its geometric estimate to a one-dimensional
density bound.

\begin{proposition}
\label{prop-new-cube-cell}
Let $n\in\mathbb{N}\cap[2,\infty)$, let
$S\subset\mathbb{R}^n$ be a unit cube with arbitrary orientation, and let
$H$ be a half-space. Then
\begin{align}\label{eq-new-cube-section-strong}
\mathcal H^{n-1}(\partial H\cap S)
\geq
(4|H\cap S|\,|S\setminus H|)^{\frac{n-1}{n}}.
\end{align}
Moreover, for every $q\in[1,\frac{n}{n-1}]$,
\begin{align}\label{eq-new-cube-cell}
\left[
\fint_S|\mathbf{1}_H(x)-(\mathbf{1}_H)_S|^q\,dx
\right]^{\frac1q}
\leq\frac12\mathcal H^{n-1}(\partial H\cap S),
\end{align}
and the constant $\frac12$ in \eqref{eq-new-cube-cell} is optimal.
\end{proposition}

To obtain the geometric estimate, we project the uniform measure on the
cube onto a unit normal to the cutting hyperplane. The resulting density
function is the area of the corresponding parallel section; its
$\frac{1}{n-1}$ power is concave, while its second moment is
$\frac{1}{12}$ independently of the normal. The required lower bound is
therefore a consequence of the following one-dimensional lemma.

\begin{lemma}
\label{lem-new-quantile-density}
Let $n\in\mathbb{N}\cap[2,\infty)$ and $L\in(0,\infty)$. Let
$f:(-L,L)\to(0,\infty)$ be even and satisfy
$\int_{-L}^{L}f(t)\,dt=1$.
Assume that $f^{\frac{1}{n-1}}$ is concave
and that
\begin{align}\label{eq-new-quantile-second-moment-assumption}
\int_{-L}^{L}t^2f(t)\,dt=\frac{1}{12}.
\end{align}
Define $F(t):=\int_{-L}^{t}f(s)\,ds$, let
$\mathcal Q:=F^{-1}$ be the inverse function of $F$, and let
$\mathcal I(v):=f(\mathcal Q(v))$.
Then, for every $v\in(0,1)$,
\begin{align}\label{eq-new-quantile-density-lower}
\mathcal I(v)
\geq[4v(1-v)]^{\frac{n-1}{n}}.
\end{align}
\end{lemma}

\begin{proof}
The evenness of $f$ implies $\mathcal Q(1-v)=-\mathcal Q(v)$ and hence
$\mathcal I(1-v)=\mathcal I(v)$ for $v\in(0,1)$. Since the right-hand
side of \eqref{eq-new-quantile-density-lower} has the same symmetry, it
suffices to fix $m\in(0,\frac12]$.
We divide the proof into two steps.

\textbf{Step 1.} \noindent\textit{In this step, we derive a lower bound
for $\mathcal I(m)^2$.}
Define
$
J(v):=\mathcal I(v)^{\frac{n}{n-1}}.
$
If $h:=f^{\frac{1}{n-1}}$, then
$J(v)=h(\mathcal Q(v))^n$. Since $h$ is concave, its right derivative
$h'_+$ exists at every point of $(-L,L)$ and is non-increasing. Moreover,
$h$ is continuous, so $F'=f=h^{n-1}>0$ is continuous. The inverse
function theorem therefore implies
$\mathcal Q=F^{-1}\in C^1((0,1))$ and, for every $v\in(0,1)$,
\begin{align*}
\mathcal Q'(v)
=\frac{1}{F'(\mathcal Q(v))}
=\frac{1}{f(\mathcal Q(v))}
=\frac{1}{h(\mathcal Q(v))^{n-1}}.
\end{align*}
By the right-hand chain rule, we obtain,
for every $v\in(0,1)$,
\begin{align*}
J'_+(v)
=nh(\mathcal Q(v))^{n-1}h'_+(\mathcal Q(v))
\mathcal Q'(v)
=n h'_+(\mathcal Q(v)).
\end{align*}
Since $h'_+$ is non-increasing and $\mathcal Q$ is increasing,
$J'_+$ is non-increasing. Since $J$ is continuous,
by the one-dimensional
right-derivative criterion for convexity, applied to $-J$,
\cite[Theorem 5.3.1, p.\,34]{HiriartUrrutyLemarechal1993},
we find that $J$ is concave.
Its finite endpoint limit exists, and we extend it to zero by
$J(0):=J(0+)\geq0$; the extension remains concave. Moreover,
by the evenness of $f$, we have,
for every $v\in(0,1)$,
$
J(v)=J(1-v).
$
Thus $J$ is nondecreasing on $(0,\frac12]$. Using also the concavity of
$J$ and $J(0+)\geq0$, we obtain
\begin{align}\label{eq-new-J-upper-bounds}
J(r)\leq
\begin{cases}
J(m),&r\in(0,m],\\
\dfrac{rJ(m)}{m},&r\in[m,\tfrac12].
\end{cases}
\end{align}
Indeed, if $m\leq r\leq\frac12$, then the concavity of $J$ on
$[0,r]$ and $J(0)\geq0$ imply
\begin{align*}
J(m)&\geq\frac{m}{r}J(r)
+\left(1-\frac{m}{r}\right)J(0)
\geq\frac{m}{r}J(r),
\end{align*}
which is the second inequality in \eqref{eq-new-J-upper-bounds}.
The inverse-function identity and \eqref{eq-new-J-upper-bounds} imply,
for almost every
$r\in(0,\frac12)$,
\begin{align}\label{eq-new-quantile-derivative-lower}
\mathcal Q'(r)=\frac{1}{\mathcal I(r)}=J(r)^{-\frac{n-1}{n}}
\geq
\begin{cases}
J(m)^{-\frac{n-1}{n}},&r\in(0,m],\\
J(m)^{-\frac{n-1}{n}}
\left(\frac{m}{r}\right)^{\frac{n-1}{n}},&r\in[m,\frac12].
\end{cases}
\end{align}
The symmetry of $f$ implies $\mathcal Q(\frac{1}{2})=0$. Integrating
\eqref{eq-new-quantile-derivative-lower} from $s$ to $\frac{1}{2}$, we find that
\begin{align}\label{eq-new-quantile-integrated-lower}
|\mathcal Q(s)|
\geq J(m)^{-\frac{n-1}{n}}
\begin{cases}
n\,2^{-\frac{1}{n}}m^{\frac{n-1}{n}}-(n-1)m-s,
&s\in(0,m],\\[1mm]
n\,m^{\frac{n-1}{n}}
\left(2^{-\frac{1}{n}}-s^{\frac{1}{n}}\right),
&s\in[m,\frac12].
\end{cases}
\end{align}
Denote the integral of the square of the piecewise factor on the
right-hand side of \eqref{eq-new-quantile-integrated-lower} by
\begin{align}\label{eq-new-Knm}
K_n(m)
:=\int_0^m
\left[
n\,2^{-\frac{1}{n}}m^{\frac{n-1}{n}}-(n-1)m-s
\right]^2\,ds
+n^2m^{\frac{2(n-1)}{n}}
\int_m^{\frac12}
\left(2^{-\frac{1}{n}}-s^{\frac{1}{n}}\right)^2\,ds.
\end{align}
Squaring \eqref{eq-new-quantile-integrated-lower}, integrating, and using
\eqref{eq-new-Knm}, the identities $\mathcal Q(F(t))=t$ and $F'=f$,
symmetry, and \eqref{eq-new-quantile-second-moment-assumption}, we obtain
\begin{align*}
\frac{1}{12}
&=\int_{-L}^{L}t^2f(t)\,dt
=\int_{-L}^{L}\mathcal Q(F(t))^2F'(t)\,dt
=\int_0^1\mathcal Q(s)^2\,ds\\
&=2\int_0^{\frac{1}{2}}\mathcal Q(s)^2\,ds
\geq2J(m)^{-\frac{2(n-1)}{n}}K_n(m)
=2\mathcal{I}(m)^{-2}K_n(m).
\end{align*}
Thus
\begin{align}\label{eq-new-quantile-Knm-lower}
\mathcal I(m)^2\geq24K_n(m).
\end{align}
This is the desired estimate of this step.

\textbf{Step 2.} \noindent\textit{In this step, we prove the lower bound
\eqref{eq-new-barrier-energy-lower} for $K_n(m)$ and complete the proof.}
We claim that
\begin{align}\label{eq-new-barrier-energy-lower}
24K_n(m)\geq
[4m(1-m)]^{\frac{2(n-1)}{n}}.
\end{align}
For the algebraic estimate, let
$y:=(2m)^{\frac{1}{n}}\in(0,1]$; this converts the fractional powers of
$2m$ below into integer powers of $y$. Direct integration in
\eqref{eq-new-Knm} implies
\begin{align*}
24K_n(m)
=y^{2n-2}A_n(y),
\end{align*}
where
\begin{align*}
A_n(y)
:=\frac{6n^2}{(n+1)(n+2)}
-\frac{3n(n-1)}{n+1}y^{n+1}
+\frac{(3n-2)(n-1)}{n+2}y^{n+2}.
\end{align*}
Since
$
[4m(1-m)]^{\frac{2(n-1)}{n}}
=y^{2n-2}(2-y^n)^{\frac{2(n-1)}{n}},
$
to show \eqref{eq-new-barrier-energy-lower},
it remains to prove that, for every $y\in[0,1]$,
\begin{align}\label{eq-new-An-lower}
A_n(y)\geq(2-y^n)^{\frac{2(n-1)}{n}}.
\end{align}
For $n=2$, for every $y\in[0,1]$,
this follows immediately from
$
A_2(y)-(2-y^2)=y^2(1-y)^2\geq0.
$
Assume now that $n\geq3$ and define
$
D_n(y):=A_n(y)-(2-y^n)^{\frac{2(n-1)}{n}}.
$
A direct differentiation implies
$
D_n'(y)=(n-1)y^{n-1}g_n(y),
$
where
\begin{align*}
g_n(y):=2(2-y^n)^{\frac{n-2}{n}}-3ny+(3n-2)y^2.
\end{align*}
Moreover,
\begin{align*}
g_n'(y)
&=2(3n-2)y-3n
-2(n-2)y^{n-1}(2-y^n)^{-\frac{2}{n}},
\end{align*}
and
\begin{align*}
g_n''(y)
&=2(3n-2)-2(n-2)\left[
(n-1)y^{n-2}(2-y^n)^{-\frac{2}{n}}
+2y^{2n-2}(2-y^n)^{-\frac{n+2}{n}}
\right].
\end{align*}
Since both
$y^{n-2}(2-y^n)^{-\frac{2}{n}}$ and
$y^{2n-2}(2-y^n)^{-\frac{n+2}{n}}$
are strictly increasing on $[0,1]$,
the formula above shows that
$g_n''$ is strictly decreasing,
and hence $g_n'$ is
strictly concave. Since
$$
g_n'(0)=-3n<0<n=g_n'(1),
$$
$g_n'$ has a unique zero in $(0,1)$.
Thus $g_n$ first strictly decreases
and then strictly increases. Furthermore,
$
g_n(0)=2^{\frac{2(n-1)}{n}}>0$ and
$g_n(1)=0$.
Let $a\in(0,1)$ be the unique zero of $g_n'$. Since $g_n$ is
strictly increasing on $[a,1]$ and $g_n(1)=0$, one has $g_n(a)<0$ and
$g_n<0$ on $[a,1)$. On the other hand, $g_n$ is strictly decreasing on
$[0,a]$ and $g_n(0)>0>g_n(a)$. It therefore has a unique zero
$b\in(0,a)$, is positive on $[0,b)$, and is negative on $(b,1)$.
Therefore, $D_n$ first increases and then decreases, so its minimum on
$[0,1]$ is attained at an endpoint. Clearly, $D_n(1)=0$. On the other
hand, the weighted arithmetic--geometric mean inequality implies
$$
2^{\frac{n-2}{n}}
=1^{\frac{2}{n}}2^{\frac{n-2}{n}}
\leq\frac{2}{n}+\frac{2(n-2)}{n}
=\frac{2(n-1)}{n}.
$$
Hence
\begin{align*}
D_n(0)
=\frac{6n^2}{(n+1)(n+2)}
-2^{\frac{2(n-1)}{n}}
\geq\frac{2(n-2)(n^2-2n-2)}
{n(n+1)(n+2)}>0.
\end{align*}
We conclude that $D_n\geq0$ on $[0,1]$. This proves
\eqref{eq-new-An-lower} and hence \eqref{eq-new-barrier-energy-lower}.
Combining \eqref{eq-new-quantile-Knm-lower} and
\eqref{eq-new-barrier-energy-lower}, and taking square roots, proves
\eqref{eq-new-quantile-density-lower}.
This completes the proof of Lemma \ref{lem-new-quantile-density}.
\end{proof}

We now prove the cube-section estimate,
that is, Proposition \ref{prop-new-cube-cell}.

\begin{proof}[Proof of Proposition \ref{prop-new-cube-cell}]
Let $m:=\min\{|H\cap S|,|S\setminus H|\}$.
Since $|S|=1$, the two volumes sum to one, and hence
$m\in[0,\frac12]$.
The case $m=0$ is immediate, so assume that $m\in(0,\frac12]$.
After a rigid motion, we may take
$S=[-\frac12,\frac12]^n$. Choose a unit normal $u$ and
$t_0\in\mathbb{R}$ such that
\begin{align*}
H=\{x\in\mathbb{R}^n:x\cdot u<t_0\}.
\end{align*}
Let $X$ be uniformly distributed on $S$ and let $Y:=X\cdot u$.
By the coarea formula, the density of $Y$ is the
parallel-section function
\begin{align*}
f(t):=\mathcal H^{n-1}
\left(S\cap\{x\in\mathbb{R}^n:x\cdot u=t\}\right).
\end{align*}
Its support is $[-L,L]$, where
$L=\frac12\sum_{i=1}^{n}|u_i|$, and $f$ is strictly positive on
$(-L,L)$. Since $f$ is the density of $Y$,
$\int_{-L}^{L}f(t)\,dt=1$. It is even, and the Brunn--Minkowski theorem
(see, for instance, \cite[p.\,361]{Gardner2002}) shows that
$f^{\frac{1}{n-1}}$ is concave on $(-L,L)$. Moreover,
\begin{align*}
\int_{-L}^{L}t^2f(t)\,dt
=\mathbb E[(X\cdot u)^2]
=\frac{1}{12}\sum_{i=1}^{n}u_i^2
=\frac{1}{12}.
\end{align*}
Let $F$, $\mathcal Q$, and $\mathcal I$ be as in Lemma
\ref{lem-new-quantile-density}. Then $|H\cap S|=F(t_0)$ and
$
\mathcal H^{n-1}(\partial H\cap S)=f(t_0).
$
If $|H\cap S|\leq\frac12$, then $m=|H\cap S|=F(t_0)$ and
$t_0=\mathcal Q(m)$. If $|H\cap S|>\frac12$, the evenness of $f$ implies
$m=|S\setminus H|=1-F(t_0)=F(-t_0)$ and hence
$-t_0=\mathcal Q(m)$. In both cases,
$f(t_0)=\mathcal I(m)$. Therefore, Lemma
\ref{lem-new-quantile-density} implies
\begin{align*}
\mathcal H^{n-1}(\partial H\cap S)
=\mathcal I(m)
\geq[4m(1-m)]^{\frac{n-1}{n}}
=\left[4|H\cap S|\,|S\setminus H|\right]^{\frac{n-1}{n}},
\end{align*}
which proves \eqref{eq-new-cube-section-strong}.

It remains to derive the cell inequality \eqref{eq-new-cube-cell}.
Fix $q\in[1,\frac{n}{n-1}]$. By H\"older's inequality,
Lemma \ref{lem-new-characteristic-formula},
Jensen's inequality for the concave function
$t\mapsto t^{\frac{1}{n-1}}$ on $[0,1]$, and
\eqref{eq-new-cube-section-strong}, we find that
\begin{align*}
\left[
\fint_S|\mathbf{1}_H(x)-(\mathbf{1}_H)_S|^q\,dx
\right]^{\frac{1}{q}}
&\leq
\left[
\fint_S|\mathbf{1}_H(x)-(\mathbf{1}_H)_S|^{\frac{n}{n-1}}\,dx
\right]^{\frac{n-1}{n}}\\
&=\left[
m(1-m)^{\frac{n}{n-1}}+(1-m)m^{\frac{n}{n-1}}
\right]^{\frac{n-1}{n}}\\
&=\left\{
m(1-m)\left[(1-m)^{\frac{1}{n-1}}+m^{\frac{1}{n-1}}\right]
\right\}^{\frac{n-1}{n}}\\
&\leq\frac12[4m(1-m)]^{\frac{n-1}{n}}
\leq\frac12\mathcal H^{n-1}(\partial H\cap S).
\end{align*}
Here the proceeding inequality follows from the concavity of
$t\mapsto t^{\frac{1}{n-1}}$:
\begin{align*}
(1-m)^{\frac{1}{n-1}}+m^{\frac{1}{n-1}}
\leq 2^{1-\frac{1}{n-1}}.
\end{align*}
Finally, take $H$ to be a half-space whose boundary passes through the
center of $S$ and is parallel to one of its faces. Then
$(\mathbf{1}_H)_S=\frac12$,
$|\mathbf{1}_H-(\mathbf{1}_H)_S|=\frac12$ almost everywhere on $S$, and
$\mathcal H^{n-1}(\partial H\cap S)=1$. Hence equality holds in
\eqref{eq-new-cube-cell} for every $q\in[1,\frac{n}{n-1}]$, which proves
that the factor $\frac12$ is optimal.
This completes the proof of Proposition \ref{prop-new-cube-cell}.
\end{proof}

\begin{remark}\label{rem-new-cube-section-earlier}
We now discuss the relation
between Proposition \ref{prop-new-cube-cell}
and the earlier hyperplane-section inequalities
in \cite[Theorem 2, p.\,190]{Peri1999}.
For a centrally symmetric convex body $K$ cut by a hyperplane $\Pi$,
let $u$ be a unit normal to $\Pi$, and let $\Pi^+$ and $\Pi^-$ denote
the two half-spaces determined by $\Pi$. Peri
\cite[Theorem 2, p.\,190]{Peri1999} proved
\begin{align*}
|K\cap\Pi^+|\,|K\cap\Pi^-|
\leq
\left(\int_K|\langle u,x-c(K)\rangle|\,dx\right)
\mathcal H^{n-1}(K\cap\Pi),
\end{align*}
where $c(K)$ is the centroid of $K$ (and hence its center of symmetry in
the present setting). Taking $K=S$ and $\Pi=\partial H$, set
\begin{align*}
v:=|H\cap S|\quad\text{and}\quad
A:=\mathcal H^{n-1}(\partial H\cap S).
\end{align*}
The preceding inequality then gives
$v(1-v)\leq M_1(S,u)A$, where
\begin{align*}
M_1(S,u):=\int_S|\langle u,x-c(S)\rangle|\,dx.
\end{align*}
The sharp $L^1$--$L^2$ Khinchin inequality for sums of independent
uniformly distributed random variables
\cite[p. 197]{LatalaOleszkiewicz1995} gives
$M_1(S,u)\leq\frac14$. Indeed, let $e_1,\ldots,e_n$ be the orthonormal
edge directions of $S$, put $u_i:=\langle u,e_i\rangle$, and write
$x-c(S)=\frac12\sum_{i=1}^nt_ie_i$, where the $t_i$ are independent and
uniformly distributed on $[-1,1]$. Since $\sum_{i=1}^nu_i^2=1$, that
inequality at exponent one gives
\begin{align*}
M_1(S,u)
&=\frac12\mathbb E\left|\sum_{i=1}^nu_it_i\right|
\leq\frac{\sqrt3}{4}
\left(\mathbb E\left|\sum_{i=1}^nu_it_i\right|^2\right)^{\frac12}
=\frac14.
\end{align*}
Consequently, combining Peri's estimate with this Khinchin inequality gives
$4v(1-v)\leq A,$
which is equivalent to the case $q=1$
of \eqref{eq-new-cube-cell}. In
comparison, \eqref{eq-new-cube-section-strong}
gives the stronger nonlinear profile
$A\geq[4v(1-v)]^{\frac{n-1}{n}},$
and yields the full range $1\leq q\leq\frac{n}{n-1}$. Peri's later
moment inequalities
\cite[Corollary 2 and Theorem 3, pp.\,406--407]{Peri2008}
apply to general
convex bodies, but do not directly yield
this nonlinear cube profile or the
$q>1$ cell estimates. Conversely, Proposition
\ref{prop-new-cube-cell} is specific to cubes.
\end{remark}

\subsection{The Upper Bound via Half-Space Localization}

In this subsection, we establish the
upper bound in Theorem \ref{the-new-exact}
for $q\in[1,\frac{n}{n-1}]$.
The local estimate \eqref{eq-new-cube-cell} has exactly the form needed for
the blow-up argument. We now combine Proposition
\ref{prop-new-cube-cell} with the uniform blow-up and moving-cube lemmas in
Subsection \ref{5.3} to transfer it to arbitrary sets of finite perimeter.
The factor $\frac12$ is unchanged in this passage.

\begin{proposition}
\label{prop-new-localization}
Let $n\in\mathbb{N}\cap[2,\infty)$.
Let $q\in[1,\frac{n}{n-1}]$. Then every measurable set
$E\subset Q_0$ with finite perimeter satisfies
\begin{align}\label{eq-new-localization-conclusion}
\limsup_{\varepsilon\to0^+}I_{\varepsilon,q}^{*}(E)
\leq\frac12P(E,Q_0).
\end{align}
\end{proposition}

\begin{proof}
Throughout the proof, for every measurable set
$E$, $\sigma_E$ denotes the measure defined in
\eqref{eq-new-relative-variation-extension}. Thus $\sigma_E$ records only
the variation inside $Q_0$. It does not include the boundary variation on
$\partial Q_0$ created by extending $\mathbf{1}_E$ by zero outside
$Q_0$. Fix $\eta,\tau,\delta\in(0,1)$,
and let $G$ be the compact set supplied
by Lemma \ref{lem-new-uniform-blowup}.
Then $\sigma_E(Q_0\setminus G)<\delta$.
We divide the proof into two steps.

\textbf{Step 1.} \noindent\textit{In this step, we show the following
uniform local estimate of the oscillation in terms of the perimeter
measure:} there exists $r_0\in(0,1)$ such that, for every
$\varepsilon\in(0,r_0)$, every oriented cube
$Q\subset Q_0$ of edge length $\varepsilon$, and
every $x\in Q^\circ\cap G$ satisfying
\begin{align}\label{eq-new-local-tau-detail}
\operatorname{osc}_q(\mathbf{1}_E;Q)\geq\tau,
\end{align}
it holds that
\begin{align}\label{eq-new-local-sharp-estimate-detail}
\varepsilon^{n-1}\operatorname{osc}_q(\mathbf{1}_E;Q)
\leq\left(\frac12+\eta\right)\sigma_E(Q^\circ).
\end{align}
Suppose, to the contrary, that this assertion fails. Then there exist
$\{\varepsilon_m\}_{m\in\mathbb{N}}\subset (0,1)$ with
$\varepsilon_m\to0^+$ as $m\to\infty$, points
$\{x_m\}_{m\in\mathbb{N}}\subset G$, and oriented cubes
$\{Q_m\}_{m\in\mathbb{N}}$
such that $x_m\in Q_m^\circ$,
$\ell(Q_m)=\varepsilon_m$, and
\eqref{eq-new-local-tau-detail} holds and
\begin{align}\label{eq-new-local-contradiction-detail}
\varepsilon_m^{n-1}\operatorname{osc}_q(\mathbf{1}_E;Q_m)
>\left(\frac12+\eta\right)\sigma_E(Q_m^\circ).
\end{align}
Let $S_0=[-\frac12,\frac12]^n$, $y_m$ be the center of $Q_m$, and
choose $R_m\in O(n)$ such that
$
Q_m=y_m+\varepsilon_mR_mS_0.
$
Let
$
a_m:=\frac{y_m-x_m}{\varepsilon_m}.
$
Then
$
S_m:=\frac{Q_m-x_m}{\varepsilon_m}
=a_m+R_mS_0.
$
Since $x_m\in Q_m$, we have
$
|a_m|\leq\frac{\sqrt n}{2}.
$
Thus $a_m$ belongs to the compact set
$\overline{B(0,\frac{\sqrt n}{2})}$. Moreover, $O(n)$ and $G$ are
compact. Then after passing to a subsequence,
not relabeled, we may therefore
assume that
\begin{align*}
a_m\to a,\quad R_m\to R\in O(n),
\quad\text{and}\quad x_m\to x\in G
\end{align*}
as $m\to\infty$.
It follows that
$d_{\mathrm H}(S_m,S)\to0$ as $m\to\infty$, where the unit cube $S$ is
given by $S:=a+RS_0$.
Since every $S_m$ contains the origin, so does $S$.
Let
$
\theta_m:=\frac{|E\cap Q_m|}{|Q_m|}
$
and
$
\theta:=|H_x\cap S|.
$
Since $|S_m|=|S|=1$, a change of variables, the triangle inequality, and
$0\leq\mathbf{1}_{H_x}\leq1$ imply
\begin{align*}
|\theta_m-\theta|
&=\left|
\int_{S_m}\mathbf{1}_{E_{x_m,\varepsilon_m}}(z)\,dz
-\int_S\mathbf{1}_{H_x}(z)\,dz
\right|\\
&\leq
\int_{S_m}
\left|\mathbf{1}_{E_{x_m,\varepsilon_m}}(z)
-\mathbf{1}_{H_x}(z)\right|\,dz
+\left|
\int_{S_m}\mathbf{1}_{H_x}(z)\,dz
-\int_S\mathbf{1}_{H_x}(z)\,dz
\right|\\
&\leq
\int_{S_m}
\left|\mathbf{1}_{E_{x_m,\varepsilon_m}}(z)
-\mathbf{1}_{H_x}(z)\right|\,dz
+\int_{S_m\mathbin{\triangle}S}\mathbf{1}_{H_x}(z)\,dz\\
&\leq
\int_{S_m}
\left|\mathbf{1}_{E_{x_m,\varepsilon_m}}(z)
-\mathbf{1}_{H_x}(z)\right|\,dz
+|S_m\mathbin{\triangle}S|.
\end{align*}
The first term tends to zero as $m\to\infty$ by
\eqref{eq-new-moving-cube-set} in
Lemma \ref{lem-new-moving-cube-blowup}, applied with
$r_m=\varepsilon_m$. Moreover, $a_m\to a$ and $R_m\to R$ imply
$|S_m\mathbin{\triangle}S|\to0$, and hence $\theta_m\to\theta$, as
$m\to\infty$.
Therefore, by this and Lemma
\ref{lem-new-characteristic-formula}, we find that
\begin{align*}
&\lim_{m\to\infty}\left[
\fint_{Q_m}|\mathbf{1}_E(y)-(\mathbf{1}_E)_{Q_m}|^q\,dy
\right]^{\frac{1}{q}}\\
&\quad=\lim_{m\to\infty}
[\theta_m(1-\theta_m)^q+(1-\theta_m)\theta_m^q]^{\frac{1}{q}}
=[\theta(1-\theta)^q+(1-\theta)\theta^q]^{\frac{1}{q}}\\
&\quad=\left[
\fint_S|\mathbf{1}_{H_x}(z)-(\mathbf{1}_{H_x})_S|^q\,dz
\right]^{\frac{1}{q}}.
\end{align*}
The lower bound in \eqref{eq-new-local-tau-detail} shows that the last
quantity is at least $\tau>0$. Since the middle expression in the
preceding display vanishes exactly when
$\frac{|H_x\cap S|}{|S|}\in\{0,1\}$, we conclude that
$\frac{|H_x\cap S|}{|S|}\in(0,1)$.
In particular,
\begin{align}\label{eq-new-limit-section-positive}
\mathcal{H}^{n-1}(\partial H_x\cap S)>0.
\end{align}
Each face of $S$ lies in a hyperplane. Two distinct hyperplanes meet in
a set of dimension at most $n-2$, so their intersection has zero
$\mathcal{H}^{n-1}$-measure. The only remaining possibility for
$\mathcal{H}^{n-1}(\partial H_x\cap\partial S)>0$ is that
$\partial H_x$ coincide with the hyperplane containing a face of $S$.
Since the half-space boundary then supports the cube from one side, this
would imply $|H_x\cap S|=0$ or $|S|$, contrary to
the result obtained above that
$\frac{|H_x\cap S|}{|S|}\in(0,1)$. Hence
$
\mathcal{H}^{n-1}(\partial H_x\cap\partial S)=0.
$
The convergences $x_m\to x$, $\varepsilon_m\to0$, and
$d_{\mathrm H}(S_m,S)\to0$ established above, together with
$\mu_x(\partial S)
=\mathcal{H}^{n-1}(\partial H_x\cap\partial S)=0$, verify the
hypotheses of \eqref{eq-new-moving-cube-measure} in Lemma
\ref{lem-new-moving-cube-blowup} with $r_m=\varepsilon_m$.
Using $Q_m^\circ=x_m+\varepsilon_mS_m^\circ$, the definition of the
push-forward measure, the definition of $\mu_{x_m,\varepsilon_m}$,
\eqref{eq-new-moving-cube-measure}, and the definition of $\mu_x$, we
obtain
\begin{align*}
\varepsilon_m^{1-n}\sigma_E(Q_m^\circ)
&=\varepsilon_m^{1-n}\sigma_E
\left(x_m+\varepsilon_mS_m^\circ\right)\\
&=\varepsilon_m^{1-n}\sigma_E\left(
\left\{y\in\mathbb R^n:
\frac{y-x_m}{\varepsilon_m}\in S_m^\circ\right\}\right)\\
&=\varepsilon_m^{1-n}
\left[\left(y\mapsto\frac{y-x_m}{\varepsilon_m}\right)_{\#}
\sigma_E\right](S_m^\circ)
=\mu_{x_m,\varepsilon_m}(S_m^\circ)\\
&\to\mu_x(S)
=\mathcal{H}^{n-1}(\partial H_x\cap S)
\end{align*}
as $m\to\infty$.
Dividing \eqref{eq-new-local-contradiction-detail} by
$\varepsilon_m^{n-1}$ and passing to the limit, we obtain
\begin{align*}
\operatorname{osc}_q(\mathbf{1}_{H_x};S)
\geq\left(\frac12+\eta\right)
\mathcal{H}^{n-1}(\partial H_x\cap S).
\end{align*}
Combining this inequality with \eqref{eq-new-cube-cell} and
\eqref{eq-new-limit-section-positive}, we obtain
\begin{align*}
0
&<\eta\mathcal{H}^{n-1}(\partial H_x\cap S)\\
&=\left(\frac12+\eta\right)
\mathcal{H}^{n-1}(\partial H_x\cap S)
-\frac12\mathcal{H}^{n-1}(\partial H_x\cap S)\\
&\leq
\operatorname{osc}_q(\mathbf{1}_{H_x};S)
-\frac12\mathcal{H}^{n-1}(\partial H_x\cap S)
\leq0,
\end{align*}
which is impossible. Therefore,
\eqref{eq-new-local-sharp-estimate-detail} holds.

\textbf{Step 2.} \noindent\textit{In this step, we prove the upper bound
\eqref{eq-new-localization-conclusion}.}
Let $\varepsilon\in(0,r_0)$ and fix an arbitrary family
$
\{Q_j\}_{j\in J}\in\mathscr{F}_{\varepsilon,1}^{*}(Q_0).
$
We divide $J$ into the following three disjoint classes:
$$
J_1
:=\left\{j\in J:\ \operatorname{osc}_q(\mathbf{1}_E;Q_j)<\tau\right\},
$$
$$
J_2
:=\left\{j\in J:\ \operatorname{osc}_q(\mathbf{1}_E;Q_j)\geq\tau,
\ Q_j^\circ\cap G\ne\emptyset\right\},
$$
and
$$
J_3
:=\left\{j\in J:\ \operatorname{osc}_q(\mathbf{1}_E;Q_j)\geq\tau,
\ Q_j^\circ\cap G=\emptyset\right\}.
$$

For $j\in J_1$, the defining inequality for $J_1$, together with the
cardinality bound
$\#J_1\leq\#J\leq\varepsilon^{1-n}$ from Definition
\ref{def-new-perimeter-isotropic} implies
\begin{align}\label{J1}
\varepsilon^{n-1}\sum_{j\in J_1}
\operatorname{osc}_q(\mathbf{1}_E;Q_j)
\leq\tau\varepsilon^{n-1}\#J_1
\leq\tau\varepsilon^{n-1}\#J
\leq\tau\varepsilon^{n-1}\varepsilon^{1-n}
=\tau.
\end{align}
For each $j\in J_2$, choose
$x_j\in Q_j^\circ\cap G$. Since $\varepsilon\in(0,r_0)$ and the
defining inequality for $J_2$ implies the lower bound
\eqref{eq-new-local-tau-detail}, we may apply
\eqref{eq-new-local-sharp-estimate-detail} to $Q_j$ and $x_j$.
Using this estimate, the pairwise disjointness of the interiors
$Q_j^\circ$, their inclusion in $Q_0$, and
$\sigma_E(Q_0)=P(E,Q_0)$ from
\eqref{eq-new-relative-variation-extension}, we obtain
\begin{align}\label{J2}
&\varepsilon^{n-1}\sum_{j\in J_2}
\operatorname{osc}_q(\mathbf{1}_E;Q_j)\nonumber\\
&\quad\leq\left(\frac12+\eta\right)
\sum_{j\in J_2}\sigma_E(Q_j^\circ)
=\left(\frac12+\eta\right)
\sigma_E\left(\bigcup_{j\in J_2}Q_j^\circ\right)\nonumber\\
&\quad\leq\left(\frac12+\eta\right)\sigma_E(Q_0)
=\left(\frac12+\eta\right)P(E,Q_0).
\end{align}

For the estimate of $J_3$, the choice of $G$ ensures
$\sigma_E(Q_0\setminus G)<\delta$.
For $j\in J_3$, let
$$
\theta_j:=\frac{|E\cap Q_j|}{|Q_j|}
\quad\hbox{and}\quad
m_j:=\min\{\theta_j,1-\theta_j\}.
$$
From the exact
formula and the two-sided estimate in Lemma
\ref{lem-new-characteristic-formula},
the fact $m_j^{\frac{1}{q}}\leq m_j^{\frac{n-1}{n}}$ because
$m_j\in[0,\frac12]$ and
$q\leq\frac{n}{n-1}$, $|Q_j|=\varepsilon^n$, Lemma
\ref{lem-new-relative-isoperimetric},
\eqref{eq-new-relative-variation-extension}, the
pairwise disjointness of the interiors $Q_j^\circ$, the inclusion
$
\bigcup_{j\in J_3}Q_j^\circ\subset Q_0\setminus G,
$
and $\sigma_E(Q_0\setminus G)<\delta$, we infer that
\begin{align}\label{J3}
&\varepsilon^{n-1}\sum_{j\in J_3}
\left[\fint_{Q_j}
|\mathbf{1}_E(x)-(\mathbf{1}_E)_{Q_j}|^q\,dx
\right]^{\frac{1}{q}}\nonumber\\
&\quad=\varepsilon^{n-1}\sum_{j\in J_3}
\left[\theta_j(1-\theta_j)^q
+(1-\theta_j)\theta_j^q\right]^{\frac{1}{q}}\nonumber\\
&\quad\leq C_{(q)}\varepsilon^{n-1}
\sum_{j\in J_3}m_j^{\frac1q}
\leq C_{(q)}\varepsilon^{n-1}
\sum_{j\in J_3}m_j^{\frac{n-1}{n}}\nonumber\\
&\quad=C_{(q)}\sum_{j\in J_3}
\min\{|E\cap Q_j|,|Q_j\setminus E|\}^{\frac{n-1}{n}}\nonumber\\
&\quad\leq C_{(n,q)}\sum_{j\in J_3}P(E,Q_j)
=C_{(n,q)}\sum_{j\in J_3}\sigma_E(Q_j^\circ)\nonumber\\
&\quad=C_{(n,q)}
\sigma_E\left(\bigcup_{j\in J_3}Q_j^\circ\right)
\leq C_{(n,q)}\sigma_E(Q_0\setminus G)
\leq C_{(n,q)}\delta.
\end{align}

Since the admissible family fixed above was arbitrary and the three
bounds are independent of that family, combining
\eqref{J1}, \eqref{J2}, \eqref{J3},
\eqref{eq-new-I-general}, and \eqref{eq-new-I-unit},
we find that, for every
$\varepsilon\in(0,r_0)$,
\begin{align*}
I_{\varepsilon,q}^{*}(E)
&=\sup_{\{Q_j\}_{j\in J}\in
\mathscr{F}_{\varepsilon,1}^{*}(Q_0)}
\varepsilon^{n-1}\sum_{j\in J}
\operatorname{osc}_q(\mathbf{1}_E;Q_j)\\
&=\sup_{\{Q_j\}_{j\in J}\in
\mathscr{F}_{\varepsilon,1}^{*}(Q_0)}
\sum_{i=1}^3\varepsilon^{n-1}\sum_{j\in J_i}
\operatorname{osc}_q(\mathbf{1}_E;Q_j)\\
&\leq\tau+\left(\frac12+\eta\right)P(E,Q_0)
+C_{(n,q)}\delta.
\end{align*}
Although $r_0$ may depend on $\eta$, $\tau$, and $\delta$, the preceding
estimate holds on a punctured neighborhood of zero for each fixed
triple $(\eta,\tau,\delta)\in(0,1)^3$. We may therefore first take the
upper limit as $\varepsilon\to0^+$ and then take the infimum over all
such triples to obtain
\begin{align*}
\limsup_{\varepsilon\to0^+}I_{\varepsilon,q}^{*}(E)
\leq\inf_{\eta,\tau,\delta\in(0,1)}
\left\{\tau+\left(\frac12+\eta\right)P(E,Q_0)
+C_{(n,q)}\delta\right\}
=\frac12P(E,Q_0),
\end{align*}
which proves \eqref{eq-new-localization-conclusion}.
This completes the proof of Proposition \ref{prop-new-localization}.
\end{proof}

\subsection{The Planar Upper Bound Beyond $q=2$}

Proposition \ref{prop-new-localization}
establishes the sharp upper bound in
every dimension $n\geq2$ for $q\in[1,\frac{n}{n-1}]$, and hence in the
case $n=2$ for
$q\in[1,2]$. We nevertheless record a direct planar proof, which makes both
the geometry and the equality case transparent and avoids the quantile
argument. We then
average the cell estimate over small common translations and transfer the
resulting bound to arbitrary planar sets, thereby extending the upper bound
beyond $q=2$.

\begin{lemma}\label{lem-new-planar-cell}
Let $S\subset\mathbb{R}^2$ be a unit square with arbitrary orientation
and let $H$ be a half-plane. Then, for every $q\in[1,2]$,
\begin{align}\label{eq-new-planar-cell-q}
\left[
\fint_S|\mathbf{1}_H(x)-(\mathbf{1}_H)_S|^q\,dx
\right]^{\frac{1}{q}}
\leq\frac{1}{2}\mathcal{H}^1(\partial H\cap S),
\end{align}
where the constant $\frac12$ is sharp.
\end{lemma}

\begin{proof}
Fix $q\in[1,2]$ and let $\theta:=|H\cap S|$. By the monotonicity of
the normalized $L^p$ norms and Lemma
\ref{lem-new-characteristic-formula}, applied with exponent $2$, we have
\begin{align*}
\left[\fint_S|\mathbf{1}_H(x)-(\mathbf{1}_H)_S|^q\,dx\right]^{\frac{1}{q}}
&\leq
\left[\fint_S|\mathbf{1}_H(x)-(\mathbf{1}_H)_S|^2\,dx\right]^{\frac{1}{2}}\\
&=\left[\theta(1-\theta)^2+(1-\theta)\theta^2\right]^{\frac{1}{2}}
=[\theta(1-\theta)]^{\frac{1}{2}}.
\end{align*}
It remains to prove
$[\theta(1-\theta)]^{\frac{1}{2}}\leq
\frac{1}{2}\mathcal H^1(\partial H\cap S)$. The cases
$\theta\in\{0,1\}$ are immediate, so assume $\theta\in(0,1)$. The remaining
inequality is unchanged if $H$ is replaced by the half-plane on the other
side of $\partial H$, since this sends $\theta$ to $1-\theta$ and leaves
$\partial H$ unchanged. Thus we may assume $\theta\leq\frac12$.
After a rigid motion, we may assume $S=[0,1]^2$. Let $(\alpha,\beta)$ be
the outward unit normal to $\partial H$. Then, for some $\tau\in\mathbb R$,
$H=\{(x,y):\alpha x+\beta y<\tau\}$
and $\alpha^2+\beta^2=1$.
The symmetries
\begin{align*}
(x,y)\mapsto(1-x,y),
\quad (x,y)\mapsto(x,1-y),
\quad\text{and}\quad (x,y)\mapsto(y,x)
\end{align*}
all map $S$ onto itself and preserve area and chord length. The first two
change the signs of $\alpha$ and $\beta$, respectively, and modify $\tau$
accordingly, while the last interchanges $\alpha$ and $\beta$. Hence a
suitable composition transforms the normal into
$
(a,b):=
\left(\max\{|\alpha|,|\beta|\},\min\{|\alpha|,|\beta|\}\right).
$
Let $t$ denote the transformed value of $\tau$. We may then write
\begin{align*}
H=\{(x,y):ax+by<t\}\quad\text{with}\quad
0\leq b\leq a \quad\text{and}
\quad a^2+b^2=1.
\end{align*}
By central symmetry, the line
$ax+by=\frac{a+b}{2}$ bisects $S$.
Hence $0<\theta\leq\frac12$ implies
$
0<t\leq\frac{a+b}{2}.
$
If $b=0$, the section has length
one, so the assertion follows from $\theta(1-\theta)\leq\frac{1}{4}$. Assume
$b>0$.
If $0<t\leq b$, then $H\cap S$ is a right triangle,
$\theta=\frac{t^2}{2ab}$,
and
$\mathcal H^1(\partial H\cap S)=\frac{t}{ab}$. Thus
\begin{align*}
[\theta(1-\theta)]^{\frac{1}{2}}
\leq\sqrt\theta
=\frac{t}{\sqrt{2ab}}
\leq\frac{t}{2ab}
=\frac12\mathcal H^1(\partial H\cap S).
\end{align*}
If $b\leq t\leq\frac{a+b}{2}$, the cutting line joins the lower and upper
sides of $S$, and its length is $\frac{1}{a}\geq1$. Hence
\begin{align*}
[\theta(1-\theta)]^{\frac{1}{2}}
\leq\frac12
\leq\frac12\mathcal H^1(\partial H\cap S).
\end{align*}
This proves \eqref{eq-new-planar-cell-q}. Finally, a line through the
center parallel to a pair of opposite sides
cuts off an area of $\frac{1}{2}$ and has
section length one; Lemma \ref{lem-new-characteristic-formula}
establishes
equality for every $q$, proving sharpness.
This completes the proof of Lemma
\ref{lem-new-planar-cell}.
\end{proof}

The preceding pointwise estimate
\eqref{eq-new-planar-cell-q} is sharp at $q=2$ and fails for $q>2$.
To reach and slightly exceed the explicit
exponent $\frac{19}{9}$, we combine the chord length in
$S$ with its average over small translations $S+z$. The next lemma makes
this precise.
In what follows, for $q\in(0,\infty)$ and $\theta\in[0,1]$, write
\begin{align}\label{eq-new-Phi-q}
\Phi_q(\theta)
:=\left[\theta(1-\theta)^q+(1-\theta)\theta^q\right]^{\frac{1}{q}}.
\end{align}
Let
$\rho:=\frac{41}{100}$
and let $\nu_\rho$ be the
\emph{radial probability measure} on $B(0,\rho)$
given by
\begin{align}\label{eq-new-planar-averaging-density}
d\nu_\rho(z)
:=\frac{\mathbf 1_{B(0,\rho)}(z)}
{2\pi\rho\sqrt{\rho^2-|z|^2}}\,dz.
\end{align}

\begin{lemma}\label{lem-new-averaged-planar-cell}
There exists a universal constant
$q_*\in(\frac{19}{9},3)$ such that, for every unit square $S\subset\mathbb R^2$
with arbitrary orientation, every
half-plane $H$, and every $q\in[2,q_*]$,
\begin{align}\label{eq-new-averaged-planar-cell}
\Phi_q(|H\cap S|)
\leq
\frac{13}{56}\mathcal H^1(\partial H\cap S)
+\frac{15}{56}\int_{B(0,\rho)}
\mathcal H^1\left(\partial H\cap(S+z)\right)\,d\nu_\rho(z).
\end{align}
\end{lemma}

\begin{proof}
We first verify two properties of $\nu_\rho$. First, the formula in
\eqref{eq-new-planar-averaging-density} defines a probability measure.
Indeed, by polar coordinates,
\begin{align*}
\int_{\mathbb R^2}d\nu_\rho(z)
&=\int_{B(0,\rho)}
\frac{dz}{2\pi\rho\sqrt{\rho^2-|z|^2}}
=\int_0^{2\pi}\int_0^\rho
\frac{r}{2\pi\rho\sqrt{\rho^2-r^2}}\,dr\,d\vartheta\\
&=\frac1\rho\int_0^\rho
\frac{r}{\sqrt{\rho^2-r^2}}\,dr
=\frac1\rho
\left[-\sqrt{\rho^2-r^2}\right]_{r=0}^{r=\rho}
=1.
\end{align*}
For the second property, let $e$ be any unit vector and choose a unit
vector $e^\perp$ orthogonal to $e$. Writing $z=ue+ve^\perp$ and applying
Fubini's theorem, we obtain, for every integrable function
$g:[-\rho,\rho]\to\mathbb R$,
\begin{align*}
\int_{B(0,\rho)}g(e\cdot z)\,d\nu_\rho(z)
=\int_{-\rho}^{\rho}g(u)
\left[
\int_{-\sqrt{\rho^2-u^2}}^{\sqrt{\rho^2-u^2}}
\frac{dv}{2\pi\rho\sqrt{\rho^2-u^2-v^2}}
\right]du
=\frac1{2\rho}\int_{-\rho}^{\rho}g(u)\,du.
\end{align*}
The last identity is the property of $\nu_\rho$ used below to derive
\eqref{eq-new-uniform-projection-chord-average}.

We next reduce the geometric inequality to two explicit
parameterizations of normalized cuts.
The cases $|H\cap S|\in\{0,1\}$ are immediate. In other cases, the
symmetries used in the proof of Lemma \ref{lem-new-planar-cell}, together
with replacement of $H$ by the half-plane on the other side of
$\partial H$, reduce to
\begin{align*}
S=[0,1]^2,\
H=\{(x,y):ax+by<t\},\
0\leq b\leq a,\  a^2+b^2=1,\text{ and }
0\leq t\leq\frac{a+b}{2}.
\end{align*}
Put
\begin{align*}
p:=\frac ba\in[0,1],\quad
c:=\frac1a=\sqrt{1+p^2},\quad
\tau:=\frac ta,
\end{align*}
and, for $s\in\mathbb R$, define
\begin{align*}
F_p(s):=
\left|\{(x,y)\in[0,1]^2:x+py<s\}\right|.
\end{align*}
For every $p>0$, one has
\begin{align}\label{eq-new-Fp-profile}
F_p(s)=
\begin{cases}
0,&s\leq0,\\
\dfrac{s^2}{2p},&0\leq s\leq p,\\
s-\dfrac p2,&p\leq s\leq1,\\
1-\dfrac{(1+p-s)^2}{2p},&1\leq s\leq1+p,\\
1,&s\geq1+p.
\end{cases}
\end{align}
The limiting profile is $F_0(s)=\min\{1,\max\{0,s\}\}$. Thus
$F_p(s)$ is the area of the portion of the unit square lying below the
line $x+py=s$. Moreover, at every point where $F_p$ is differentiable,
\begin{align*}
\mathcal H^1\left(\{(x,y)\in[0,1]^2:x+py=s\}\right)
=cF_p'(s).
\end{align*}
By the definition of $F_p(s)$, we easily find that
$F_p(s)$ is nondecreasing in $s$ and non-increasing in $p$.

Let $\theta:=|H\cap S|=F_p(\tau)$ and
$\ell:=\mathcal H^1(\partial H\cap S)$. The uniform projection property
and the fundamental theorem of calculus imply
\begin{align}\label{eq-new-uniform-projection-chord-average}
A:=\int_{B(0,\rho)}
\mathcal H^1\left(\partial H\cap(S+z)\right)\,d\nu_\rho(z)
=\frac{F_p(\tau+\rho c)-F_p(\tau-\rho c)}{2\rho}.
\end{align}
We distinguish two geometric regimes.
If $b>0$ and $0\leq t\leq b$, then
$\partial H$ joins the adjacent sides $x=0$ and $y=0$, and $H\cap S$
is a possibly degenerate triangle. To obtain a parametrization
on $[0,1]^2$, let $r:=\frac{t}{b}\in[0,1]$. Then
\begin{align}\label{eq-new-quantitative-triangle-parameters}
\tau=pr,\quad
\theta=\frac{pr^2}{2},\quad
\ell=cr,\quad\text{and}\quad (p,r)\in[0,1]^2.
\end{align}
For this parameterization we include the edge $p=0$ by continuity. The
remaining case $b=t=0$ has $\theta=0$ and is immediate.
If $b\leq t\leq\frac{a+b}{2}$, then $\partial H$ joins the opposite sides
$y=0$ and $y=1$, and $H\cap S$ is a possibly degenerate trapezoid.
To obtain a parametrization on $[0,1]^2$, let
$r:=\frac{2(t-b)}{a-b}\in[0,1]$, with the endpoint $p=1$ adjoined by
continuity. Then
\begin{align}\label{eq-new-quantitative-trapezoid-parameters}
\tau=p+\frac{1-p}{2}r,\quad
\theta=\frac{p+(1-p)r}{2},\quad
\ell=c,\quad\text{and}\quad (p,r)\in[0,1]^2.
\end{align}
The interface $t=b$ is represented by $r=1$ in
\eqref{eq-new-quantitative-triangle-parameters} and by $r=0$ in
\eqref{eq-new-quantitative-trapezoid-parameters}; there the trapezoid
degenerates to the same triangle.
Thus the two copies of $[0,1]^2$ in
\eqref{eq-new-quantitative-triangle-parameters} and
\eqref{eq-new-quantitative-trapezoid-parameters} contain all normalized
cuts, including their limiting cases.

We next establish the elementary one-variable estimates needed below. First,
using the monotonicity of normalized $L^q$ means, we find that, for
$\theta\in[0,1]$ and $q\in[1,3]$,
\begin{align}\label{eq-new-Phi-below-half}
\Phi_q(\theta)
&\leq\Phi_3(\theta)
=\left[
\theta(1-\theta)^3+(1-\theta)\theta^3
\right]^{\frac13}\nonumber\\
&=\left[
\theta(1-\theta)\left((1-\theta)^2+\theta^2\right)
\right]^{\frac13}\nonumber\\
&=\left[
\frac18-2\left(\theta(1-\theta)-\frac14\right)^2
\right]^{\frac13}
\leq\frac12.
\end{align}
We now reduce the endpoint estimate to a scalar sufficient condition.
Since
$
\frac{19}{9}=2\cdot\frac59+\frac94\cdot\frac49,
$
the weighted H\"older inequality with conjugate exponents
$\frac95$ and $\frac94$ implies
\begin{align}\label{eq-new-Phi-log-convex-19-9}
\Phi_{\frac{19}{9}}(\theta)
&\leq\Phi_2(\theta)^{\frac{10}{19}}
\Phi_{\frac{9}{4}}(\theta)^{\frac{9}{19}}\notag\\
&=\left\{\left[\theta(1-\theta)\right]^9
\left[(1-\theta)^{\frac{5}{4}}
+\theta^{\frac{5}{4}}\right]^4
\right\}^{\frac{1}{19}}
=:G(\theta).
\end{align}
Here
\begin{align*}
\Phi_2(\theta)^2
=\theta(1-\theta)
\quad\text{and}\quad
\Phi_{\frac{9}{4}}(\theta)^{\frac{9}{4}}
=\theta(1-\theta)
\left[(1-\theta)^{\frac{5}{4}}
+\theta^{\frac{5}{4}}\right].
\end{align*}
Let
\begin{align*}
R:=\frac{13}{28}\ell+\frac{15}{28}A.
\end{align*}
Thus $\frac{R}{2}$ is exactly the right-hand side of
\eqref{eq-new-averaged-planar-cell}. In view of
\eqref{eq-new-Phi-log-convex-19-9}, its endpoint case follows from
$R\geq2G(\theta)$. Since both sides are nonnegative, this is equivalent
to
\begin{align}\label{eq-new-endpoint-sufficient-condition}
R^{19}\geq
2^{19}\left[\theta(1-\theta)\right]^9
\left[(1-\theta)^{\frac{5}{4}}
+\theta^{\frac{5}{4}}\right]^4.
\end{align}

It remains to prove \eqref{eq-new-endpoint-sufficient-condition}.
We first claim that $G$, the right-hand side of
\eqref{eq-new-Phi-log-convex-19-9}, is increasing on
$[0,\frac12]$. Indeed, $\Phi_2$ is increasing there by the first
identity above. For $\Phi_{\frac{9}{4}}$, let
$y:=[\frac{\theta}{1-\theta}]
^{\frac{1}{4}}\in[0,1]$. Then
\begin{align*}
\frac{d}{d\theta}\Phi_{\frac{9}{4}}(\theta)
^{\frac{9}{4}}
=(1-\theta)^{\frac{9}{4}}(1-y)
\left(\sum_{j=0}^{8}y^j-\frac94y^4\right)
\geq0,
\end{align*}
where the inequality follows from $\sum_{j=0}^{8}y^j\geq9y^4$.

We now prove \eqref{eq-new-endpoint-sufficient-condition} uniformly
on the two parameter squares
in \eqref{eq-new-quantitative-triangle-parameters} and
\eqref{eq-new-quantitative-trapezoid-parameters}. Since these squares
are closed, the boundary cases $p=0$ and $p=1$ are included, with
$F_0$ interpreted as above. We cover each parameter square by finitely
many dyadic rectangles. A computer is used only to carry out this finite
subdivision and check the explicit rational inequalities stated below.

To bound the square and fourth roots by rational numbers, fix $N:=90$.
For a nonnegative rational number $x$ and an integer $d\geq2$, define
\begin{align*}
k_d(x):=\max\left\{k\in\mathbb{Z}_+:
k^d\leq 2^{Nd}x\right\}
\end{align*}
and let
\begin{align*}
L_d(x):=2^{-N}k_d(x)
\quad\text{and}\quad
U_d(x):=
\begin{cases}
L_d(x),&L_d(x)^d=x,\\
L_d(x)+2^{-N},&L_d(x)^d<x.
\end{cases}
\end{align*}
By construction,
$
L_d(x)\leq x^{\frac{1}{d}}
\leq U_d(x),
$
and both bounds are dyadic rational numbers. Consequently, after
clearing denominators, every comparison involving these bounds is an
inequality between integers.

For each of the two parameterizations in
\eqref{eq-new-quantitative-triangle-parameters} and
\eqref{eq-new-quantitative-trapezoid-parameters}, let
\begin{align*}
\mathcal R=[p_-,p_+]\times[r_-,r_+]\subset[0,1]^2
\end{align*}
be a dyadic rectangle in the corresponding $(p,r)$-parameter square,
and put
\begin{align*}
c_-:=L_2(1+p_-^2)\quad\text{and}\quad c_+:=U_2(1+p_+^2).
\end{align*}
In the \emph{triangular regime} define
\begin{align*}
\tau_-:=p_-r_-,\quad
\tau_+:=p_+r_+,\quad
\theta_+:=\frac{p_+r_+^2}{2},\quad\text{and}\quad
\ell_-:=c_-r_-,
\end{align*}
and in the \emph{opposite-side regime} define
\begin{align*}
\tau_-:=p_-+\frac{1-p_-}{2}r_-,\quad
\tau_+:=p_++\frac{1-p_+}{2}r_+,
\end{align*}
\begin{align*}
\theta_+:=\frac{p_++(1-p_+)r_+}{2},\quad\text{and}\quad
&\ell_-:=c_-.
\end{align*}
The monotonicity of $F_p$ implies the following lower bound for the averaged
chord in \eqref{eq-new-uniform-projection-chord-average} throughout
$\mathcal R$:
\begin{align*}
A_-:=\max\left\{0,
\frac{F_{p_+}(\tau_-+\rho c_-)
-F_{p_-}(\tau_+-\rho c_-)}{2\rho}\right\}.
\end{align*}
Let
\begin{align*}
R_-:=\frac{13}{28}\ell_-+\frac{15}{28}A_-.
\end{align*}
Then $R\geq R_-$ throughout $\mathcal R$.
In both parameterizations, $0\leq\theta\leq\theta_+\leq\frac12$
throughout $\mathcal R$.
Define the rational upper bounds
\begin{align*}
\widehat S_+
:=(1-\theta_+)U_4(1-\theta_+)
+\theta_+U_4(\theta_+)
\quad\text{and}\quad
\widehat W_+
:=\left[\theta_+(1-\theta_+)\right]^9
\widehat S_+^4.
\end{align*}
Since $G$, the right-hand side of
\eqref{eq-new-Phi-log-convex-19-9}, is increasing on
$[0,\frac12]$, and since
$x^{\frac{1}{4}}\leq U_4(x)$, we obtain, throughout
$\mathcal R$,
\begin{align}\label{eq-new-box-weight-upper-bound}
G(\theta)^{19}
&\leq G(\theta_+)^{19}\notag
=\left[\theta_+(1-\theta_+)\right]^9
\left[(1-\theta_+)^{\frac{5}{4}}
+\theta_+^{\frac{5}{4}}\right]^4\notag\\
&\leq\left[\theta_+(1-\theta_+)\right]^9
\left[(1-\theta_+)U_4(1-\theta_+)
+\theta_+U_4(\theta_+)\right]^4\notag\\
&=\left[\theta_+(1-\theta_+)\right]^9\widehat S_+^4
=\widehat W_+.
\end{align}

We now claim that if $\mathcal R$ satisfies
\begin{align}\label{eq-new-exact-box-test}
R_-^{19}
>2^{19}\widehat W_+,
\end{align}
then \eqref{eq-new-averaged-planar-cell} with
$q=\frac{19}{9}$ holds for every
normalized cut whose parameters $(p,r)$ lie in $\mathcal R$.
Indeed, by their definitions, $R_-\geq0$ and $\widehat W_+\geq0$. Taking
nineteenth roots and using
\eqref{eq-new-Phi-log-convex-19-9},
\eqref{eq-new-box-weight-upper-bound}, and $R\geq R_-$, we obtain,
for every $(p,r)\in\mathcal R$,
\begin{align*}
\Phi_{\frac{19}{9}}(\theta)
&\leq G(\theta)
\leq\widehat W_+^{\frac{1}{19}}
<\frac{R_-}{2}
\leq\frac R2
=\frac{13}{56}\ell+\frac{15}{56}A.
\end{align*}
This proves the claim. Moreover, every quantity in
\eqref{eq-new-exact-box-test} is rational; hence, after clearing
denominators, \eqref{eq-new-exact-box-test} is an inequality between
integers.

It remains to show that every point of the two parameter squares lies
in a rectangle to which either \eqref{eq-new-exact-box-test} or a direct
argument applies, and then to extend the endpoint estimate
to the full stated range. We do this in three steps.

We first identify the rectangles for which the desired inequality can
be proved without \eqref{eq-new-exact-box-test}. Let $\mathcal R$ be a
dyadic rectangle in the opposite-side parameter square and suppose that
\begin{align}\label{eq-new-plateau-box-test}
\tau_--\rho c_+\geq p_+
\quad\text{and}\quad
\tau_++\rho c_+\leq1.
\end{align}
Since $p\leq p_+$, $\tau_-\leq\tau\leq\tau_+$, and $c\leq c_+$
for every $(p,r)\in\mathcal R$, we have
\begin{align*}
p\leq p_+
\leq\tau_--\rho c_+
\leq\tau-\rho c
\leq\tau+\rho c
\leq\tau_++\rho c_+
\leq1.
\end{align*}
Thus $[\tau-\rho c,\tau+\rho c]\subset[p,1]$, on which
\eqref{eq-new-Fp-profile} implies
$F_p(s)=s-\frac{p}{2}$. By
\eqref{eq-new-uniform-projection-chord-average}
and \eqref{eq-new-quantitative-trapezoid-parameters},
we find that
\begin{align*}
A=\frac{F_p(\tau+\rho c)-F_p(\tau-\rho c)}{2\rho}
=\frac{(\tau+\rho c-\frac{p}{2})
-(\tau-\rho c-\frac{p}{2})}{2\rho}
=c=\ell.
\end{align*}
Since $c=\sqrt{1+p^2}\geq1$, \eqref{eq-new-Phi-below-half} now implies
\begin{align*}
\Phi_{\frac{19}{9}}(\theta)
\leq\frac12
\leq\frac c2
=\frac{13}{56}\ell+
\frac{15}{56}A.
\end{align*}
Hence the desired inequality holds for every normalized cut whose
parameters lie in a rectangle satisfying
\eqref{eq-new-plateau-box-test}.

Second, we construct finite dyadic partitions of the two parameter squares.
In the triangular parameter square, a rectangle is retained once it
satisfies \eqref{eq-new-exact-box-test}; otherwise its longer side is
bisected. In the opposite-side parameter square, a rectangle is first
retained if it satisfies \eqref{eq-new-plateau-box-test}; if not, it is
retained if it satisfies \eqref{eq-new-exact-box-test}; only when
neither condition holds is its longer side bisected. When the side
lengths are equal, the $p$-interval is bisected.

The accompanying exact-arithmetic verification carries out this
subdivision using the rational bounds defined above. It leaves no
unresolved rectangle and produces the following finite dyadic
partitions, where the initial parameter square is regarded as generation
zero:
\begin{align*}
\begin{array}{c@{\qquad}r@{\qquad}r@{\qquad}r@{\qquad}r}
\text{regime}
&\eqref{eq-new-exact-box-test}
&\eqref{eq-new-plateau-box-test}
&\text{total}
&\text{maximal generation}\\ \hline
\text{triangle}&97\,061&0&97\,061&26\\
\text{opposite-side}&5\,281&64&5\,345&21.
\end{array}
\end{align*}
The second and third columns record the numbers of retained rectangles
treated by \eqref{eq-new-exact-box-test} and
\eqref{eq-new-plateau-box-test}, respectively, and the last column
records the maximal generation reached. Thus this verification
constructs a finite collection of retained rectangles. Since the two
children produced at any bisection have union equal to their parent,
these rectangles cover both parameter squares. Moreover, the rational
bounds ensure that the estimates above hold uniformly on each retained
rectangle, not merely at finitely many selected parameter values. The
claim proved
before Step 1 applies to the rectangles in the second column, and the
direct argument in Step 1 applies to those in the third column. Hence the
desired inequality holds for every admissible parameter pair. By
\eqref{eq-new-quantitative-triangle-parameters} and
\eqref{eq-new-quantitative-trapezoid-parameters}, it therefore covers
every normalized cut.

All comparisons in this finite construction are inequalities between
rational numbers and hence, after clearing denominators, between
integers. The complete subdivision and the corresponding comparisons
are carried out by the exact-arithmetic program reproduced in
Appendix \ref{app-new-planar-cell-certificate}.
We have now proved \eqref{eq-new-averaged-planar-cell} for
$q=\frac{19}{9}$.

We next pass from
$q=\frac{19}{9}$ to
$2\leq q\leq\frac{19}{9}$.
For fixed $\theta$, $\Phi_q(\theta)$ is a normalized $L^q$ norm and
therefore increases with $q$. Hence, for every
$q\in[2,\frac{19}{9}]$,
\begin{align*}
\Phi_q(\theta)
\leq\Phi_{\frac{19}{9}}(\theta)
\leq\frac{13}{56}\ell+\frac{15}{56}A.
\end{align*}

It remains to extend the estimate slightly beyond $\frac{19}{9}$.
Let $\mathscr R_0$ be the finite family of retained rectangles
accepted by \eqref{eq-new-exact-box-test}, and set
\begin{align*}
\gamma
:=\min_{\mathcal R\in\mathscr R_0}
\left\{\frac{R_-(\mathcal R)}{2}
-\widehat W_+(\mathcal R)^{\frac{1}{19}}\right\}>0,
\end{align*}
where the arguments indicate the rectangle from which the bounds are
computed. The estimates above then give
$
\Phi_{\frac{19}{9}}(\theta)\leq\frac{R}{2}-\gamma
$
for every cut whose parameters lie in one of these rectangles. On each
plateau rectangle, the preceding computation gives $A=\ell=c\geq1$;
hence \eqref{eq-new-Phi-below-half} proves
\eqref{eq-new-averaged-planar-cell} for every $q\in[1,3]$.

The map $(q,\theta)\mapsto\Phi_q(\theta)$ is uniformly continuous on
$[\frac{19}{9},3]\times[0,1]$. Therefore, there exists
$q_*\in(\frac{19}{9},3)$ such that
\begin{align*}
\sup_{\theta\in[0,1]}
\left|\Phi_q(\theta)-\Phi_{\frac{19}{9}}(\theta)\right|<\gamma
\end{align*}
for every $q\in[\frac{19}{9},q_*]$. The desired estimate now follows on the
rectangles in $\mathscr R_0$ from the strict margin and on the plateau rectangles
from the preceding direct bound. Together with the monotonicity argument,
this proves the stated range $q\in[2,q_*]$.

This completes the proof of Lemma \ref{lem-new-averaged-planar-cell}.
\end{proof}

\begin{remark}\label{rem-new-planar-constants-not-optimized}
The proof gives a universal exponent $q_*>\frac{19}{9}$, but it
does not extract a useful numerical value from the smallest margin in the
finite subdivision. Thus $\frac{19}{9}$ remains the explicit exponent certified
directly by the rational inequalities.
The averaging radius and the two coefficients in Lemma
\ref{lem-new-averaged-planar-cell} were chosen for convenience.
Optimizing these parameters may enlarge the exponent range obtained from
this particular averaged cell inequality. This observation concerns only
the present proof and does not determine the optimal planar exponent.
\end{remark}

The pointwise planar bound
$\Phi_q(\theta)\leq\frac{\ell}{2}$ used in Step 1 of
Proposition \ref{prop-new-localization} holds for $1\leq q\leq2$.
Figure \ref{fig-new-planar-cell-degeneracies} illustrates why it fails for
$q>2$ and how Lemma \ref{lem-new-averaged-planar-cell} repairs it by
translation averaging; Lemma \ref{lem-new-aggregate-planar-localization}
below then transfers and sums the averaged estimate.
We next globalize the averaged cell estimate.
Indeed, for a $45^\circ$ vertex cut with shaded leg length $r$,
$\theta=\frac{r^2}{2}$, $\ell=\sqrt2r$, and
$\frac{\Phi_q(\theta)}{\ell}
\sim_q r^{\frac{2}{q}-1}\to\infty$
as $r\to0^+$ for $q>2$.
Thus $\ell$ cannot control $\Phi_q(\theta)$ pointwise, whereas
$A:=\int_{B(0,\rho)}\ell(z)\,d\nu_\rho(z)$, with
$\ell(z):=\mathcal H^1(\partial H\cap(S+z))$, satisfies
$A\to\frac{\rho}{2}>0$. The central cut in panel (a) still has
$\theta=\frac{1}{2}$ and $\ell=A=1$, preserving equality.

\begin{figure}[htbp]
\centering
\begin{tikzpicture}[scale=1.05,line cap=round,line join=round,
every node/.style={font=\small}]
\fill[gray!25] (0,0.6) rectangle (2,1.6);
\draw[blue!65!black,dashed] (0.25,0.85) rectangle (2.25,2.85);
\draw[red!65,dashed,very thick] (0.25,1.6)--(2.25,1.6);
\draw[thick] (0,0.6) rectangle (2,2.6);
\draw[red,very thick] (0,1.6)--(2,1.6);
\node[anchor=east] at (-0.03,2.45) {$S$};
\node[blue!65!black,anchor=west] at (2.28,2.73) {$S+z$};
\node at (1,-0.02) {(a) central: $\theta=\frac12$, $\ell=A=1$};

\fill[blue!10] (4.42,0.22)--(5.6,0.22)--(4.42,1.40)--cycle;
\draw[blue!65!black,dashed] (4.42,0.22) rectangle (6.42,2.22);
\draw[red!65,dashed,very thick] (4.42,1.40)--(5.6,0.22);
\fill[gray!25] (4.8,0.6)--(5.22,0.6)--(4.8,1.02)--cycle;
\draw[thick] (4.8,0.6) rectangle (6.8,2.6);
\draw[red,very thick] (4.8,1.02)--(5.22,0.6);
\draw[->,blue!65!black] (4.78,0.58)--(4.44,0.24);
\node[blue!65!black] at (4.60,0.33) {$z$};
\node[red,anchor=east] at (4.38,1.42) {$\partial H$};
\node at (5.16,1.08) {$\ell$};
\node[red!65] at (5.02,1.54) {$\ell(z)>\ell$};
\node[blue!65!black,anchor=west] at (4.42,2.36) {$S+z$};
\node[anchor=west] at (6.82,2.45) {$S$};
\node at (3.4,3.25) {$A=\displaystyle\int_{B(0,\rho)}\ell(z)\,d\nu_\rho(z)$};
\node at (5.6,-0.02) {(b) vertex: $\ell\to0$,
$A\to\frac{\rho}{2}$};
\end{tikzpicture}
\caption{Translation averaging repairs the vertex-cut degeneration}
\label{fig-new-planar-cell-degeneracies}
\end{figure}
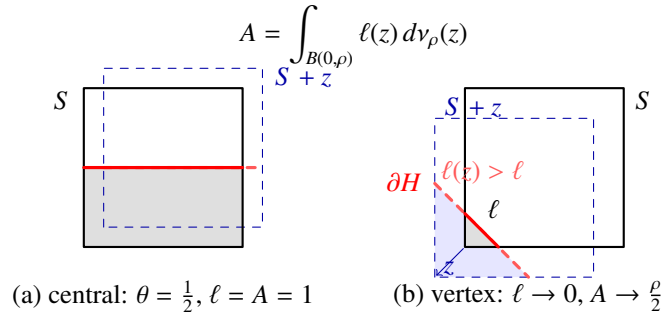

The common translation enters the summation below separately. For each fixed
$z$, every square $Q_j$ is shifted by the same vector $\varepsilon z$; hence
the interiors $(Q_j+\varepsilon z)^\circ$ remain pairwise disjoint and
\begin{align*}
\sum_j\sigma_E((Q_j+\varepsilon z)^\circ)
\leq\sigma_E(\mathbb R^2).
\end{align*}
Only after this estimate do we integrate with respect to $z$.

We now state the resulting aggregate localization
lemma.

\begin{lemma}
\label{lem-new-aggregate-planar-localization}
{Let $q_*$ be as in Lemma
\ref{lem-new-averaged-planar-cell}. If $q\in[2,q_*]$ and}
$E\subset Q_0=(0,1)^2$ is measurable with
$P(E,Q_0)<\infty$, then
\begin{align}\label{eq-new-aggregate-planar-localization}
\limsup_{\varepsilon\to0^+}I_{\varepsilon,q}^{*}(E)
\leq\frac12P(E,Q_0).
\end{align}
\end{lemma}

\begin{proof}
Let $\rho=\frac{41}{100}$ and let $\nu_\rho$ be the measure in
\eqref{eq-new-planar-averaging-density}.
Recall that the relative variation measure $\sigma_E$, defined in
\eqref{eq-new-relative-variation-extension}, is extended by zero outside
$Q_0$ and satisfies
$\sigma_E(\mathbb R^2)=P(E,Q_0)$.

Fix $\kappa,\zeta\in(0,1)$. Lemma
\ref{lem-new-uniform-blowup}, applied with $\kappa$, provides a compact set
$G\subset\partial^*E\cap Q_0$ such that
\begin{align}\label{eq-new-aggregate-good-set}
\sigma_E(Q_0\setminus G)<\kappa,
\end{align}
the normal is continuous on $G$, and the set and perimeter blow-ups are
uniform there.

The proof now has two parts. We first transfer Lemma
\ref{lem-new-averaged-planar-cell} to any sufficiently small square
meeting $G$. We then sum the resulting estimate over an admissible
family and control the squares that avoid $G$ by their total perimeter.

We first prove that there exists $r_0\in(0,1)$ such that any oriented
square $Q\subset Q_0$ of edge length
$\varepsilon\in(0,r_0)$ with $Q^\circ\cap G\ne\emptyset$ satisfies
\begin{align}\label{eq-new-uniform-averaged-local-estimate}
2\varepsilon\operatorname{osc}_q(\mathbf{1}_E;Q)
\leq\frac{13}{28}\sigma_E(Q^\circ)
+\frac{15}{28}\int_{B(0,\rho)}
\sigma_E\left((Q+\varepsilon z)^\circ\right)\,d\nu_\rho(z)
+\zeta\varepsilon.
\end{align}
Suppose otherwise. {Then there exist
$\{\varepsilon_m\}_{m\in\mathbb N}\subset(0,1)$ with
$\varepsilon_m\to0^+$ as $m\to\infty$, points
$\{x_m\}_{m\in\mathbb N}\subset G$, and oriented squares
$\{Q_m\}_{m\in\mathbb N}$ such that, for every $m\in\mathbb N$,
$Q_m\subset Q_0$ has edge length $\varepsilon_m$,
$x_m\in Q_m^\circ\cap G$, and
\eqref{eq-new-uniform-averaged-local-estimate} fails with
$Q=Q_m$ and $\varepsilon=\varepsilon_m$.}
{Let}
$
S_m:=\frac{Q_m-x_m}{\varepsilon_m}.
$
Repeating the compactness and blow-up argument in Step 1 of the proof of
Proposition \ref{prop-new-localization}, after passing to a subsequence we
may assume that $x_m\to x\in G$ and $S_m\to S$ in Hausdorff distance
as $m\to\infty$, where $S$ is a unit square. Let
\begin{align*}
\theta_m:=\frac{|E\cap Q_m|}{|Q_m|}
=|E_{x_m,\varepsilon_m}\cap S_m|
\quad\text{and}\quad
\theta:=|H_x\cap S|.
\end{align*}
Then the moving-cube set convergence used there {implies}
$\theta_m\to\theta$ as $m\to\infty$.
Consequently, Lemma \ref{lem-new-characteristic-formula} {implies}
\begin{align}\label{eq-new-good-cube-oscillation-limit}
\operatorname{osc}_q(\mathbf{1}_E;Q_m)\to\Phi_q(\theta)
\end{align}
as $m\to\infty$.
The strict reverse of
\eqref{eq-new-uniform-averaged-local-estimate}, together with the
nonnegativity of the two perimeter terms, implies
$\operatorname{osc}_q(\mathbf{1}_E;Q_m)>\frac{\zeta}{2}$. Passing to the
limit as $m\to\infty$ and using
\eqref{eq-new-good-cube-oscillation-limit}, we obtain
$\Phi_q(\theta)\geq\frac{\zeta}{2}$, and therefore $\theta\in(0,1)$. Thus
$\partial H_x$ cuts the interior of $S$. The same geometric observation
used in Step 1 of the proof of Proposition \ref{prop-new-localization}
then {implies}
\begin{align}\label{eq-new-aggregate-boundary-null}
\mathcal H^1(\partial H_x\cap\partial S)=0.
\end{align}
Let $\mu_m:=\mu_{x_m,\varepsilon_m}$. The moving-cube measure convergence
used there, together with
\eqref{eq-new-aggregate-boundary-null} {implies}
\begin{align}\label{eq-new-untranslated-perimeter-limit}
\varepsilon_m^{-1}\sigma_E(Q_m^\circ)
=\mu_m(S_m^\circ)
\to\mu_x(S)
=\mathcal H^1(\partial H_x\cap S)
\end{align}
as $m\to\infty$.
For the translated terms, observe that
$\mu_x(\partial(S+z))=0$ for $\nu_\rho$-almost every $z$. Indeed, the
exceptional translations are those for which $\partial H_x$ contains a
face of $S+z$; they lie in a finite union of affine lines and hence have
zero $\nu_\rho$-measure because $\nu_\rho$ is absolutely continuous with
respect to Lebesgue measure. {For every} other $z$,
\eqref{eq-new-moving-cube-measure}, applied to
$S_m+z\to S+z$ as $m\to\infty$, implies
\begin{align*}
\mu_m((S_m+z)^\circ)
\to
\mathcal H^1\left(\partial H_x\cap(S+z)\right)
\end{align*}
as $m\to\infty$.
All the squares $S_m+z$, with $z\in B(0,\rho)$, lie in a fixed compact
set $K$, and the local weak-$*$ convergence
$\mu_m\stackrel{*}{\rightharpoonup}\mu_x$ as $m\to\infty$, established in
the proof of Lemma \ref{lem-new-moving-cube-blowup}, implies
$\sup_{m\in\mathbb{N}}\mu_m(K)<\infty$. Thus, by the dominated convergence theorem,
we obtain
\begin{align}\label{eq-new-averaged-perimeter-limit}
\varepsilon_m^{-1}\int_{B(0,\rho)}
\sigma_E\left((Q_m+\varepsilon_m z)^\circ\right)\,d\nu_\rho(z)
&=\int_{B(0,\rho)}\mu_m((S_m+z)^\circ)\,d\nu_\rho(z)\notag\\
&\to
\int_{B(0,\rho)}
\mathcal H^1\left(\partial H_x\cap(S+z)\right)\,d\nu_\rho(z)
\end{align}
as $m\to\infty$. Divide the strict reverse of
\eqref{eq-new-uniform-averaged-local-estimate} by $\varepsilon_m$ and pass
to the limit as $m\to\infty$ using
\eqref{eq-new-good-cube-oscillation-limit},
\eqref{eq-new-untranslated-perimeter-limit}, and
\eqref{eq-new-averaged-perimeter-limit}. We obtain
\begin{align*}
2\Phi_q(|H_x\cap S|)
\geq\frac{13}{28}\mathcal H^1(\partial H_x\cap S)
+\frac{15}{28}\int_{B(0,\rho)}
\mathcal H^1\left(\partial H_x\cap(S+z)\right)\,d\nu_\rho(z)
+\zeta,
\end{align*}
contrary to Lemma \ref{lem-new-averaged-planar-cell}. This proves
\eqref{eq-new-uniform-averaged-local-estimate}.

Fix an admissible family
$\{Q_j\}_{j\in J}\in\mathscr F_{\varepsilon,1}^{*}(Q_0)$ with
$\varepsilon\in(0,r_0)$, and split its index set into
\begin{align*}
J_1:=\{j:Q_j^\circ\cap G\ne\emptyset\}
\quad\text{and}\quad
J_2:=\{j:Q_j^\circ\cap G=\emptyset\}.
\end{align*}
{For every} fixed $z$, the common translation preserves pairwise
disjointness of the square interiors. Since $\sigma_E$ is zero outside
$Q_0$, summing \eqref{eq-new-uniform-averaged-local-estimate} over
$J_1$, using $\#J_1\leq\varepsilon^{-1}$ and then
$\nu_\rho(B(0,\rho))=1$ and
$\frac{13}{56}+\frac{15}{56}=\frac{1}{2}$, we obtain
\begin{align}\label{eq-new-aggregate-good-squares}
\varepsilon\sum_{j\in J_1}\operatorname{osc}_q(\mathbf{1}_E;Q_j)
&\leq\frac{13}{56}\sum_{j\in J_1}\sigma_E(Q_j^\circ)
+\frac{15}{56}\int_{B(0,\rho)}\sum_{j\in J_1}
\sigma_E((Q_j+\varepsilon z)^\circ)\,d\nu_\rho(z)
+\frac{\zeta}{2}\varepsilon\#J_1\notag\\
&\leq\left(\frac{13}{56}+\frac{15}{56}\right)
\sigma_E(\mathbb R^2)+\frac{\zeta}{2}
=\frac{1}{2}P(E,Q_0)+\frac{\zeta}{2}.
\end{align}

It remains to estimate $J_2$. For $j\in J_2$, let
\begin{align*}
\theta_j:=\frac{|E\cap Q_j|}{|Q_j|}
\quad\text{and}\quad
\quad m_j:=\min\{\theta_j,1-\theta_j\}.
\end{align*}
If $J_2=\emptyset$, the estimate below is immediate. Otherwise, since
$\frac{2}{q}\in(0,1]$, {using}
Lemma \ref{lem-new-characteristic-formula}, H\"older's
inequality, the relative isoperimetric inequality on each square, the bound
$\#J_2\leq\varepsilon^{-1}$, the pairwise disjointness of the square
interiors, and \eqref{eq-new-aggregate-good-set}, we obtain
\begin{align}\label{eq-new-aggregate-bad-squares}
\varepsilon\sum_{j\in J_2}\operatorname{osc}_q(\mathbf{1}_E;Q_j)
&\leq C_q\varepsilon\sum_{j\in J_2}m_j^{\frac{1}{q}}
=C_q\varepsilon\sum_{j\in J_2}
(m_j^{\frac{1}{2}})^{\frac{2}{q}}\notag\\
&\leq C_q\varepsilon(\#J_2)^{1-\frac{2}{q}}
\left(\sum_{j\in J_2}m_j^{\frac{1}{2}}\right)^{\frac{2}{q}}\notag\\
&\leq C_q\varepsilon(\varepsilon^{-1})^{1-\frac{2}{q}}
\left(C\varepsilon^{-1}
\sum_{j\in J_2}\sigma_E(Q_j^\circ)\right)^{\frac{2}{q}}\notag\\
&\leq C_q\varepsilon(\varepsilon^{-1})^{1-\frac{2}{q}}
(C\varepsilon^{-1}\kappa)^{\frac{2}{q}}
=C_q\kappa^{\frac{2}{q}}.
\end{align}
Here the last inequality uses
$\bigcup_{j\in J_2}Q_j^\circ\subset Q_0\setminus G$.

The estimates \eqref{eq-new-aggregate-good-squares} and
\eqref{eq-new-aggregate-bad-squares} are uniform over the admissible
family. Taking the supremum, we obtain
\begin{align*}
I_{\varepsilon,q}^{*}(E)
\leq\frac{1}{2}P(E,Q_0)+\frac{\zeta}{2}
+C_q\kappa^{\frac{2}{q}}
\end{align*}
for all sufficiently small $\varepsilon$. We first take the upper limit
as $\varepsilon\to0^+$ and then let $\zeta,\kappa\to0^+$. We obtain
\eqref{eq-new-aggregate-planar-localization}. This completes the proof of
Lemma \ref{lem-new-aggregate-planar-localization}.
\end{proof}

\subsection{Proof of the Exact Perimeter Formulas}

We now prove the exact formulas in Theorem \ref{the-new-exact}.

\begin{proof}[Proof of Theorem \ref{the-new-exact}]
For $q=1$, the assertion is Lemma \ref{lem-new-known-q1} with $M=1$.
If $q\in(0,1)$, H\"older's inequality on each cube {implies}
$I_{\varepsilon,q}^{*}(E)\leq I_{\varepsilon,1}^{*}(E)$. Lemma
\ref{lem-new-known-q1} supplies the upper limit and Proposition
\ref{prop-new-universal-lower} the matching lower limit, proving
\eqref{eq-new-exact-all-dim} in this range.

{Let $q_*$ be as in Lemma
\ref{lem-new-averaged-planar-cell}. Now let} $q>1$ belong to one of the ranges in the
statement. If
$q\leq \frac{n}{n-1}$ and $P(E,Q_0)<\infty$, Proposition
\ref{prop-new-localization} {implies}
\begin{align}\label{eq-new-exact-perimeter-upper-general}
\limsup_{\varepsilon\to0^+}I_{\varepsilon,q}^{*}(E)
\leq\frac12P(E,Q_0).
\end{align}
If $n=2$ and {$q\in(2,q_*]$}, the same estimate follows instead
from Lemma \ref{lem-new-aggregate-planar-localization}. Thus
\eqref{eq-new-exact-perimeter-upper-general} holds throughout the stated
range whenever the perimeter is finite.

Moreover, every exponent under consideration is at most $3$. Lemma
\ref{lem-new-characteristic-formula} and
\eqref{eq-new-Phi-below-half} therefore
{imply}, {for every} cube $Q$,
\begin{align*}
\left[\fint_Q|\mathbf1_E(x)-(\mathbf1_E)_Q|^q\,dx\right]^{\frac{1}{q}}
=\Phi_q\left(\frac{|E\cap Q|}{|Q|}\right)\leq\frac12,
\end{align*}
where $\Phi_q$ is the same as in \eqref{eq-new-Phi-q}.
Since $\#J\leq\varepsilon^{1-n}$, it follows directly from
\eqref{eq-new-I-unit} that
\begin{align}\label{eq-new-exact-truncation-upper}
I_{\varepsilon,q}^{*}(E)\leq\frac12
\end{align}
{for every} measurable $E\subset Q_0$. If $P(E,Q_0)<\infty$, combining
\eqref{eq-new-exact-perimeter-upper-general} and
\eqref{eq-new-exact-truncation-upper} {implies}
\begin{align*}
\limsup_{\varepsilon\to0^+}I_{\varepsilon,q}^{*}(E)
\leq\frac12\min\{1,P(E,Q_0)\}.
\end{align*}
If $P(E,Q_0)=\infty$, the same conclusion follows from
\eqref{eq-new-exact-truncation-upper}. Proposition
\ref{prop-new-universal-lower} supplies the reverse inequality in every
case {and completes the proof of Theorem
\ref{the-new-exact}}.
\end{proof}

\section{A Counterexample and Bounds
for the Optimal Exponent\label{sec-new-open-problems}}

Theorem \ref{the-new-exact} shows that, in dimension two, the proved
validity range of the exact formula
extends strictly beyond the natural
exponent $2$. We next identify a dimension-independent range of exponents
for which the exact perimeter formula fails in every dimension $n\geq2$.

{For every} $t\in(0,1]$ and $q\in[2,\infty)$, {let}
\begin{align*}
C_t(q):=\frac{4}{\sqrt{1+t^2}}\Phi_q\left(\frac{t}{8}\right).
\end{align*}
The map $q\mapsto C_t(q)$ is strictly increasing.
Moreover, {for every} $t\in(0,1]$,
\begin{align}\label{eq-new-packing-end-values}
C_t(2)^2
=\frac{8t-t^2}{4(1+t^2)}<1
\quad\text{and}\quad
\lim_{q\to\infty}C_t(q)
=\frac{8-t}{2\sqrt{1+t^2}}>1.
\end{align}
Thus {there exists} a unique $r(t)>2$ such that
$C_t(r(t))=1$. Define
\begin{align}\label{eq-new-q-packing}
q_0:=\inf_{0<t\leq1}r(t).
\end{align}
If we {let} $C_0(q):=0$,
then $(t,q)\mapsto C_t(q)$ extends continuously to
$[0,1]\times[2,3]$. The first inequality in
\eqref{eq-new-packing-end-values} and compactness
{imply}
$\max_{t\in[0,1]}C_t(2)<1$.
The same strict inequality therefore
holds {for every} $q$
in a fixed right neighborhood of $2$, and hence
$q_0>2$.

{We now turn the packing profile into a counterexample.}

\begin{proposition}
\label{prop-new-packing-obstruction}
Let $n\in{\mathbb{N}}\cap[2,\infty)$
and $q\in(q_0,\infty)$.
The fixed set
\begin{align}\label{eq-new-fixed-packing-set}
E:=\{x\in Q_0:x_1+x_2<\tfrac14\}
\end{align}
satisfies
\begin{align}\label{eq-new-packing-failure}
\liminf_{\varepsilon\to0^+}I_{\varepsilon,q}^{*}(E)
>\frac12P(E,Q_0)
=\frac12\min\{1,P(E,Q_0)\}.
\end{align}
Consequently, the exact formula fails in {all dimensions}
{for every}
$q>q_0$.
\end{proposition}

\begin{proof}
We first describe the planar core of the construction. The full
$n$-dimensional family is obtained below by taking products with intervals
of length $\varepsilon$ in the remaining coordinates.
Let
\begin{align*}
\Gamma:=\{(x_1,x_2)\in\mathbb R^2:x_1+x_2=\tfrac14\}.
\end{align*}
In the $(x_1,x_2)$-plane, introduce the orthonormal affine coordinates
$u:=\frac{x_1-x_2}{\sqrt2}$ and
$v:=\frac{x_1+x_2-\frac14}{\sqrt2}$.
Thus $\Gamma=\{v=0\}$, the $u$-direction is tangent to $\Gamma$, and
$\{x_1+x_2<\frac14\}=\{v<0\}$.

Choose $t\in(0,1]$ such that $C_t(q)>1$, which is possible by
\eqref{eq-new-q-packing}. Put
\begin{align*}
e_1:=\frac{(1,t)}{\sqrt{1+t^2}}
\quad\text{and}\quad
e_2:=\frac{(-t,1)}{\sqrt{1+t^2}}.
\end{align*}
In the $(u,v)$-coordinate plane, let
\begin{align*}
K_t:=\{\xi e_1+\zeta e_2:|\xi|\leq\tfrac12,
\ |\zeta|\leq\tfrac12\}
\end{align*}
be the unit square centered at the origin with edge directions $e_1$ and
$e_2$. For each $k\in\mathbb Z$, let
\begin{align*}
z_k^+
:=\left(k\sqrt{1+t^2},\frac{1}{2\sqrt{1+t^2}}\right)
\quad\text{and}\quad
z_k^-:=\left(k\sqrt{1+t^2}+\frac{t}{\sqrt{1+t^2}},
-\frac{1}{2\sqrt{1+t^2}}\right).
\end{align*}
{Let} $Q_k^\pm:=z_k^\pm+K_t$. The resulting two rows of squares are
illustrated in Figure \ref{fig-new-two-row-packing}.

\begin{figure}[htbp]
\centering
\begin{tikzpicture}[scale=1.15,line cap=round,line join=round]
\foreach \x in {0,1.2806,2.5612,3.8418}
{
\begin{scope}[shift={(\x,0.3904)},rotate=38.66]
\fill[gray!25] (-0.5,-0.5)--(-0.5,-0.1)--(0,-0.5)--cycle;
\draw[thick] (-0.5,-0.5) rectangle (0.5,0.5);
\end{scope}
}

\foreach \x in {0.6247,1.9053,3.1859,4.4665}
{
\begin{scope}[shift={(\x,-0.3904)},rotate=38.66]
\fill[gray!25] (-0.5,-0.5) rectangle (0.5,0.5);
\fill[white] (0,0.5)--(0.5,0.5)--(0.5,0.1)--cycle;
\draw[thick] (-0.5,-0.5) rectangle (0.5,0.5);
\end{scope}
}

\draw[red,very thick] (-0.85,0)--(5.25,0);
\node[red] at (5.38,0.22) {$\Gamma$};
\node[anchor=south] at (1.92,1.18) {$\theta=\frac{t}{8}$};
\node[anchor=north] at (2.55,-1.18) {$\theta=1-\frac{t}{8}$};
\end{tikzpicture}
\caption{The planar two-row configuration with $t=\frac45$.
{The shaded regions are
$Q_k^\pm\cap\{(u,v)\in\mathbb R^2:\ v<0\}$,
$k\in\mathbb Z$.}}
\label{fig-new-two-row-packing}
\end{figure}
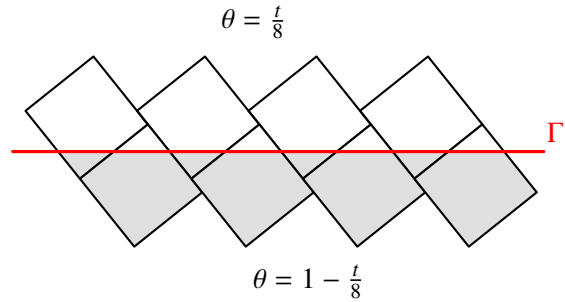

The squares $Q_k^+$ and $Q_k^-$,
$k\in\mathbb Z$, form the upper and lower rows, respectively, and have
pairwise disjoint interiors. Indeed, the displacement
between two centers in the same row has $(e_1,e_2)$-coordinates
$(m,-mt)$ with $m\in\mathbb Z\setminus\{0\}$,
whose first coordinate has absolute value at least one. The displacement
$z_j^--z_k^+$, with $m:=j-k$, has $(e_1,e_2)$-coordinates
$(m,-mt-1)$.
If $m\ne0$, its first coordinate has absolute value at least one, while
if $m=0$, its second coordinate equals $-1$. This proves the claim.

The interface line $\Gamma=\{v=0\}$
{cuts off a volume fraction $\frac{t}{8}$ from every
$Q_k^+$ and the complementary fraction from every $Q_k^-$}. To see
the first assertion, write a point of $Q_k^+$ as
$z_k^++\xi e_1+\zeta e_2$ and translate
$(\xi,\zeta)\in[-\frac{1}{2},\frac{1}{2}]^2$ to $[0,1]^2$. The inequality $v<0$ cuts off
a right triangle with legs $\frac{1}{2}$ and $\frac{t}{2}$, hence area $\frac{t}{8}$. The second
assertion follows by symmetry. In both rows the local oscillation is
therefore $\Phi_q(\frac{t}{8})$.

The relative reduced boundary of the set in
\eqref{eq-new-fixed-packing-set} is
$
\bigl(\Gamma\cap(0,1)^2\bigr)\times(0,1)^{n-2}.
$
The planar segment $\Gamma\cap(0,1)^2$ has length
$L:=\frac{\sqrt2}{4}$
and hence $P(E,Q_0)=L<1$. Scale the two-row family
$\{Q_k^\pm\}_{k\in\mathbb Z}$ by $\varepsilon$ in the $(u,v)$-plane.
Retain only the planar squares
whose closures lie in $(0,1)^2$.
In the remaining $n-2$ directions, take the products with the usual
interior grid of intervals of length $\varepsilon$. The disjointness just
proved shows that these products form a family of oriented
$n$-dimensional cubes with pairwise disjoint interiors.
Indeed, the planar factor is a square of edge length
$\varepsilon$, every remaining factor is an interval of the same length,
and the planar rotation is extended by the identity in the last
$n-2$ coordinate directions.

In each planar row, consecutive centers are separated by
$\sqrt{1+t^2}$ in the $u$-direction. After scaling by $\varepsilon$,
this separation becomes $\varepsilon\sqrt{1+t^2}$. Hence, after
discarding {a bounded} number of squares near the two endpoints, the two rows contain
$\frac{2L}{\sqrt{1+t^2}}\varepsilon^{-1}+O(1)$ squares. In each of the remaining directions,
the interior grid contains $\varepsilon^{-1}+O(1)$ intervals. Hence the
resulting family has cardinality
\begin{align}\label{eq-new-packing-count}
N_\varepsilon
=\left(\frac{2L}{\sqrt{1+t^2}}\varepsilon^{-1}+O(1)\right)
\left(\varepsilon^{-1}+O(1)\right)^{n-2}
=\left(\frac{2L}{\sqrt{1+t^2}}+O(\varepsilon)\right)
\varepsilon^{1-n}
\end{align}
as $\varepsilon\to0^+$.
All constants implicit in the $O(1)$ terms are independent of
$\varepsilon$ (with $t$ and $n$ fixed). They account only for the bounded
number of planar squares discarded near the two endpoints of the segment
and for the boundary intervals omitted by the interior grids in the
remaining coordinate directions.
Since
$
\frac{2L}{\sqrt{1+t^2}}
=\frac1{\sqrt{2(1+t^2)}}<1,
$
this family satisfies the cardinality budget
$N_\varepsilon\leq\varepsilon^{1-n}$ {for every} sufficiently small
$\varepsilon$. Every retained cube has volume fraction $\frac{t}{8}$ or
$1-\frac{t}{8}$. Lemma \ref{lem-new-characteristic-formula},
\eqref{eq-new-packing-count}, and the definition of $C_t(q)$ now {imply}
\begin{align*}
\liminf_{\varepsilon\to0^+}I_{\varepsilon,q}^{*}(E)
\geq\frac{2L}{\sqrt{1+t^2}}\Phi_q(\frac{t}{8})
=\frac12C_t(q)L
>\frac12L
=\frac12P(E,Q_0).
\end{align*}
This proves \eqref{eq-new-packing-failure}.
{This completes the proof of Proposition
\ref{prop-new-packing-obstruction}.}
\end{proof}

\begin{remark}\label{rem-new-packing-numerics}
A numerical estimate can be reproduced as follows. For each fixed
$t\in(0,1]$, solve the monotone equation $C_t(q)=1$ for $q=r(t)$ by
bisection, and then minimize $r(t)$ by a one-dimensional search.
Equivalently, if $F(t,q):=\log C_t(q)$, an interior critical point
satisfies
\begin{align*}
F(t,q)=0\quad\text{and}\quad \partial_tF(t,q)=0.
\end{align*}
This computation finds a single critical point and suggests that it is
the unique global minimizer, with
\begin{align*}
t_*\approx0.80978475\quad\text{and}\quad q_0\approx2.1539098432.
\end{align*}
These decimal approximations are included only for orientation and are
not used in Proposition \ref{prop-new-packing-obstruction}. For a
completely explicit upper bound, take $t=\frac{4}{5}$ and {let}
$\widehat q:=r(\frac45)$. Then $\frac{t}{8}=\frac{1}{10}$, and
$\widehat q>2$ is the unique root of
\begin{align*}
9^{\widehat q}+9
=10\left(\frac{\sqrt{41}}2\right)^{\widehat q}.
\end{align*}
The strict upper bound $\widehat q<2.1541$ admits a
fully rational certificate. Indeed, put
\begin{align*}
q_1:=\frac{21541}{10000}=2+s,\quad
s:=\frac{1541}{10000},\quad
L:=\log 3,\quad\text{and}\quad M:=\frac12\log\frac{41}{36}.
\end{align*}
For $q\geq2$, let
$
D(q):=9^q+9-10\left(\frac{\sqrt{41}}2\right)^q.
$
For $z>1$ and $u=\frac{z-1}{z+1}$, the identity
\begin{align*}
\log z=2\sum_{k=0}^{N}\frac{u^{2k+1}}{2k+1}+R_N
\quad\text{and}\quad
0<R_N<\frac{2u^{2N+3}}{(2N+3)(1-u^2)},
\end{align*}
applied with $(z,N)=(3,7)$ and $(z,N)=(\frac{41}{36},1)$,
respectively, gives
\begin{align*}
1.09861<L<1.09862
\quad\text{and}\quad
0.06502<M<0.06503.
\end{align*}
Hence $x:=s(L-M)>\frac{15927}{100000}$ and
$y:=s(L+M)<\frac{17932}{100000}$. Let
$P(t):=1+t+\frac{t^2}{2}+\frac{t^3}{6}$ and
$Q(t):=1-t+\frac{t^2}{2}-\frac{t^3}{6}$. Taylor's theorem gives
$e^x>P(x)$ and $e^{-y}>Q(y)$; moreover, $P$ is increasing and $Q$ is
decreasing on the intervals under consideration. Since
$\log(\frac{\sqrt{41}}2)=L+M$ and
$\log(\frac9{\sqrt{41}/2})=L-M$, we obtain
\begin{align*}
\frac{D(q_1)}
{(\frac{\sqrt{41}}2)^s}
&=81e^x+9e^{-y}-\frac{205}{2}\notag\\
&>81P\left(\frac{15927}{100000}\right)
+9Q\left(\frac{17932}{100000}\right)-\frac{205}{2}
=\frac{29667702707511}{6000000000000000}>0.
\end{align*}
On the other hand, $D(2)=-\frac{25}{2}$. The uniqueness of the root
therefore proves
$q_0\leq\widehat q<2.1541$
(numerically, $\widehat q\approx2.1540092537$). Thus the explicit choice
$t=\frac45$ alone proves failure for every $q>\widehat q$.
\end{remark}

{For each $n\in\mathbb{N}\cap[2,\infty)$, define the
\emph{optimal exponent}}
\begin{align*}
q(n)
:=\sup\left\{p>0:
\lim_{\varepsilon\to0^+}I_{\varepsilon,p}^{*}(E)
=\frac12\min\{1,P(E,Q_0)\}
\ \text{{for every} measurable }E\subset Q_0\right\}.
\end{align*}
The validity range is downward closed: if the formula holds for an
exponent $q$, then monotonicity of local $L^q$ means {implies} the upper bound
{for every} smaller positive exponent, while Proposition
\ref{prop-new-universal-lower} supplies the reverse bound. Therefore,
Theorem \ref{the-new-exact} and Proposition
\ref{prop-new-packing-obstruction} {establish}
$
\frac{n}{n-1}
\leq q(n)
\leq q_0,
$
and the planar result improves the lower bound to
$
{\frac{19}{9}<q(2)}
\leq q_0.
$
In particular, $2$ is not the optimal planar exponent, while the fixed
half-space section \eqref{eq-new-fixed-packing-set} rules out every
exponent above approximately $2.15391$.

{At $q=q_0$, one has
$C_t(q_0)\leq1$ for any $t\in(0,1]$.
Thus Proposition \ref{prop-new-packing-obstruction} rules out any
$q>q_0$, but does not establish failure at the endpoint.
Consequently, the exact value of $q(2)$ remains undetermined in
$(\frac{19}{9},q_0]$.
For $n\geq3$, the corresponding unresolved interval is
$(\frac{n}{n-1},q_0]$.
Even after $q(n)$ is determined, one must still decide
whether the exact formula holds at $q=q(n)$.}

\begin{openproblem}\label{prob-new-optimal-exact-range}
{What is the value of $q(n)$ for $n\geq2$,
and does the exact formula hold at $q=q(n)$?}
\end{openproblem}

\appendix

\section{Exact-Arithmetic Verification Program for
Lemma \ref{lem-new-averaged-planar-cell}}
\label{app-new-planar-cell-certificate}

The program below is the exact-arithmetic verification program used in
the finite subdivision in Step 2 of the proof of Lemma
\ref{lem-new-averaged-planar-cell}. An identical standalone copy is supplied
with the manuscript as the supplementary file
\texttt{planar\_cell\_q19\_9.py}. The program uses only the Python 3
standard library. Every acceptance decision is made using exact rational
arithmetic and outward-rounded dyadic root bounds; floating-point values
play no role. For every ordinary retained rectangle, the acceptance
inequality is strict.

Running the program with Python 3 produces the following output:
\begin{lstlisting}[basicstyle=\ttfamily\small,numbers=none]
triangle: accepted=97061, face=0, max_depth=26, failures=0
trapezoid: accepted=5345, face=64, max_depth=21, failures=0
CERTIFIED q=19/9: True
\end{lstlisting}
Here \texttt{trapezoid} denotes the opposite-side regime in the proof.
The output above was reproduced with Python 3.12.14 by running
\texttt{python planar\_cell\_q19\_9.py}. For completeness and direct
inspection, the identical source code is retained below.

\begingroup
\raggedbottom
\begin{lstlisting}[style=certificate,
literate={19>(2}{{{19 > (2}}}7]
"""Rigorous certificate for the quantitative planar cell inequality.

This standard-library-only program proves the cell estimate at q=19/9
for rho=41/100 and eta=15/28, using the radial probability measure

    d nu_rho(z) = 1_{|z|<rho} dz /
                   (2*pi*rho*sqrt(rho^2-|z|^2)).

Every one-dimensional projection of nu_rho is uniform on [-rho,rho].
After the normalization explained in the paper, the two geometric
regimes are parameterized by (p,r) in [0,1]^2.  Log-convexity between
L^2 and L^(9/4) reduces the endpoint to an algebraic sufficient
condition.  The program covers both parameter squares by finitely many
dyadic boxes and verifies that condition with exact rational interval
arithmetic.  Square and fourth roots are enclosed by outward-rounded
dyadic rationals; no floating-point result is used in an acceptance
decision.  Boxes on which every shifted cut stays in the central
plateau are handled directly.
"""

from collections import deque
from fractions import Fraction as F


RHO = F(41, 100)
ETA = F(15, 28)
ROOT_BITS = 90
MAX_DEPTH = 32


def integer_nth_root_floor(n, degree):
    """Return floor(n**(1/degree)) exactly for nonnegative integers."""
    if n < 2:
        return n
    x = 1 << ((n.bit_length() + degree - 1) // degree)
    while True:
        y = ((degree - 1) * x + n // (x ** (degree - 1))) // degree
        if y >= x:
            break
        x = y
    while (x + 1) ** degree <= n:
        x += 1
    while x ** degree > n:
        x -= 1
    return x


def root_bounds(x, degree):
    """Outward dyadic bounds for x**(1/degree), with x rational >= 0."""
    assert x >= 0
    scale = 1 << ROOT_BITS
    floor_arg = (x.numerator * (scale ** degree)) // x.denominator
    k = integer_nth_root_floor(floor_arg, degree)
    lo = F(k, scale)
    if lo ** degree == x:
        return lo, lo
    return lo, F(k + 1, scale)


def cdf(p, y):
    """F_p(y)=P(X+pY<y), where X,Y are uniform on [0,1]."""
    if y <= 0:
        return F(0)
    if y >= 1 + p:
        return F(1)
    if p == 0:
        return y
    if y <= p:
        return y * y / (2 * p)
    if y <= 1:
        return y - p / 2
    return 1 - (1 + p - y) ** 2 / (2 * p)


def cdf_interval(p0, p1, y0, y1):
    """Use that F_p(y) increases in y and decreases in p."""
    return cdf(p1, y0), cdf(p0, y1)


def interpolation_bound_ok(rhs_lo, theta_hi):
    """Certify the L^2--L^(9/4) upper bound at q=19/9.

    Log-convexity gives

      Phi_(19/9) <= Phi_2^(10/19) Phi_(9/4)^(9/19).

    If u=theta(1-theta) and
    s=(1-theta)^(5/4)+theta^(5/4), the nineteenth power of
    twice the right-hand side is 2^19 u^9 s^4.  This upper
    bound is increasing for theta in [0,1/2].  The fourth roots
    below are rounded upward at theta_hi.
    """
    if rhs_lo < 0:
        return False
    assert F(0) <= theta_hi <= F(1, 2)
    one_minus = 1 - theta_hi
    u_hi = theta_hi * one_minus
    s_hi = (
        one_minus * root_bounds(one_minus, 4)[1]
        + theta_hi * root_bounds(theta_hi, 4)[1]
    )
    algebraic_hi = (u_hi ** 9) * (s_hi ** 4)
    return rhs_lo ** 19>(2 ** 19) * algebraic_hi


def certify_box(kind, p0, p1, r0, r1):
    """Try to certify one dyadic (p,r)-box."""
    c_lo = root_bounds(1 + p0 * p0, 2)[0]
    c_hi = root_bounds(1 + p1 * p1, 2)[1]

    if kind == "triangle":
        # p=b/a, c=1/a, r=t/b:
        # y=t/a=pr, theta=pr^2/2, ell=cr.
        y_lo, y_hi = p0 * r0, p1 * r1
        theta_hi = p1 * r1 * r1 / 2
        chord_lo = c_lo * r0
    else:
        # r parameterizes b<=t<=(a+b)/2:
        # y=t/a=p+(1-p)r/2,
        # theta=[p+(1-p)r]/2, ell=c.
        y_lo = p0 + (1 - p0) * r0 / 2
        y_hi = p1 + (1 - p1) * r1 / 2
        theta_hi = (p1 + (1 - p1) * r1) / 2
        chord_lo = c_lo

        # Every translated level stays on the plateau.  Here the
        # averaged and untranslated chords both equal c>=1, whereas
        # 2 Phi_q(theta)<=1 for q<=3.
        if y_lo - RHO * c_hi >= p1 and y_hi + RHO * c_hi <= 1:
            return True, "face"

    y_minus_lo = y_lo - RHO * c_hi
    y_minus_hi = y_hi - RHO * c_lo
    y_plus_lo = y_lo + RHO * c_lo
    y_plus_hi = y_hi + RHO * c_hi

    cdf_plus_lo, _ = cdf_interval(p0, p1, y_plus_lo, y_plus_hi)
    _, cdf_minus_hi = cdf_interval(p0, p1, y_minus_lo, y_minus_hi)
    average_lo = (cdf_plus_lo - cdf_minus_hi) / (2 * RHO)
    average_lo = max(F(0), average_lo)

    rhs_lo = (1 - ETA) * chord_lo + ETA * average_lo
    return (
        interpolation_bound_ok(rhs_lo, theta_hi),
        "ordinary",
    )


def certify_regime(kind):
    """Adaptively cover one copy of [0,1]^2 by certified boxes."""
    pending = deque([(F(0), F(1), F(0), F(1), 0)])
    accepted = face_boxes = deepest = 0
    failures = []

    while pending:
        p0, p1, r0, r1, depth = pending.pop()
        ok, box_type = certify_box(kind, p0, p1, r0, r1)
        if ok:
            accepted += 1
            face_boxes += box_type == "face"
            deepest = max(deepest, depth)
            continue
        if depth >= MAX_DEPTH:
            failures.append((p0, p1, r0, r1))
            continue

        if p1 - p0 >= r1 - r0:
            midpoint = (p0 + p1) / 2
            pending.append((midpoint, p1, r0, r1, depth + 1))
            pending.append((p0, midpoint, r0, r1, depth + 1))
        else:
            midpoint = (r0 + r1) / 2
            pending.append((p0, p1, midpoint, r1, depth + 1))
            pending.append((p0, p1, r0, midpoint, depth + 1))

    print(
        f"{kind}: accepted={accepted}, face={face_boxes}, "
        f"max_depth={deepest}, failures={len(failures)}"
    )
    return not failures


if __name__ == "__main__":
    certified = certify_regime("triangle") and certify_regime("trapezoid")
    print(f"CERTIFIED q=19/9: {certified}")
    if not certified:
        raise SystemExit(1)
\end{lstlisting}
\endgroup

\bigskip

\noindent
Hongchao Jia and Xianjie Yan (Corresponding author)

\smallskip

\noindent  Institute of Contemporary Mathematics,
School of Mathematics and Statistics, Henan University,
Kaifeng 475004, The People's Republic of China

\smallskip

\noindent {\it E-mails}: \texttt{hongchaojia@henu.edu.cn} (H. Jia)

\noindent\phantom{{\it E-mails:}} \texttt{xianjieyan@henu.edu.cn} (X. Yan)

\end{document}